\documentclass[twoside=off]{scrbook}

\usepackage{customtex}
\input{commands.tex}
\usepackage[style=alphabetic, backend=bibtex]{biblatex}

\begin{document}
	
	
\frontmatter
$\,$ 

\vspace{3cm}
\begin{center}
\huge\textbf{Chow-Witt Rings of Classifying Spaces of Products of Multiplicative and Cyclic groups} \normalsize

\vspace{2em}
\Large Andrea Lachmann \normalsize

\vspace{2em}
30 Apr 2025

\end{center}
\vspace{2cm}
\textbf{Abstract.}
We compute the total Chow-Witt rings of the classifying space $B\mu_n$ of the roots of unity, as well as the products $B\Gm \times B\mu_n$ and $B\mu_m \times B\mu_n$ for all $m,n \geq 1$ based on the strategy by di Lorenzo and Mantovani (2023) for $B\mu_n$ with $n$ even.
Moreover we compute the total $I^j$-cohomology and Chow-Witt rings of $\projective^q \times \projective^r$ for all $q,r \geq 1$ and of $B\Gm \times B\Gm$.
\pagebreak

\renewcommand{\theeverything}{\Alph{everything}}

\chapter{Introduction}

Chow-Witt groups have gained increasing attention in intersection theory and enumerative geometry over the past 20 years.
Classically, a central concept of intersection theory are Chow groups of a scheme containing information about the subschemes in a given dimension.
Another invariant of schemes is the Witt ring of quadratic forms over a scheme introduced by \cite[Chapter I §5]{knebusch}.
Chow-Witt groups were then conceived by \cite{barge-morel} via a fiber product of Chow and Witt theory, and implemented in technical detail by \cite{fasel-thesis}.

One object of interest in this area are characteristic classes of vector bundles like Chern classes, Euler classes and Stiefel-Whitney classes living in the different theories mentioned above.
In topology the Euler class assigns to a real vector bundle an element in the singular cohomology of the base space with a local coefficient system (integral coefficients if the vector bundle is orientable).
One of its properties is being an obstruction for the vector bundle to split off a trivial summand, meaning if such a splitting exists then the Euler class vanishes.
Now for vector bundles over a scheme, one can of course consider characteristic classes with values in Chow groups, but it turns out that these satisfy the classical obstruction property only for $\Complex$-schemes.
For schemes over general fields and in particular over $\Real$ a characteristic class with this obstruction property lives in the Chow-Witt groups (compare \cref{thm:euler-class-splitting}).

Motivic classifying spaces of algebraic groups were introduced by \cite{morel-voevodsky}, imitating the long-familiar concept of classifying spaces in topology.
The \'{e}tale and geometric classifying space classify $G$-bundles in the sense that according to \cite[§4 Proposition 1.15]{morel-voevodsky} in the Nisnevich local motivic homotopy category of schemes, the set of maps from some $X$ into the classifying space of $G$ is in bijection to the set of isomorphism classes of $G$-torsors over $X$.
Around the same time Totaro independently defined the Chow ring of a classifying space by means of scheme approximations \cite{totaro98}, without defining the classifying space itself which usually does not exist in the category of schemes. 
This turned out to agree with the definition of Morel-Voevodsky and provides an important tool for computing Chow rings which was later extended to Chow-Witt groups by \cite[Theorem 3.3]{asok-fasel16}.

Some examples for computations of Chow-Witt rings include projective space $\projective^n$ by \cite{fasel-projbundle}, which also provides a powerful tool for $I^j$-cohomology of projective bundles.
The ring structure for $\projective^n$ and $B\Gm$ was computed by \cite{wendt} as a special case of a Grassmannian.
Later \cite{diLorenzo-Mantovani} computed the Chow-Witt ring of $B\mu_n$ for even $n$.
We will heavily build on these two articles in this work.
Further examples can be found in \cite{hornbostel-wendt} and \cite{HXZ}.

In this work we aim to compute several new examples for Chow-Witt rings of classifying spaces of groups and products thereof.
We will extend the argument of \cite{diLorenzo-Mantovani}, considering $B\mu_n$ as a $\Gm$-bundle over $B\Gm$ and then applying a localization sequence to deduce the Chow-Witt groups of the former from those of the latter, to $n$ odd (\cref{chowwitt-ring-of-Bmun-odd}):
\begin{theorem}
	Let $n$ be an odd natural number and $k$ a perfect field of characteristic coprime to $2$ and $n$.
	Denote by $\chowwitt^\tot(X)$ the total Chow-Witt ring of a smooth scheme $X$.
	Then there is an isomorphism of graded $\GW(k)$-algebras
	\[ \chowwitt^\tot (B\mu_n) \cong \GW(k)[\eulercwz]/(I(k) \cdot \eulercwz, n \cdot \eulercwz) \]
	where $\eulercwz \in \chowwitt^1(B\mu_n, \trivialbundle(1))$ denotes the Euler class of $\trivialbundle_{B\mu_n}(-1)$.
\end{theorem}

Then we compute the $I^j$-cohomology of $\projective^m \times \projective^n$ (\cref{thm:Ij-of-PxP}) using the projective bundle formula of \cite{fasel-projbundle}, 
then combine this with the Chow groups of the same space (\cite[Theorem 2.10, 2.12]{totaro-bluebook}) via a fiber product statement of \cite[Lemma 2.11]{hornbostel-wendt}, see \cref{prop:chowwitt-of-PxP} and \cref{prop:chowwitt-of-BGmxBGm}.
\begin{theorem}
	For $k$ a perfect field of characteristic coprime to $2$, there is an isomorphism of graded $\GW(k)$-algebras
	\begin{multline*} 
		\chowwitt^\tot(B\Gm \times B\Gm) \cong 
		\GW(k) [\eulercwa, \eulercwb, \eulercwc, \hyperbolica, \hyperbolicb, \hyperbolicc]/(I(k) \cdot (\hyperbolica, \hyperbolicb, \hyperbolicc, \eulercwa, \eulercwb, \eulercwc), \\
		\eulercwa^2 + \eulercwb^2 + \hyperbolicc \eulercwa \eulercwb - \eulercwc^2, \mathcal{J})  
	\end{multline*}
	where $\eulercwa$, $\eulercwb$ and $\eulercwc$ correspond to the Euler classes of $\trivialbundle(1,0)$, $\trivialbundle(0,1)$, $\trivialbundle(1,1)$, respectively, the $\hyperbolica$, $\hyperbolicb$, $\hyperbolicc$ are a kind of analogue of the hyperbolic form $\hyperbolic$ living in the twisted degree $0$ groups, and $\mathcal{J}$ is a relation ideal. These are described in more detail in \cref{prop:chowwitt-of-PxP}.
\end{theorem}

Later we will use a very similar strategy like for a single copy of $B\mu_n$ to deduce the Chow-Witt rings of $B\Gm \times B\mu_n$ (\cref{thm:chowwitt-of-BGmxBmun}) and $B\mu_m \times B\mu_n$ (\cref{thm:chowwitt-of-BmumxBmun}) from that of $B\Gm \times B\Gm$.
For some values of $m$ and $n$ this argument is inconclusive for degree $0$ which is why in these cases we employ a comparison with Witt cohomology, which in turn can be computed via a Künneth formula developed by \cite[Theorem 4.7]{HMW}.
Most of the ring structure can be derived from that of $B\Gm \times B\Gm$ by precisely understanding the isomorphisms identifying those Chow-Witt groups whose twists become isomorphic over $B\Gm \times B\mu_n$ respectively $B\mu_m \times B\mu_n$.
These rings depend on the parity of $m$ and $n$.
Their algebra presentations are so lengthy in total that we decided against including them in the introduction and the reader should refer to the respective theorems for details.

It will turn out that some but not all of the products considered - namely $\projective^m \times \projective^n$, $B\Gm \times B\Gm$, $B\Gm \times B\mu_n$, $B\mu_m \times B\mu_n$ - satisfy a Künneth isomorphism for Chow-Witt rings, see \cref{remark:kuenneth-formula}.
This provides examples to the already known fact that such an isomorphism cannot be true for Chow-Witt rings of general products.
It remains an open question to establish a nice set of conditions under which this holds.
Such conditions are known for both Chow rings \cite[Theorem 2.12]{totaro-bluebook} and Witt cohomology \cite[Proposition 4.7]{HMW}, but not for $I^j$-cohomology or Chow-Witt rings.

\section*{Notation and Conventions}
Throughout all of this paper, let $k$ be a \textbf{perfect field of characteristic not $2$}.
Perfectness is relevant for some motivic homotopy arguments, for example avoiding smoothness issues since every regular finite type scheme over a perfect field is already smooth.
There are methods to circumvent this condition in many applications developed by \cite{elmanto-khan}, and it might be possible that it can be entirely removed, see also \cite[Section 2.3]{HMW}.
Characteristic different from $2$ is a common assumption in the theory of quadratic forms, because quadratic forms agreeing with symmetric bilinear forms simplifies matters.
Recent works, e.g. \cite{feld}, accomplish at least the construction of Chow-Witt groups as well as functoriality and long exact localization sequences in arbitrary characteristic. 
However there remain arguments, such as the decomposition of Chow-Witt groups into Chow groups and $I^j$-cohomology (\cref{sec:fiber-products}) and the projective bundle theorem for $I^j$-cohomology (\cref{thm:proj-bundle-Ij-thm}), that are only known to be true in characteristic unequal $2$.

All \textbf{rings} are \textbf{associative and unital}.
All \textbf{schemes} are considered to be \textbf{separated of finite type over $k$}.
This implies that the image of a closed subscheme under a proper morphism of schemes is again closed, which is necessary to construct a pushforward map on Chow-Witt groups. 
When working with cohomological Chow-Witt groups or any other kind of cohomology, that is, from \cref{sec:properties} onward, all schemes are required to be smooth.
For a scheme $X$ and an integer $i$ denote by $X_{(i)}$ the set of points of dimension $i$ in $X$, i.e.\ points whose closure has dimension $i$.
If $X$ is smooth and thus has a fixed dimension, denote by $X^{(i)}$ the set of points of codimension $i$.

For a commutative ring $A$ and symbols $x_1, x_2, \ldots$, we denote by $A \langle x_1, x_2, \ldots \rangle$ the free $A$-module generated by $x_1, x_2, \ldots$, not to be confused with the polynomial ring $A[x_1, x_2, \ldots]$.

\noindent We use the following standard notation. \vspace{1em}

\begin{tabular}{r l}
$\GW(k)$ & Grothendieck-Witt ring of quadratic forms over $k$ \\
$\witt(k)$ & Witt ring of quadratic forms modulo hyperbolic spaces \\
$I(k)$ & the fundamental ideal of $\GW(k)$ or $\witt(k)$ \\
$\chow_i(X)$, $\chow^i(X)$ & $i$-th (co)homological Chow group of a scheme $X$ \\
$\chowmodtwo_i(X)$, $\chowmodtwo^i(X)$ & $i$-th (co)homological Chow group modulo $2$ \\
$\chowwitt_i(X, \linebundle)$, $\chowwitt^i(X, \linebundle)$ & $i$-th (co)homological Chow-Witt group twisted by a line \\
& bundle $\linebundle$ over $X$ \\
$\KMW_*(k)$ & the graded Milnor-Witt K-theory ring \\
$\KM_*(k)$ & the graded Milnor K-theory ring \\
$I^*(k) = \KW_*(k)$ & the graded ring consisting of powers of $I(k)$, where \\
& $I^{\leq 0} (k) \coloneqq \witt(k)$
\end{tabular} \vspace{1em}\\
All of these will be introduced in detail in \cref{sec:Chow-Witt-rings}.

\section*{Acknowledgments}
During the writing of this thesis, the author was supported by the GRK 2240 \enquote{Algebro-geometric methods in algebra, arithmetic and topology} of the DFG.
I would like to thank my advisors Jens Hornbostel and Marcus Zibrowius for all the support and patience they offered me.
Also I am grateful to Jule Hänel, Burt Totaro, Matthias Wendt and Alex Ziegler for many helpful discussions, and to Leon Hendrian, Yordan Toshev and Alex Ziegler for their comments on earlier versions of this thesis.

\renewcommand{\theeverything}{\thechapter.\arabic{everything}}

\newgeometry{top=2.5cm}
\tableofcontents
\restoregeometry

\mainmatter
\chapter{Chow-Witt Rings} \label{sec:Chow-Witt-rings}

This chapter introduces Chow-Witt groups and rings and some of their fundamental properties.
In some places we will only sketch ideas but do not include proofs, and the reader should be aware that some results cited from the literature are highly non-trivial.

\section{Motivation and Relation to Chow Groups}

The Chow-Witt ring is a quadratic refinement of the Chow ring, so to provide some context we will start with a quick recollection on the latter.
A more detailed account can be found e.g.\ in \cite[Chapter 1]{fulton}. 

Let $X$ be a scheme over $k$.
The group of algebraic $i$-cycles of $X$ is the free abelian group on closed subschemes of $X$ of dimension $i$, and the $i$-th Chow group $\chow_i(X)$ is then obtained by dividing out rational equivalence.
This can be expressed by the exact sequence
\[ \bigoplus_{W \in X_{(i+1)}} \kappa(W)^{\times} \xrightarrow{\div} \bigoplus_{V \in X_{(i)}} \Integer\langle V \rangle \to \chow_i(X) \to 0 \, . \]

The idea of Chow-Witt groups $\chowwitt_i(X)$ is to consider algebraic cycles $\sum n_V [V]$ not with integral coefficients $n_V$, but instead coefficients in the Grothendieck-Witt ring $\GW(k)$, the group completion of quadratic forms over $k$ up to isometry, equipped with direct sum and tensor product.
This is motivated by the fact that the Grothendieck-Witt ring captures the arithmetic of the base field in more detail. 
Through this one hopes to generalize certain results that hold only over the base field $\Complex$ in classical intersection theory (using Chow groups), considering that $\GW(\Complex) \cong \Integer$.
For example, the zeroth motivic stable stem is isomorphic to $\GW(k)$ by \cite[6.4.1]{morel-introduction} whereas its topological analogue is $\pi_0^s(\mathbb{S}) \cong \Integer$.
A further example is the following splitting theorem of Morel, matching the obstruction property of topological Euler classes.

\begin{theorem}[{\cite[Thm 8.14]{morel-lecturenotes}}]\label{thm:euler-class-splitting}
	Assume $r \geq 4$. 
	Let $X$ be a smooth affine k-scheme of dimension $\leq r$, and let $\xi$ be an oriented algebraic vector bundle of rank $r$. 
	Then:
	\[ \text{$\xi$ splits off a trivial line bundle} \Leftrightarrow \eulercw(\xi) = 0 \in \chowwitt^r(X) \]
\end{theorem}

The analogous statement for Chern classes under the assumption that $k$ is algebraically closed was proved by \cite[Theorem 3.8]{murthy}.
This obstruction property for vector bundles to split off a trivial summand is also true for Euler classes of topological spaces which take values in singular cohomology with integral coefficients.
In fact, there is also a direct connection between $I^j$-cohomology, which is a quotient of Chow-Witt groups, and singular cohomology:
For a smooth scheme $X$ over a field $k \supseteq \Real$ and a line bundle $\linebundle$ over $X$, Jacobson \cite{jacobson} defines a real cycle class map 
\[ H^i(X, I^j, \linebundle) \to H^i_{\sing}(X(\Real); \Integer(\linebundle)) \]
where $X(\Real)$ denotes the set of real points of $X$ with the analytic topology.
This is an isomorphism if $j > \dim X$ \cite[Corollary 8.3]{jacobson}, or if $j \geq i$ and $X$ is cellular \cite[Theorem 5.7]{HWXZ}.

In order to do construct such a quadratic refinement it will be helpful to rephrase the above definition of Chow groups in terms of Milnor K-theory.
\begin{definition}
	Let $F$ be a finitely generated field extension over $k$.
	Its Milnor K-theory $\KM_*(F)$ is the graded-commutative ring generated by the symbols $[a]$ for all $a \in F^\times$ in degree $1$, modulo the relations $[a][1-a]$ for all $a \neq 0, \, 1$ and $[a] + [b] = [ab]$ for all $a, b \in F^\times$.
	The $i$-th degree of this ring is the $i$-th Milnor K-theory group of $F$, denoted $\KM_i(F)$.
\end{definition}
Observing that we have $F^\times = \KM_1(F)$, $\Integer = \KM_0(F)$ and $0 = \KM_{-1}(F)$, the exact sequence defining the $i$-th Chow group above can be expressed using the chain complex
\[ \ldots \to \bigoplus_{x \in X_{(i+1)}} \KM_1(\kappa(x)) \xrightarrow{\div} \bigoplus_{y\in X_{(i)}} \KM_0(\kappa(y)) \xrightarrow{div} \bigoplus_{z\in X_{(i-1)}} \KM_{-1}(\kappa(z)) \to \ldots \]
known as the Gersten complex $C_*(X, \KM_*, i)$ and the $i$-th Chow group is the homology at the middle term.
Naively we want to replace $\KM$ in the above chain complex by its \enquote{quadratic refinement}, Milnor-Witt K-theory $\KMW$:
\begin{definition}\label[definition]{def:KMW}
	Let $F$ be a finitely generated field extension over $k$.
	Its Milnor-Witt K-theory $\KMW_*(F)$ is the graded (non-commutative) ring generated by the symbols
	\begin{enumerate}[itemindent=2em]
		\item $[a]$ for all $a \in F^\times$ in degree $1$,
		\item $\eta$ in degree $-1$
	\end{enumerate}
	modulo the relations
	\begin{enumerate}[itemindent=2em]
		\item $[a] \cdot [1-a]$ for $a \neq 0, \, 1$,
		\item $[a] + [b] + \eta \cdot [a] \cdot [b] - [ab]$ for $a, \, b \in F^\times$,
		\item $\eta \cdot [a] - [a] \cdot \eta$ for $a \in F^\times$,
		\item $\eta^2 \cdot [-1] + 2 \eta$.
	\end{enumerate}
	The $i$-th degree of this ring is the $i$-th Milnor-Witt K-theory group of $F$, which is denoted by $\KM_i(F)$.
\end{definition}
In this setup, however, it is more difficult to define the map $\div$.
This will be the task of the next section.

This construction of Milnow-Witt K-theory as a quadratic refinement of Milnor K-theory is partly justified by the following result of Morel.
\begin{proposition}[{\cite[3.10]{morel-lecturenotes}}]
	Let $F$ be a finitely generated field extension over $k$.
	For $a \in F^\times$, denote by $\langle a \rangle$ the $1$-dimensional quadratic form over $k$ defined by $(x,y) \mapsto a \cdot x \cdot y$.
	Then the assignment
	\begin{align*}
		\GW(k) & \to \KMW_0(F) \\
		\langle a \rangle &\mapsto 1 + \eta \cdot [a]
	\end{align*}
	extends to a well defined ring isomorphism.
\end{proposition}
\begin{example}
\begin{enumerate}
	\item The Grothendieck-Witt ring of any quadratically closed field is isomorphic to $\Integer$ \cite[II.3.1]{lam}.
	\item $\GW(\Real) \cong \Integer[\Integer/2]$ \cite[II.3.2]{lam}.
\end{enumerate}
\end{example}
In view of this isomorphism we will denote the elements $1 + \eta[a]$ in $\KMW_0(F)$ by $\langle a \rangle$.
In particular, the hyperbolic form $\hyperbolic = \langle 1 \rangle + \langle -1 \rangle = 1 + \langle -1 \rangle \in \GW(F)$ maps to $2 + \eta [-1]$ and we will denote this element by $\hyperbolic$ as well.
Note that under this identification the last relation from \cref{def:KMW} becomes $\eta \hyperbolic$, and the second relation guarantees $\langle ab \rangle = \langle a \rangle \langle b \rangle$.

There are two obvious maps comparing Milnor and Milnor-Witt K-theory:

\begin{definition}\label{def:hyperbolic-map}
	Multiplication with the hyperbolic form $\hyperbolic$ defines a map of graded $\KM_*(F)$-modules
	\[ \hyperbolic \colon \KM_*(F) \to \KMW_*(F) \]
	called the hyperbolic map.
	Dividing out the ideal generated by the element $\eta$ defines a graded ring epimorphism
	\[ \forgetful \colon \KMW_*(F) \to \KM_*(F) \]
	which we will call the reduction map.
\end{definition}
The composition $\forgetful \circ \hyperbolic$ equals multiplication with $\forgetful(\hyperbolic) = \forgetful(2 + \eta[-1]) = 2$.
Both $\hyperbolic$ and $\forgetful$ will extend to maps comparing Chow and Chow-Witt rings.
\begin{example}
\begin{enumerate}
	\item For a quadratically closed field, both $\KM_0$ and $\KMW_0$ are isomorphic to $\Integer$.
	The hyperbolic map is given by multiplication with $2$, and the reduction map is the identity.
	\item The Grothendieck-Witt group of $\Real$ is isomorphic to the group ring $\Integer[\Integer/2]$ with the elements of $\Integer/2$ corresponding to the quadratic forms $\langle 1 \rangle$ and $\langle -1 \rangle$.
	The hyperbolic map is given by multiplication with the hyperbolic form $\langle 1 \rangle + \langle -1 \rangle$, 
	and the reduction map is the $\Integer$-linear extension of the map sending both $\langle 1 \rangle$ and $\langle -1 \rangle$ to $1 \in \KM_0(\Real) \cong \Integer$
\end{enumerate}
\end{example}

\section{The Gersten-Witt Complex}

In this section we follow the exposition of \cite{fasel-lectures} (the original construction is due to \cite{fasel-groupes}) to define a family of chain complexes where the $i$-th degree of the $j$-th complex is
\[ C_*(X, \KMW_*, j) \coloneqq \bigoplus_{x \in X_{(i)}} \KMW_{j+i} (\kappa(x)) \]
and the differentials mimic the divisor map $\div$.
Due to its technical complexity the explicit description of this chain complex is rarely used in computations, as opposed to the tools introduced in \cref{sec:fiber-products,sec:properties}.

First let $F$ be a finitely generated field extension over the base field $k$ and $\nu \colon F \to \Integer$ a discrete valuation.
Denote by $\trivialbundle_\nu$ the valuation ring, $\mathfrak{m}_\nu$ its maximal ideal and $\kappa(\nu) = \trivialbundle_\nu/\mathfrak{m}_\nu$ its residue field, and choose a uniformizing parameter, i.e.\ a generator $\pi$ of $\mathfrak{m}_\nu$.
For $u \in \trivialbundle_\nu$, denote the image of $u$ in $\kappa(\nu)$ by $\overline{u}$.
\begin{theorem}[Morel] \label{prop:gersten-witt-differential}
	Under the above conditions, there exists a unique morphism of graded abelian groups
	\[ \partial^\pi_\nu \colon \KMW_*(F) \to \KMW_{*-1}(\kappa(\nu)) \]
	satisfying 
	\begin{enumerate}
		\item $\partial^\pi_\nu([\pi]\cdot [u_2] \cdot  \ldots \cdot [u_n]) = [\overline{u_2}] \cdot \ldots \cdot [\overline{u_n}]$,
		\item $\partial^\pi_\nu([u_1] \cdot \ldots \cdot [u_n]) = 0$,
		\item $\partial^\pi_\nu(\eta \cdot \alpha) = \eta \partial^\pi_\nu(\alpha)$
	\end{enumerate}
	for any $u_1, \ldots u_n \in \mathcal{O}_\nu^\times$ and $\alpha \in \KMW_*(F)$.
\end{theorem}

\begin{proof}
	See \cite[3.15]{morel-lecturenotes}.
	The idea is to define an auxilary map into the polynomial ring
	\begin{align*}
		\Theta^\pi_\nu \colon \Integer \times \mathcal{O}_\nu^\times = F^\times & \to \KMW_*(\kappa(\nu))[x]/(x^2)/(x^2 - x \cdot [-1]) \\
		(\pi^n\cdot u) & \mapsto [\overline{u}] + \left( \sum_{i=0}^{n-1} \langle (-1)^i \rangle \langle \overline{u} \rangle\right) x \\
		\eta &\mapsto \eta
	\end{align*}
	where the symbol $x$ lives in degree $1$.
	Then check that this is compatible with the relations of Milnor-Witt K-theory and thus can be extended to a graded ring homomorphism on $\KMW_*(F)$.
	Now set
	\[ s^\pi_\nu(\alpha) + \partial^\pi_\nu(\alpha)x \coloneqq \Theta^\pi_\nu(\alpha) \]
	and $\partial^\pi_\nu$ is the desired map.
	
	Uniqueness follows because the group $\KMW_n(F)$ is generated by all symbols of the form $\eta^m\cdot [u_1] \cdot  \ldots \cdot [u_n]$ and $\eta^m[\pi]\cdot [u_2] \cdot  \ldots \cdot [u_n]$.
\end{proof}


For Milnor-Witt K-theory, unlike Milnor K-theory, this residue homomorphism depends not only on $\nu$ but also on the choice of $\pi$.
To remove this dependency we will introduce twisted Milnor-Witt K-theory.

\begin{definition}
	Let $F$ be a finitely generated field extension over $k$ and $\linebundle$ a line bundle over $F$.
	Denote by $\linebundle^0$ the complement of the zero section of $\linebundle$.
	Define 
	\[ \KMW_*(F, \linebundle) \coloneqq \KMW_*(F) \otimes_{\Integer[F^\times]} \Integer[\linebundle^0] \, . \]
	The analogous definition can be made for Milnor K-theory and this allows to define twisted versions of the hyperbolic and reduction map:
	\[ \hyperbolic_\linebundle \colon \KM_*(F, \linebundle) \to \KMW_*(F, \linebundle), \quad \forgetful_\linebundle \colon \KMW_*(F, \linebundle) \to \KM_*(F, \linebundle) \]
\end{definition}
Recall that a line bundle over $F$ is just a $1$-dimensional vector space, and its zero section is the zero element.
For any $l \in \linebundle^0$ there is an isomorphism $\linebundle \cong F$ descending to $\linebundle^0 \cong F^\times$ and ultimately
\begin{align*}
	\KMW_*(F) &\xrightarrow{\cong} \KMW_*(F, \linebundle) \\
	a & \mapsto a \otimes l
\end{align*}
but there is no canonical choice for this.
Twisted Milnor-Witt K-theory is a priori not a ring, but only a graded $\KMW_*(F)$-module.

Now for a valuation $\nu \colon F \to \Integer$ and a uniformizing parameter $\pi$ define a twisted residue homomorphism:
\begin{align*}
	\partial^\pi_\nu \colon \KMW_*(F,\linebundle) & \to \KMW_{*-1}(\kappa(\nu), (\mathfrak{m}_\nu/\mathfrak{m}_\nu^2)^\vee \otimes \linebundle_{\kappa(\nu)}) \\
	a \otimes l & \mapsto \partial^\pi_\nu(a) \otimes \overline{\pi}^\vee \otimes l
\end{align*}
Here $\linebundle_{\kappa(\nu)}$ denotes the restriction of $\linebundle$ to $\kappa(\nu)$, $\overline{\pi}$ denotes the class of $\pi$ in $\mathfrak{m}_\nu/\mathfrak{m}_\nu^2$, and $\overline{\pi}^\vee$ its $\kappa(\nu)$-dual.
The appearance of $(\mathfrak{m}_\nu/\mathfrak{m}_\nu)^\vee$ will be explained in a moment.
This homomorphism is independent of the choice of $\pi$: 
For any $\pi^\prime \coloneqq u \cdot \pi$, we have $(\overline{\pi}^\prime )^\vee = \langle u ^{-1} \rangle \overline{\pi}^\vee$ and further $\partial_\nu^{\pi^\prime} (a) = \langle u ^{-1} \rangle \partial_\nu^\pi (a)$ by \cite[Remark 1.9]{fasel-lectures}, therefore we compute
\[ \partial^{\pi^\prime}_\nu (a) \otimes (\overline{\pi}^\prime) ^\vee = \langle u^{-1} \rangle \partial^\pi_\nu (a) \otimes \langle u \rangle \overline{\pi}^\vee = \partial^\pi_\nu(a) \otimes \overline{\pi}^\vee \, . \]
Thus one can drop $\pi$ from the notation.
Note that this is still dependent on the choice of $\nu$.

For the construction of Chow-Witt groups we will twist by a line bundle of the following form.
\begin{definition}
	Let $f \colon X \to Y$ be a morphism of schemes, with $f^\sharp \colon f^{-1}\trivialbundle_Y \to \trivialbundle_X$ the associated map of sheaves.
	\begin{enumerate}
		\item Let $\vectorbundlesheaf$ be an $\trivialbundle_X$-module.
		A $Y$-derivation of $\trivialbundle_X$ into $\vectorbundlesheaf$ is a map $D \colon \trivialbundle_X \to \vectorbundlesheaf$ of $\trivialbundle_X$-modules such that $D \circ f^\sharp = 0$ and $D$ satisfies the Leibniz rule
		\[ D(ab) = aD(b) + D(a)b \, . \]
		\item The sheaf (or module) of relative differentials $\Omega_{X/Y}$ is the $\trivialbundle_X$-module representing the functor $\operatorname{Der}_{f^\sharp}(\trivialbundle_X, -)$,
		equipped with the universal derivation $d \colon \trivialbundle_X \to \Omega_{X/Y}$.
		\item The determinant $\det \vectorbundlesheaf$ of a rank $r$ $\trivialbundle_X$-module is defined as the highest exterior power $\Lambda^r \vectorbundlesheaf$.
		The determinant of $\Omega_{X/Y}$ is sometimes denoted $\omega_{X/Y}$.
	\end{enumerate}
\end{definition}
\noindent The existence of $\Omega_{X/Y}$ is shown for example in \cite[II.8]{hartshorne}.
If $f$ is smooth, $\Omega_{X/Y}$ is locally free and is also referred to as the cotangent bundle of $f$.
\begin{example}
	\begin{enumerate}
		\item The sheaf of differentials $\Omega_{X/X}$ along the identity on $X$ is the zero sheaf.
		\item For a projective space $\projective_k^r$, the sheaf $\omega_{\projective^r/k} = \det \Omega_{\projective^r/k}$ is isomorphic to the twisting sheaf $\trivialbundle_{\projective^r}(-r-1)$ \cite[II.8.20.1]{hartshorne}.
	\end{enumerate}
\end{example}
We will consider the Milnor-Witt K-theory of $\kappa(W)$ twisted by the line bundle $\det(\Omega_{\kappa(W)/k})$.
If $i \colon X \to Y$ is a regular embedding of smooth schemes, there is an isomorphism
\begin{align*}
	\det(\Omega_{X/k}) &\cong i^* \det(\Omega_{Y/k}) \otimes \det(N_XY) 
\end{align*}
of line bundles over $X$, where $N_X Y$ is the normal bundle of $X$ in $Y$ (see after \cite[Eq. 1.4]{fasel-lectures}).
For a discrete valuation $\nu \colon F \to \Integer$ with ring of integers $\trivialbundle_\nu$, maximal ideal $\mathfrak{m}_\nu$ and residue field $\kappa(\nu)$, the quotient map $\trivialbundle_\nu \to \kappa(\nu)$ induces a map $F = \operatorname{Frac}(\trivialbundle_\nu) \to \kappa(\nu)$ with kernel $\mathfrak{m}_\nu$, exhibiting $(\mathfrak{m}_\nu/\mathfrak{m}_\nu^2)^\vee$ as the normal bundle of the morphism of schemes $\Spec(\kappa(\nu)) \to \Spec(F)$. 
Therefore the above isomorphism reads
\begin{align}\label{eq:cotangent-bundle-iso}
	\det(\Omega_{\kappa(\nu)/k}) &\cong F \otimes \det(\Omega_{F/k}) \otimes (\mathfrak{m}_\nu/\mathfrak{m}_\nu^2)^\vee
\end{align}
in this case.
This also explains the appearance of $(\mathfrak{m}_\nu/\mathfrak{m}_\nu^2)^\vee$ in the construction of the twisted residue homomorphism - this homomorphism can also be interpreted as
\[ \partial^\pi_\nu \colon \KMW_*(F, \det \Omega_{F/k} \otimes \linebundle) \to \KMW_{*-1}(\kappa(\nu), \det \Omega_{\kappa(\nu)/k} \otimes \linebundle_{\kappa(\nu)}) \, .\]

Next we will construct the so-called transfer morphism which will form the foundation for both the Gersten-Witt differential and pushforwards.
Consider the polynomial ring $F[t]$ and a monic irreducible polynomial $p \in F[t]$ and denote $F(p) = F[t]/(p)$.
Then the $p$-adic valuation $\nu_p \colon F(t) \to \Integer$ determines a residue homomorphism
\begin{align*} 
\partial_p \colon \KMW_*(F(t), \det(\Omega_{F(t)/k})) \to &\KMW_{*-1}(F(p), (\mathfrak{m}_p/\mathfrak{m}_p^2)^\vee \otimes_{F[t]} \det(\Omega_{F[t]/k})) \\
\cong &\KMW_{*-1}(F(p), \det(\Omega_{F(p)/k})) \, . 
\end{align*}
Fasel \cite[1.20]{fasel-lectures} proves the following, based on \cite[3.24]{morel-lecturenotes}.
\begin{proposition}
	The sequence
	\begin{multline*} 0 \to \KMW_i(F, \det(\Omega_{F/k})) \to \KMW_i(F(t), \det(\Omega_{F(t)/k})) \\ 
	\xrightarrow{\sum_p \partial_p} \bigoplus_p \KMW_{i-1}(F(p), \det(\Omega_{F(p)/k})) \to 0  
	\end{multline*}
	where the sum in the third term goes over all monic irreducible polynomials in $F[t]$,
	is split exact.
\end{proposition}
Denote by $\partial_\infty$ the residue homomorphism determined by the discrete valuation
\begin{align*}
\nu_\infty \colon F(t) &\to \Integer\\
f/g &\mapsto \deg(g) - \deg(f) \, .
\end{align*}

\begin{definition}
\begin{enumerate}
	\item For a finite field extension $F/k$ and a monic irreducible polynomial $p \in F[t]$ define the transfer morphism as the composition
	\begin{multline*} 
		\transfer_{F(p), F} \colon \KMW_{i-1}(F(p), \det(\Omega_{F(p)/k})) \subseteq \bigoplus_p \KMW_{i-1}(F(p), \det(\Omega_{F(p)/k})) \\
		\xrightarrow{s} \KMW_i(F(t), \det(\Omega_{F(t)/k})) \xrightarrow{\partial_\infty} \KMW_{i-1}(F, \det(\Omega_{F/k})) 
	\end{multline*}
	where $s$ is a section of $\sum_p \partial_p$ as described in the previous proposition.
	\item If $L/F$ is another finite field extension, choose a filtration
	\[ F = F_0 \subseteq \ldots \subseteq F_n = L \]
	such that for all $i$ the field extension $F_i$ is of the form $F_{i-1}(p_i)$ for some monic irreducible polynomial $p_i \in F_{i-1}[t]$, and set
	\[ \transfer_{L,F} = \transfer_{F_1, F_0} \circ \ldots \circ \transfer_{F_n, F_{n-1}} \colon \KMW_i(L, \det(\Omega_{L/k})) \to \KMW_i(F, \det(\Omega_{F/k})) \, . \]
\end{enumerate}
\end{definition}
The first part of the definition is independent of the choice of section $s$ because the composite
\[ \KMW_i(F, \det \Omega_{F/k}) \to \KMW_i(F(p), \det \Omega_{F(p)/k}) \xrightarrow{\partial_\infty} \KMW_{i-1}(F, \det \Omega_{F/k}) \]
is trivial as explained in \cite[Remark 1.14]{fasel-lectures}.
The second part is proved by \cite[4.27]{morel-lecturenotes} to be independent of the chosen filtration.

Now let $X$ be a (finite type, separated) scheme over $k$.
We want to construct a residue homomorphism $\KMW_{i+1}(\kappa(x), \det(\Omega_{\kappa(x)/k})) \to \KMW_i(\kappa(y), \det(\Omega_{\kappa(y)/k}))$ for fixed $x \in X_{(i)}$, $y \in X_{(i+1)}$.
If $x \notin \overline{\{y\}}$ set the residue homomorphism to be zero.
Otherwise denote by $\widetilde{Z}$ the normalization of $\overline{\{y\}}$ with all points of codimension $\geq 2$ removed.
Normal implies regular and therefore, since $k$ is perfect, $\widetilde{Z}$ is smooth.
The morphism $i \colon \widetilde{Z} \to X$ is finite.
Under these conditions, for any point $z \in \widetilde{Z}$ of codimension $1$ lying over $y$ the composition
\[ \nu_{\div} \colon \kappa(x)^\times \xrightarrow{\div} \Integer\langle \overline{\{z^\prime\}} \mid z^\prime \in \widetilde{Z}^{(1)}, \, i(z^\prime) = y  \rangle \to \Integer \langle z \rangle  \] 
is a discrete valuation with residue field $\kappa(z)$ (see \cite[before Lemma 6.1]{hartshorne}) and we consider the associated boundary morphism 
\[ \partial_{\div} \colon \KMW_*(\kappa(x), \linebundle) \to \KMW_{*-1}(\kappa(z), (\mathfrak{m}_{\div}/\mathfrak{m}_{\div}^2)^\vee \otimes \linebundle_{\widetilde{Z}}) \, . \]
Set $\linebundle = \det \Omega_{\kappa(x)/k}$ and insert the isomorphism \cref{eq:cotangent-bundle-iso} to obtain
\[ \partial_{\div} \colon \KMW_*(\kappa(x), \det \Omega_{\kappa(x)/k}) \to \KMW_{*-1}(\kappa(z), \det \Omega_{\kappa(z)/k}) \, . \]

Further if $i(z)  =y$ the field extension $\kappa(y)/\kappa(z)$ is finite and thus there is a canonical transfer
\[ \transfer_{\kappa(z), \kappa(y)}  \colon \KMW_i(\kappa(z), \det(\Omega_{\kappa(z)/k})) \to \KMW_i(\kappa(y), \det(\Omega_{\kappa(y)/k})) \, . \]
Summing over the composition of these two maps for all $z \in \widetilde{Z}^{(1)}$ with $i(z) = y$ finally yields a boundary morphism
\[ \partial_{x,y} \colon \KMW_*(\kappa(x), \det(\Omega_{\kappa(x)/k})) \to \KMW_{*-1}(\kappa(y), \det(\Omega_{\kappa(y)/k})) \]
for the Gersten-Witt complex.
Morel \cite[5.31]{morel-lecturenotes} proves that this is in fact a chain differential under the assumption that the characteristic of $k$ is coprime to $2$,
and \cite{feld} removes this assumption.
With this we can finally make the following definition.

\begin{definition}
	Let $X$ be a scheme over $k$ and $\linebundle$ a line bundle over $X$.
	For every integer $j$ define the homological Gersten-Witt complex by
	\begin{align*}
		C_i(X, \KMW_*, j,\linebundle) &= \bigoplus_{x \in X_{(i)}} \KMW_{j+i} (\kappa(x), \linebundle \otimes \det(\Omega_{\kappa(x)/k})) \\
		\intertext{with differential $\partial \colon C_i(X, \KMW_*,j,\linebundle) \to C_{i-1}(X, \KMW_*,j,\linebundle)$ given by summing over the residue homomorphisms defined above.
		If further $X$ is smooth of dimension $d$, define the cohomological Gersten-Witt complex}
		C^i(X, \KMW_*,j,\linebundle) &= C_{d-i}(X, \KMW_*, j-d, \linebundle \otimes \det \Omega_{X/k}^\vee) = \\
		&\bigoplus_{x \in X^{(i)}} \KMW_{j-i} (\kappa(x), \linebundle \otimes \det(\Omega_{\kappa(x)/k}) \otimes \det \Omega_{X/k}^\vee) 
	\end{align*}
	with differential $\partial \colon C^i(X, \KMW_*, j,\linebundle) \to C^{i+1}(C, \KMW_*, j,\linebundle)$.
\end{definition}
The additional twist $\det \Omega_{X/k}^\vee$ will later ensure that the pullback along a flat morphism of schemes induces no change of twist on cohomological Chow-Witt groups.
This makes cohomological Chow-Witt groups with fixed twist a contravariant functor and will further lead to the ring structure being $\Pic(X)/2)$-graded.
In contrast, for homological Chow-Witt groups the definition is chosen so that pushforward induces no change of twist.
\begin{definition}
	Let $X$ be a scheme over $k$ and $\linebundle$ a line bundle over $X$.
	The homological respectively (if $X$ is smooth) cohomological Chow-Witt groups twisted by $\linebundle$ are defined by
	\begin{align*}
		\chowwitt_i(X,\linebundle) &= H_i(C_*(X,\KMW_*, -i,\linebundle)_*) \\
		\chowwitt^i(X,\linebundle) &= H^i(C^*(X, \KMW_*, i, \linebundle)) \, . \\
		\shortintertext{The Milnor-Witt cohomology groups are defined by}
		H^i(X, \KMW_j, \linebundle) &= H^i(C^*(X, \KMW_*, j, \linebundle)) \, .
	\end{align*}
\end{definition} 
Spelled out, for homological Chow-Witt groups we take homology at the left-hand term of
\[ \ldots 
\to \bigoplus_{x \in X_{(i)}} \KMW_0(\kappa(x), \det(\Omega_{\kappa(x)/k})) \to \bigoplus_{y \in X_{(i-1)}} \KMW_{-1}(\kappa(y), \det(\Omega_{\kappa(y)/k})) \to \ldots \]
and for cohomological Chow-Witt groups we take cohomology at the left-hand term of
\[ \ldots 
\to \bigoplus_{x \in X^{(i)}} \KMW_0(\kappa(x), \det(\mathfrak{m}_x/\mathfrak{m}_x^2)^\vee) \to \bigoplus_{y \in X^{(i+1)}} \KMW_{-1}(\kappa(y), \det(\mathfrak{m}_y/\mathfrak{m}_y^2)^\vee) \ldots \]
i.e.\ we always take (co)homology at the term containing subschemes $V$ of (co)dimension $i$ and $\KMW_0$, 
and in both complexes the differentials decrease dimension (=increase codimension) of subschemes and decrease the degree of $\KMW_*$.
Note that Chow-Witt groups are precisely Milnor-Witt cohomology groups in diagonal bidegrees, i.e.\ $i=j$.

\begin{remark}
	This is not the original definition;
	\cite{fasel-thesis} and \cite{fasel-groupes} originally introduced the Gersten-Witt complex as a fiber product of complexes which we will present in the next section.
	The construction of the complex as detailed above is due to \cite{morel-lecturenotes} and shown by \cite[2.8]{asok-fasel16} to be equivalent to the classical definition.
\end{remark}

\section{Fiber Product Decomposition}\label{sec:fiber-products}

In this section we will introduce some more cohomology theories and their relation to Chow-Witt groups respectively Milnor-Witt cohomology.
This will not only provide context for Chow-Witt groups but also some useful tools for later computations. 

\begin{definition}
	Let $F$ be a finitely generated field extension over $k$.
	The Witt ring $\witt(F)$ consists of isometry classes of quadratic forms over $F$ modulo metabolic forms (direct sum of several copies of the hyperbolic form).
	It becomes a commutative ring with direct sum and tensor product.
\end{definition}
\noindent It is immediate from this construction that the Witt ring is isomorphic to $\GW(F)/\hyperbolic$.
\begin{example}
\begin{enumerate}
	\item The Witt ring of a quadratically closed field is isomorphic to $\Integer/2$ \cite[3.1]{lam}.
	\item The Witt ring of $\Real$ is isomorphic to $\Integer$ \cite[3.2]{lam}.
\end{enumerate}
\end{example}
\begin{definition}
	Let $F$ be a finitely generated field extension over $k$.
	Consider the rank map $\rank \colon \GW(F) \to \Integer$.
	It descends to a map $\witt(F) \to \Integer/\rank(\hyperbolic) = \Integer/2$ which is also called rank.
	\begin{enumerate}
		\item The kernel of the latter is called the fundamental ideal $I(F)$ of the Witt ring.
		\item Denote by $I^i(F)$ the powers of this ideal and by $\ibar^i(F)$ the quotient $I^i(F)/I^{i+1}(F)$.
		\item The Witt K-theory of $F$ denoted $I^*(F)$ (or sometimes $\KW_*(F)$) is the graded ring given by $\KW_i(F) = I^i(F)$, where for $i \leq 0$ we set $I^i(F) = \witt(F)$.
		\item The quotient groups $\ibar^i(F)$ also assemble into a graded ring $\ibar^*(F)$, sometimes called reduced Witt K-theory.
	\end{enumerate}
	All of these have twisted variants $\witt(F, \linebundle)$, $I(F, \linebundle)$, $I^*(F, \linebundle)$, $\ibar^*(F, \linebundle)$ for a line bundle $\linebundle$ over $F$.
\end{definition}
It follows from this definition that
\[ \witt(F)/I(F) \cong (\GW(F) / \hyperbolic) / I(F) \cong (\GW(F) / I(F)) / \hyperbolic \cong \Integer / \rank(\hyperbolic) \cong \Integer/2 \, . \]
Thus one can form the following commutative square:
\[ \begin{tikzcd}
	\GW(F) \ar[r, "\rank"] \ar[d] & \Integer \ar[d]  \ar[r, phantom, "\cong"] & \KM_0(k) \ar[d]\\
	\witt(F) \ar[r, "\rank"] & \Integer/2 \ar[r, phantom, "\cong"] & \KM_0(k)/2
\end{tikzcd} \]
Thus the kernels of the two rank maps are isomorphic,
which justifies calling the kernel of $\rank \colon \GW(F) \to \Integer$ the fundamental ideal of the Grothendieck-Witt ring and denoting it by $I(F)$ as well. 
\begin{example}
	\begin{enumerate}
		\item If $F$ is a quadratically closed field (for example $\Complex$), then $\GW(F) \cong \Integer$ and $\witt(F) \cong \Integer/2$.
		Both rank maps are the identity and the fundamental ideal is zero.
		\item If $F = \Real$, then $\GW(F) \cong \Integer[\Integer/2]$ and $\witt(F) \cong \Integer$.
		The top rank map sends both elements $\langle 1 \rangle$, $\langle -1 \rangle$ of $\Integer/2$ to $1$ and the fundamental ideal is thus isomorphic to $\Integer$ and generated by the element $\langle 1 \rangle - \langle -1 \rangle$.
		The left vertical map sends $\langle 1 \rangle$ to $1$ and $\langle -1 \rangle$ to $-1 \in \witt(F) \cong \Integer$, and the bottom rank map is just the quotient map.
	\end{enumerate}
\end{example}
The Milnor conjecture on quadratic forms proved by \cite[4.1]{orlov-vishik-voevodsky} asserts that $\ibar^*(F)$ is isomorphic to $\KM_*(F)/2$ as graded ring.
Hence one can even form a commutative square of graded rings
\[ \begin{tikzcd}
	\KMW_*(F, \linebundle) \ar[r, "\forgetful"] \ar[d] & \KM_*(F) \ar[d]  \\
	I^*(F, \linebundle) \ar[r, "\reduction"] & \KM_*(F)/2 
\end{tikzcd} \]
which by \cite[5.3]{morel-puissances} is a pullback square.
This is a fact that, to the author's best knowledge, actually relies on the characteristic of the base field $k$ being different from $2$, compare also \cite[Rem. 1.5]{fasel-lectures}.

The inclusions $I^i(F, \linebundle) \subseteq I^{i-1}(F, \linebundle)$ are given by multiplication with $\eta \in \KMW_{-1}(F)$,
which is compatible with the differential of the Gersten-Witt complex by construction as in \cref{prop:gersten-witt-differential}. 
Therefore one can define a subcomplex
\begin{align*} 
	C^i(X, I^*, j,\linebundle) &= \bigoplus_{x \in X^{(i)}} I^{j+i} (\kappa(x), \linebundle \otimes \det\Omega_{\kappa(x)}\otimes \det \Omega_{X/k}^\vee) 
	\shortintertext{and a quotient complex}
	C^i(X, \ibar^*, j) &= \bigoplus_{x \in X^{(i)}} \ibar^{j+i} (\kappa(x)) \, .
	\intertext{Similarly, dividing out $\eta$ is compatible with the Gersten-Witt differential and thus one an define a quotient complex}
	C^i(X, \KM_*, j) & = \bigoplus_{x \in X^{(i)}} \KM_{i+j}(\kappa(x)) \, .
\end{align*}
Observe that the Gersten complex and cohomology with coefficients in $\KM_*(F)$ is independent of twist, thus those for $\KM_*(F)/2$ and $\ibar^*$ are as well and therefore the line bundle $\linebundle$ is omitted from notation in these cases.
\begin{definition} \label{def:I-cohomology}
	Let $X$ be a smooth scheme.
	Define its $i$-th $I^j$-cohomology 
	\begin{align*} H^i(X, I^j, \linebundle) &= H^i(C^*(X, I^*, j, \linebundle) \\
	\shortintertext{and reduced $I^j$-cohomology}
	H^i(X, \ibar^j) &= H^i(C^*(X, \ibar^*, j))
	\intertext{The groups $H^i(X, I^0, \linebundle)$ are sometimes also called Witt cohomology and denoted $H^i(X, \witt, \linebundle)$, in view of $I^0(F) = \witt(F)$.
		Further define the Milnor cohomology and (cohomological) Chow groups of $X$}
	H^i(X, \KM_j) &= H^i(C^*(X, \KM_*, j)) \\
	\chow^i(X) &= H^i(X, \KM_i) \, .
	\end{align*}
	
\end{definition}
\noindent According to \cite[2.8]{asok-fasel16} 
there is in fact a pullback square of chain complexes:
\[ \begin{tikzcd} \label{eq:chain-complex-fiber-product}
	C^*(X, \KMW_*, j, \linebundle) \ar[r] \ar[d] \ar[dr, phantom, pullback] & C^*(X, \KM_*, j) \ar[d] \\
	C^*(X, I^*, j, \linebundle) \ar[r] & C^*(X, \ibar^*, j)
\end{tikzcd} \]
Further one can deduce the following short exact sequences of graded rings from the definitions of $\KMW_*$, $\KM_*$, $I^*$ and $\ibar^*$, and analogous sequences of graded modules for their twisted versions.
\begin{gather*} \label{eq:KMW-decomposition-ses}
	0 \to \KM_*(F) \xrightarrow{\hyperbolic} \KMW_*(F) \to I^*(F) \to 0 \\
	0 \to \KM_*(F) \xrightarrow{\cdot 2} \KM_*(F) \to \KM_*(F)/2 \cong \ibar^*(F) \to 0 \\
	0 \to I^*(F) \to \KMW_*(F) \xrightarrow{\forgetful} \KM_*(F) \to 0 \\
	0 \to I^{*+1}(F) \to I^{*}(F) \xrightarrow{\reduction} \ibar^{*+1}(F) \to 0
\end{gather*}
All of these induce short exact sequences on chain complexes and thus long exact cohomology sequences which assemble into the \enquote{key diagram} of \cite[Section 2.4]{hornbostel-wendt}.
\begin{equation} \label{eq:key-diagram} \begin{tikzcd}
	& \chow^i(X) \ar[r, "\id"] \ar[d, "\hyperbolic_\linebundle"] & \chow^i(X) \ar[d, "\cdot 2"] & \\
	H^i(X, I^{i+1}, \linebundle) \ar[r] \ar[d, "\id"] & \chowwitt^i(X, \linebundle) \ar[r, "\forgetful"] \ar[d, "\operatorname{mod} \hyperbolic_\linebundle"] & \chow^i(X) \ar[r, "\partial_\linebundle"] \ar[d, "\operatorname{mod} 2"] & H^{i+1}(X, I^{i+1}, \linebundle) \ar[d, "\id"] \\
	H^i(X, I^{i+1}, \linebundle) \ar[r, "\eta"] & H^i(X, I^i, \linebundle) \ar[r, "\reduction"] \ar[d] & \chowmodtwo^i(X) \ar[r, "\bockstein"] \ar[d] & H^{i+1}(X, I^{i+1}, \linebundle) \\
	& 0 & 0 & 
\end{tikzcd} \end{equation}
Here $\chowmodtwo$ denotes mod $2$-Chow groups.
The maps $\forgetful$ and $\reduction$ are both called reduction map, and $\hyperbolic_\linebundle$ is called the hyperbolic map. 
The differential $\bockstein$ is referred to as the Bockstein.
The lower horizontal sequence consisting of the maps $\bockstein$, $\eta$, $\reduction$ sometimes gets the nickname Bär sequence.
The fact that the leftmost and rightmost vertical as well as the topmost horizontal arrows are identity maps is due to the pullback square of chain complexes above.
The four maps in the central square are ring homomorphisms.

Hornbostel-Wendt prove the following useful statement about this diagram.
This already mentions the ring structure on Chow-Witt and $I^j$-cohomology, which we will discuss in more detail in \cref{sec:ring-structure}.
\begin{proposition}[{\cite[Prop.\ 2.11]{hornbostel-wendt}}] %
	\label{prop:chowwitt-is-chow-x-I}
	Let $X$ be a smooth scheme over $k$. 
	Consider the canonical ring homomorphism
	\[ c \colon \chowwitt^\tot(X) \coloneqq \bigoplus_{i \in \Integer} \bigoplus_{\linebundle \in \Pic(X)/2} \chowwitt^i(X, \linebundle) \to H^{*}(X, I^{*}, -) \times_{\bigoplus_\linebundle \chowmodtwo^{*}(X)} \bigoplus_{\linebundle \in \Pic(X)/2} \ker \partial_\linebundle \]
	with structure maps for the fiber product as in the central square of the \enquote{key diagram} above.
	This morphism is always surjective. 
	It is injective in a given degree $(i, \linebundle)$ if one of the following conditions holds:
	\begin{enumerate}
		\item $\chow^i(X)$ has no non-trivial $2$-torsion.
		\item the map $\eta \colon H^i(X, I^{i+1}, \linebundle) \to H^i(X, I^i, \linebundle)$ is injective.
	\end{enumerate}
\end{proposition}
The map $c$ is in fact a morphism of $\GW(k)$-algebras, where on the right-hand side $\GW(k)$ acts on the first factor via the action of $\witt(k) \cong \GW(k)/\hyperbolic$ and on the second via $\KM_0(k) \cong \GW(k)/I(k)$.

\section{Functoriality: Flat Pullback and Proper Pushforward} \label{sec:functoriality}

The first goal of this section is to define for every proper morphism $f \colon X \to Y$ of schemes and a line bundle $\linebundle$ on $Y$ a map $f_* \colon \chowwitt_i(X, f^*\linebundle) \to \chowwitt_i(Y, \linebundle)$ on Chow-Witt groups, following the exposition of \cite[Section 2.3]{fasel-lectures} (originally due to \cite{fasel-groupes}).

Let $X$, $Y$ be schemes over $k$, $\linebundle$ a line bundle over $Y$.
Let $x \in X_{(i)}$ and $y =f(x) \in Y_{(j)}$.
We define a map 
\[ (f_*)_{x,y} \colon \KMW_{*}(\kappa(x), \det(\Omega_{\kappa(x)/k})\otimes \linebundle) \to \KMW_{*+j-i}(\kappa(y), \det(\Omega_{\kappa(y)/k}) \otimes \linebundle) \]
by setting $(f_*)_{x,y}=0$ if the field extension $\kappa(x) \subseteq \kappa(y)$ is infinite, and $(f_*)_{x,y} = \transfer_{\kappa(x), \kappa(y)}$ if that field extension is finite and thus $i=j$.

Recall that a morphism $f \colon X \to Y$ is called proper if it is separated, of finite type, and universally closed.
\begin{theorem}[Proper Pushforward (Fasel)] \label{prop:proper-pushforward}
Let $f \colon X \to Y$ be a proper morphism and denote by $c$ the codimension of its image in $Y$.
Then the morphisms 
\begin{align*}
	f_* \colon C_*(X, \KMW_*, j, f^*\linebundle) & \to C_*(Y, \KMW_*, j, \linebundle) \\
	\intertext{and, if $X$ and $Y$ are smooth,}
	f_* \colon C^*(X, \KMW_*, j, f^*\linebundle \otimes \det \Omega_{X/k}) & \to C^{*+c}(Y, \KMW_*, j+c, \linebundle \otimes \det \Omega_{Y/k})
\end{align*} 
obtained by summing over the $(f_*)_{x,y}$ defined above are both morphisms of complexes.
The resulting maps on Chow-Witt groups and Milnor-Witt cohomology are also denoted $f_*$.
\end{theorem}
\begin{proof}
	See \cite[Corollaire 10.4.5]{fasel-groupes}.
\end{proof}

As the pushforward morphism is constructed in \cite[Corollaire 10.4.5]{fasel-groupes} from the fiber product of chain complexes of \cite[2.8]{asok-fasel16} (see before \cref{prop:chowwitt-is-chow-x-I}), it is evidently compatible with the structure maps $\forgetful$, $\modulo \hyperbolic$, $\reduction$ and $\modulo 2$ of this fiber product as well as the hyperbolic map $\hyperbolic_\linebundle$ which on the fiber product can be represented as $\cdot (0, 2)$.

There is an isomorphism of line bundles $\det f^* \Omega_{Y/k}^\vee \otimes \det \Omega_{X/k} \cong \det \Omega_{X/Y}$ by \cite[Prop. II.8.11]{hartshorne}.
Hence substituting $\linebundle$ for $\det \Omega_{Y/k}^\vee \otimes \linebundle$ allows to write the pushforward map as
\begin{equation} \label{eq:pushworward-line-bundle-iso}
	\chowwitt^{i-c}(X, \det \Omega_{X/Y} \otimes f^*\linebundle) \to \chowwitt^i(Y, \linebundle) \, . 
\end{equation}

\noindent For pullback morphisms there are the following two statements.
\begin{theorem}[Flat Pullback (Fasel)] \label{prop:flat-pullback}
	Let $f \colon X \to Y$ be a flat morphism and $\linebundle$ a line bundle over $X$.
	Then there are morphisms of complexes
	\begin{align*}
		f^* \colon C_*(Y, \KMW_*, j, f^*\linebundle) & \to C_{*-\codim_Y X}(X, \KMW_*, j - \codim_Y X, \det \Omega_{X/Y} \otimes \linebundle) \\
		\intertext{and, if $X$ and $Y$ are smooth,}
		f^* \colon C^*(Y, \KMW_*, j, f^*\linebundle) & \to C^{*}(X, \KMW_*, j, \linebundle) \, .
	\end{align*} 
	The resulting maps on Chow-Witt groups and Milnor-Witt cohomology are denoted $f^*$.
\end{theorem}

\begin{proof}[Proof sketch]
	This is a very vague sketch. 
	See \cite[Corollaire 10.4.2]{fasel-groupes} for details.
	
	A morphism of schemes $f \colon X \to Y$ induces a functor $f^* \colon D^b (\presheaves(Y)) \to D^b(\presheaves(X))$ between the derived categories of bounded presheaves over $Y$ and $X$
	and that induces a morphism on Witt groups (which we will not introduce in detail in this work).
	Checking that this morphism is compatible with the inclusions of the fundamental ideals as well as the differentials of the Gersten-Witt complex produces two morphisms of complexes
	\begin{align*}
		f^* \colon C_*(Y, I^*, j, f^*\linebundle) & \to C_{*-\codim_Y X}(X, I^*, j - \codim_Y X, \linebundle \otimes \Omega_{X/Y}) \\
		\intertext{and, if $X$ and $Y$ are smooth,}
		f^* \colon C^*(Y, I^*, j, f^*\linebundle) & \to C^{*}(X, I^*, j, \linebundle) \, .
	\end{align*} 
	Forming the fiber product with the Gersten complex for Milnor K-theory, after checking compatibility of $f^*$ with the structure maps of this fiber product, yields the statement.
\end{proof}

Again, the pullback morphism is compatible with the structure maps $\forgetful$ and $\modulo \hyperbolic$ of this fiber product as well as the hyperbolic map $\hyperbolic_\linebundle$.

\begin{proposition}[General Pullback (Fasel)] \label{prop:general-pullback}
	Let $X$ and $Y$ be smooth schemes over $k$, $\linebundle$ a line bundle over $Y$ and $f \colon X \to Y$ any morphism of schemes.
	Then there is a pullback map 
	\[ f^* \colon \chowwitt^i(Y, \linebundle) \to \chowwitt^i(X, f^*\linebundle) \]
\end{proposition}

\begin{proof}
	See \cite[3.3]{fasel-lectures}.
\end{proof}

\section{Properties} \label{sec:properties}

The following will make repeated appearance in many computations throughout this work.
\begin{proposition}[Localization Sequence]%
	\label{chowwitt-localization}
	Let $X$ be a smooth scheme over $k$, $\linebundle$ a line bundle over $X$, $\iota \colon Z \hookrightarrow X$ a closed smooth subscheme of codimension $c$ and $j \colon U \hookrightarrow X$ its complement.
	Then there is for each $i \in \Integer$ a long exact sequence called the localization sequence:
	\begin{multline*} 
		\ldots \to \chowwitt^{i-c}(Z, \iota^* \linebundle \otimes \det \Omega_{Z/X}) \xrightarrow{\iota_*} \chowwitt^i(X, \linebundle) \xrightarrow{j^*} \chowwitt^i(U, j^*\linebundle) \\
		\to H^{i-c+1}(Z, \KMW_{i-\codim_X(Z)}, \iota^*\linebundle \otimes \det\Omega_{Z/X}) \to \ldots 
	\end{multline*}
\end{proposition}
This sequence continues indefinitely to the left and right meaning that Milnor-Witt cohomology groups of all possible bidegrees occur.

Localization sequences are compatible with pullbacks of scheme morphisms in the following sense.
\begin{lemma}\label{lem:locseq-pullback}
	Let $f \colon X \to Y$ be a morphism of smooth schemes and $\linebundle$ a line bundle on $Y$.
	Let $\iota \colon Z \subseteq Y$ a closed subscheme of codimension $c$ and $j \colon U = Y \setminus Z \subseteq Y$ its complement, and denote $X_Z = f^{-1}(Z)$ and $X_U = f^{-1}(U)$.
	Assume that $X_Z$ is smooth and further that there is an isomorphism of bundles $f^* N_Z Y \cong N_{X_Z} X$.
	Then there is a commutative ladder of localization sequences:
	\[ \begin{tikzcd}[column sep=small]
		\ldots \ar[r] & \chowwitt^{i-c}(Z, \iota^* \linebundle \otimes \det \Omega_{Z/Y}) \ar[r, "\iota_*"] \ar[d, "f^*"] & \chowwitt^i(Y, \linebundle) \ar[r, "j^*"] \ar[d, "f^*"] & \chowwitt^i(U, j^*\linebundle) \ar[r, "\partial"] \ar[d, "f^*"] & \ldots \\
		\ldots \ar[r] & \chowwitt^{i-c}(X_Z, f^*\iota^* \linebundle \otimes \det \Omega_{X_Z/X}) \ar[r]& \chowwitt^i(X, f^* \linebundle) \ar[r] & \chowwitt^i(X_U, f^*j^*\linebundle) \ar[r, "\partial"] & \ldots
	\end{tikzcd} \]
\end{lemma}

\begin{proof}
	The argument for Witt cohomology from \cite[Lemma 3.5]{HMW} translates directly to this setting.
\end{proof}

\begin{proof}
Fasel \cite[Section 2.2]{fasel-lectures} proves that under these conditions, there is a long exact sequence
\[ \ldots \to \chowwitt^{i-\codim_X(Z)}(Z, i^* \linebundle \otimes \det i^* \Omega_{X/k}^\vee \otimes \det \Omega_{Z/k}) \xrightarrow{i_*} \chowwitt^i(X, \linebundle) \xrightarrow{j^*} \chowwitt^i(U, j^*\linebundle) \ldots \]
The pushforward $(s_0)_*$ exists because $s_0$ is a closed immersion and thus proper.
Substituting line bundles as in \cref{eq:pushworward-line-bundle-iso} yields the statement.
\end{proof}

\begin{proposition}[Square Periodicity]\label{chowwitt-quadratic-periodicity}
	Let $X$ be a smooth scheme over $k$ and $\linebundle$ and $\linebundleb$ line bundles over $X$.
	For all $i$ there is a canonical isomorphism
	\[ \sqperiod_\linebundleb \colon \chowwitt^i(X, \linebundle) \to \chowwitt^i(X, \linebundle \otimes \linebundleb^{\otimes 2}) \, . \]
	When the line bundle $\linebundleb$ is clear from context, we will sometimes drop it from the notation and simply write $\sqperiod$.
\end{proposition}

This will be crucial for defining a total Chow-Witt ring in the next section.
The statement seems to be well known, but nevertheless the author was not able to find a written account for Chow-Witt groups, so we will give a short argument here.

\begin{proof}
	Consider again the pullback square of \cref{eq:chain-complex-fiber-product}.
	The complexes with coefficients in $\KM_*(k)$ and $\ibar^*(k) \cong \KM_*(k)/2$ are independent of twist by construction.
	For the complex with coefficients in $I^*(k)$, the square periodicity map comes from multiplication with the quadratic form $[\linebundleb^{\otimes 2} \xrightarrow{\id} \trivialbundle_X \otimes \linebundleb^{\otimes 2}] \in H^0(X, I^0, \linebundleb^{\otimes 2}) = W^0(X, \linebundleb^{\otimes 2})$.
	A proof that this is indeed an isomorphism can be found for example in \cite{balmer-calmes}.
\end{proof}

From this constructions the following is immediate.
\begin{lemma}
The following diagram commutes.
\[ \begin{tikzcd}
	\chowwitt^i(X, \linebundle) \ar[r, "\sqperiod"] \ar[dr, "\forgetful"] & \chowwitt^i(X, \linebundle \otimes \linebundleb^{\otimes 2}) \ar[d, "\forgetful"] \\
	& \chow^i(X)
\end{tikzcd} \]
\end{lemma}

\begin{proposition}[Homotopy Invariance] \label{prop:chowwitt-homotopy-invariance}
	Let $X$ be a smooth scheme and $\pi \colon E \to X$ a vector bundle.
	Then the pullback
	\[ \pi^* \colon \chowwitt^\tot (X) \to \chowwitt^\tot(E) \]
	is a ring isomorphism.
\end{proposition}

\begin{proof}
	See \cite[Corollaire 11.3.2]{fasel-groupes}.
\end{proof}

\section{Ring Structure} \label{sec:ring-structure}

The first step towards defining the multiplicative structure for the Chow-Witt ring is the exterior product
\[ \mu \colon H^i(X, \KMW_j, \linebundle) \otimes H^{i^\prime}(Y, \KMW_{j^\prime}, \linebundle^\prime) \to H^{i + i^\prime}(X \times Y, \KMW_{j + j^\prime}, \pr_1^*\linebundle \otimes \pr_2^*\linebundle^\prime) \]
for $X$ and $Y$ smooth schemes and $\linebundle$, $\linebundle^\prime$ line bundles over $X$ and $Y$, respectively.
For this consider the fiber product of chain complexes from \cref{eq:chain-complex-fiber-product}.
The external products on $C^*(X, \KM_*, j)$ as defined by \cite{rost96} and $C^*(X, I^*, j, -)$ as defined by \cite{gille-nenashev} induce a product on their fiber product $C^*(X, \KMW_*, j, -)$.
Fasel \cite[4.12]{fasel-chowwitt} proves that this product is well-defined and descends to a product on cohomology.

\begin{definition}
	Let $X$ be a smooth scheme over $k$ and $\linebundle$, $\linebundleb$ line bundles over $X$.
	Define the product on Milnor-Witt cohomology
	\[ H^i(X, \KMW_j, \linebundle) \otimes H^{i^\prime}(X, \KMW_{j^\prime}, \linebundleb) \to H^{i + i^\prime}(X, \KMW_{j + j^\prime}, \linebundle \otimes \linebundleb) \]
	as the composition of the exterior product $\mu$ with the pullback along the diagonal map $\Delta \colon X \to X \times X$.
\end{definition}
This product is associative by \cite[Prop.\ 6.6]{fasel-chowwitt}.
It is neither commutative nor anti-commutative, but satisfies the following property due to \cite[Prop.\ 2.5]{hornbostel-wendt}.
\begin{lemma} \label{lem:chowwitt-graded-commutative}
	Let $X$ be a smooth scheme over $k$ and $\linebundle$ and $\linebundleb$ line bundles over $X$.
	Let $\alpha \in \chowwitt^i(X, \linebundle)$ and $\beta \in \chowwitt^j(X, \linebundleb)$.
	Then
	\[ \beta \alpha = \langle -1 \rangle ^{ij} \alpha \beta \, . \]
\end{lemma}
It immediately follows that all classes in degree $0$ are central in the Chow-Witt ring.

In \cref{chowwitt-quadratic-periodicity} we found that Chow-Witt groups are invariant under squares of line bundles.
So when adding up all Chow-Witt groups to form a ring we don't want to include all of these isomorphic groups, but only sum over $\Pic(X)/2$.
For this we need to ensure the following compatibility.

\begin{lemma}%
	\label{square-periodicity-compatibility}
	The square periodicity isomorphism is compatible with multiplication in the sense that the following diagram commutes for all line bundles $\linebundle$, $\linebundle^\prime$, $\linebundleb$, $\linebundleb^\prime$ on $X$.
	\[ \begin{tikzcd}
		\chowwitt^i(X, \linebundle) \times \chowwitt^j (X, \linebundle^\prime) \ar[r, "\mathrm{mult}"] \ar[d, "\sqperiod_\linebundleb \times \sqperiod_{\linebundleb^\prime}"] & \chowwitt^{i+j} (X, \linebundle \otimes \linebundle^\prime) \ar[d, "\sqperiod_{\linebundleb\otimes \linebundleb^\prime}"] \\
		\chowwitt^i(X, \linebundle \otimes \linebundleb^{\otimes 2}) \times \chowwitt^j(X, \linebundle^\prime \otimes (\linebundleb^\prime) ^{\otimes 2}) \ar[r, "\mathrm{mult}"] & \chowwitt^{i+j} (X, \linebundle \otimes \linebundle^\prime \otimes (\linebundleb \otimes \linebundleb^\prime)^{\otimes 2})
	\end{tikzcd} \]	
\end{lemma}

\begin{proof}
	As explained in the proof of \cref{chowwitt-quadratic-periodicity}, considering $\KMW_*$ as the fiber product $I^* \times _{\KM_*/2} \KM_*$ exhibits the square periodicity isomorphism as the fiber product of the identity on Chow groups and the square periodicity isomorphism for $I^j$-cohomology.
	The latter is given by multiplication with the quadratic form $[\linebundleb \to \linebundleb^\vee \otimes \linebundleb^{\otimes 2}] \in H^0(X, I^0, \linebundleb^{\otimes 2})$ by \cite{balmer-calmes}.
	To prove the \namecref{square-periodicity-compatibility} we have to show that for $\alpha \in \chowwitt^i(X, \linebundle)$ and $\beta \in \chowwitt^j(X, \linebundle^\prime)$:
	\begin{align*} 
		\sqperiod_\linebundleb(\alpha) \cdot \sqperiod_{\linebundleb^\prime}(\beta) &= \alpha \cdot [\linebundleb \to \linebundleb^\vee \otimes \linebundleb^{\otimes 2}] \cdot \beta \cdot [\linebundleb^\prime \to (\linebundleb^\prime)^\vee \otimes (\linebundleb^\prime)^{\otimes 2}] \\ 
		&= \alpha \beta \cdot [\linebundleb \to \linebundleb^\vee \otimes \linebundleb^{\otimes 2}] \cdot [\linebundleb^\prime \to (\linebundleb^\prime)^\vee \otimes (\linebundleb^\prime)^{\otimes 2}] \\
		&= \alpha\beta \cdot [\linebundleb \otimes \linebundleb^\prime \to (\linebundleb \otimes \linebundleb^\prime)^\vee \otimes (\linebundleb \otimes \linebundleb^\prime)^{\otimes 2}] \\ 
		&= \sqperiod_{\linebundleb\otimes \linebundleb^\prime} (\alpha \beta)
	\end{align*}
	The second equality holds because the class $[\linebundleb \to \linebundleb^\vee \otimes \linebundleb^{\otimes 2}]$ lives in degree $0$ and is thus in the center of the Witt cohomology ring as per \cref{lem:chowwitt-graded-commutative}.
	The third equality follows because the product in the Witt ring $\witt(k)$ is given by tensor product of quadratic forms.
\end{proof}

A very detailed account of all technicalities concerning the choices of line bundles representing classes in $\Pic(X)/2$ and the necessary compatibilities can be found in \cite{balmer-calmes}.

\begin{definition} \label{total-chow-witt-ring}
	We define the total Chow-Witt ring of a smooth scheme $X$ as the graded ring
	\[ \chowwitt^\tot(X) \coloneqq \bigoplus_{i \in \Integer} \bigoplus_{\linebundle \in \Pic(X)/2} \chowwitt^i(X, \linebundle) \]
	with the multiplication detailed above.
\end{definition}
We will use the notation $\chowwitt^\tot$ for the $\Integer \times (\Pic(X)/2)$-graded ring, as opposed to $\chowwitt^*$ which in the literature is commonly used to denote the $\Integer$-graded ring of Chow-Witt groups with trivial twist.
The computations in this work are all concerned with the total Chow-Witt ring $\chowwitt^\tot$.

\begin{proposition}
	Let $X$, $Y$ be smooth schemes over $k$ and $f \colon X \to Y$ a morphism.
	Then the pullback map
	\[ f^* \colon \chowwitt^\tot(Y) \to \chowwitt^\tot(X) \]
	is a ring homomorphism.
\end{proposition}

\begin{proof}
	Previously in \cref{prop:general-pullback} we have constructed a pullback homomorphism of groups
	\[ f^* \colon \chowwitt^i(Y, \linebundle) \to \chowwitt^i(X, f^* \linebundle) \]
	for each $i \in \Integer$ and $\linebundle$ a line bundle over $Y$.
	These assemble into a map
	\[ \bigoplus_i \bigoplus_{\linebundle \in \Pic(Y)} \chowwitt^i(Y, \linebundle) \to \bigoplus_i \bigoplus_{\linebundle \in \Pic(Y)} \chowwitt^i(X, f^* \linebundle)  \]
	which is shown in \cite[7.2]{fasel-chowwitt} to be a ring homomorphism.
	Further $f$ induces a pullback map $f^* \colon \Pic(Y) \to \Pic(X)$ \textemdash which is in general neither injective nor surjective \textemdash and this in turn induces 
	\[ \bigoplus_i \bigoplus_{\linebundle \in \Pic(Y)} \chowwitt^i(X, f^* \linebundle) \to \bigoplus_i \bigoplus_{\linebundleb \in \Pic(X)} \chowwitt^i(X, \linebundleb) \, . \]
	Here the multiplicative structure on the left hand side is inherited from that of the right hand side and hence it is easy to check that this is a ring homomorphism as well.
	Now to get to the total Chow-Witt rings indexed over the $\modulo 2$-Picard group, consider the diagram
	\[ \begin{tikzcd}
		\bigoplus_{i} \bigoplus_{\linebundle \in \Pic(Y)} \chowwitt^i(Y, \linebundle) \ar[r, "f^*"] \ar[d] & \bigoplus_{i} \bigoplus_{\linebundleb \in \Pic(X)} \chowwitt^i(X, \linebundleb) \ar[d] \\
		\bigoplus_{i} \bigoplus_{\linebundle \in \Pic(Y)/2} \chowwitt^i(Y, \linebundle) \ar[r, dashed, "f^*"] & \bigoplus_{i} \bigoplus_{\linebundleb \in \Pic(X)/2} \chowwitt^i(X, \linebundleb) 
	\end{tikzcd} \]
	where the vertical arrows are the quotient maps dividing out all square periodicity isomorphisms.
	The bottom horizontal map $f^*$ then emerges from the universal property of quotients and is automatically a ring homomorphism.
\end{proof}

There are also ring structures on Chow groups, $I^j$-cohomology and $\ibar^j$-cohomology defined in an analogous way as composition of an exterior product and the pullback along the diagonal map.
The reduction maps 
\begin{align*} 
	\modulo \hyperbolic_\linebundle &\colon \chowwitt^i(X, \linebundle) \to H^i(X, I^i, \linebundle) \\
	\forgetful &\colon \chowwitt^i(X, \linebundle) \to \chow^i(X) \\
	\reduction &\colon H^i(X, I^i, \linebundle) \to H^i(X, \ibar^i) 
\end{align*}
are all $\Integer$-graded ring homomorphisms with respect to these products,
essentially because they are induced by graded ring homomorphisms
\begin{align*}
	\KMW_*(k) &\to \KMW_*(k)/\eta = \KM_*(k) \\
	\KMW_*(k) &\to \KMW_*(k)/\hyperbolic = I^*(k) \\ 
	I^*(k) &\to I^*(k)/\eta = I^*(k)/I^{*+1}(k) = \ibar^*(k) \, .
\end{align*}
In the case of $\forgetful$ and $\reduction$, these ring homomorphisms forget the $\Pic(X)/2$-grading on the domain and regard it only as a $\Integer$-graded ring, with $i$-th degree
\[ \bigoplus_{\linebundle \in \Pic(X)/2} \chowwitt^i(X, \linebundle) \to \chow^i(X) \quad \text{respectively} \: \bigoplus_{\linebundle \in \Pic(X)/2} H^i(X, I^i, \linebundle) \to H^i(X, \ibar^i) \, . \] 

\begin{definition}
	Let $X$ be a smooth scheme.
	We denote the Chow ring of $X$ by
	\begin{align*}
		\chow^*(X) &= \bigoplus_{i \in \Integer} \chow^i(X) \\
		\shortintertext{and the total $I^j$-cohomology ring by}
		H^\tot(X, I^*) &= \bigoplus_{i \in \Integer} \bigoplus_{\linebundle \in \Pic(X)/2} H^i(X, I^i, \linebundle)
	\end{align*}
	with the multiplicative structures detailed above.
\end{definition}

\section{Equivariant Chow-Witt Rings} \label{sec:equivariant-chowwitt}

In this section we define the Chow-Witt ring $\chowwitt^\tot(BG)$ of the classifying space $BG$ of an algebraic group $G$ following Totaro \cite{totaro98}.
Such a classifying space always exists in Morel-Voevodsky's $\affine^1$-homotopy category, but not in the category of schemes.
It is however possible to approximate $BG$ by certain schemes and these suffice (and are in fact quite useful) to compute the Chow-Witt ring.

\begin{proposition}[Totaro, Asok-Fasel] \label{lem:admissible-gadgets}
	Let $G$ be a linear algebraic group over $k$ and $s$ a natural number. 
	Let $V$ be a finite-dimensional faithful $G$-representation over $k$ such that $G$ acts freely outside a $G$-invariant closed subset $S \subseteq V$ of codimension $\geq s$, and the quotient $(V \setminus S)/G$ exists as a scheme over $k$. 
	Then $\Pic((V\setminus S)/G)$ is independent of the choices of $V$ and $S$ if $s \geq 3$,
	the Chow group $\chow^i((V \setminus S)/G)$ is independent of $V$ and $S$,
	and for a line bundle $\linebundle \in \Pic((V\setminus S)/G)$ and $i < s-1$, the Chow-Witt group $\chowwitt^i((V\setminus S)/G, \linebundle)$ is independent of $V$ and $S$.
	We say that $(V \setminus S)/G$ is an approximation for $BG$ in codimension $< s-1$.
\end{proposition}

In the $\affine^1$-homotopy category, the colimit over a family $\{ (V_i \setminus S_i)/G \}_{i \in \Natural}$ as in the proposition with $\codim S_i \geq i$ is in fact homotopy equivalent to the classifying space $B_{\mathrm{\acute{e}t}}G$ \cite[Prop.\ 4.2.6]{morel-voevodsky}.
The computations in this work, however, will all take place in the category of schemes.
The above result is due to \cite[1.1]{totaro98} for Chow groups and was observed by \cite[3.3]{asok-fasel16} and \cite[3.1]{hornbostel-wendt} to transfer to Chow-Witt groups.

\begin{proof}
	For the statement about the Picard group, note that according to \cite[1.30]{eisenbud-harris} it is isomorphic to the first Chow group $\chow^1((V\setminus S)/G)$
	which is independent of $V$ and $S$ by \cite[2.5]{totaro-bluebook}.
	
	Fix a (finite-dimensional faithful) representation $V$ and let $S \subseteq S^\prime \subseteq V$ be two subsets satisfying the conditions.
	Consider the decomposition
	\[ (S^\prime \setminus S)/G \xhookrightarrow{i} (V \setminus S/G) \xhookleftarrow{j} (V \setminus  S^\prime)/G \, .  \]
	The subset $S/G \subseteq S^\prime/G$ being closed, it follows from excision \cite[Lemma 2.13]{asok-fasel16} that $j^*$ is an isomorphism on Chow-Witt groups in codimension $< s-1$.
	In case $S$ and $S^\prime$ are smooth this also follows from the localization sequence
	\begin{multline*} 
		\ldots \to \chowwitt^{i-r}((S^\prime \setminus S)/G, i^* \linebundle \otimes \det \Omega_{\left( (S^\prime \setminus S)/G \right) / \left( (V \setminus S)/G \right) }) \\ 
		\xrightarrow{i_*} \chowwitt^i((V \setminus S)/G, \linebundle) \xrightarrow{j^*} \chowwitt^i((V \setminus S^\prime)/G, j^* \linebundle) \to \ldots  
	\end{multline*}
	where $r$ is the codimension of $(S^\prime \setminus S)/G$ in $(V \setminus S)/G$, since Chow-Witt groups (and, in fact, all Milnor-Witt and similar cohomology groups) vanish in negative codimension.
	For two closed subsets $S \nsubseteq S^\prime$ apply the same argument to $S \subseteq S \cup S^\prime$ and $S^\prime \subseteq S \cup S^\prime$.
	
	Now let $V$ and $W$ be two representations with subsets $S \subseteq V$ and $T \subseteq W$ satisfying the conditions, and let $s$ be a mutual lower bound for the codimensions of these subsets.
	Consider the vector bundles $(V \times W) \setminus (S \times W)/G \to (V \setminus S)/G$ and $(V \times W) \setminus (V \times T)/G \to (W \setminus T)/G$.
	Independence of $S$ as proven above shows that the two total spaces have isomorphic Chow-Witt groups in degrees $< s-1$,
	and homotopy invariance shows that the same is true for the base spaces.
\end{proof}

\begin{example}\label{ex:admissible-gadgets}
	Consider the group $\Gm$, acting by multiplication on $V = \affine^{r+1}$ for some $r \geq 1$.
	This action is free outside of $S = \{ 0 \}$ which has codimension $r+1$.
	The quotient $(\affine^{r-1} \setminus \{ 0\})/\Gm \cong \projective^r$ exists as a scheme.
	Thus $\projective^r$ is an approximation for $B\Gm$ in codimension $< r$.
	
	Fasel \cite[11.8]{fasel-projbundle} computes
	\begin{align*}
		\chowwitt^i(\projective^r, \trivialbundle) &\cong \begin{cases*}
			\GW(k) & $i=0$, or $i=r$ and $r$ odd \\
			\Integer & $0 < i \leq r$ and $i$ even \\
			2 \Integer & $0 < i < r$ and $i$ odd 
		\end{cases*}\\
		\chowwitt^i(\projective^r, \trivialbundle_{\projective^r}(1)) &\cong \begin{cases*}
			2 \Integer & $0 \leq i < r$ and $i$ even \\
			\Integer & $0 \leq i \leq r$ and $i$ odd \\
			\GW(k) & $i=r$ and $r$ even.
		\end{cases*}\\
		\shortintertext{Thus we find}
		\chowwitt^i(B\Gm, \trivialbundle_{B\Gm}) &\cong \begin{cases*}
			\GW(k) & $i=0$ \\
			\Integer & $i > 0$ and $i$ even \\
			2 \Integer & $i > 0$ and $i$ odd \\
		\end{cases*}\\
		\chowwitt^i(B\Gm, \trivialbundle_{B\Gm}(1)) &\cong \begin{cases*}
			2\Integer & $i \geq 0$ and $i$ even \\
			\Integer & $i \geq 0$ and $i$ odd
		\end{cases*}
	\end{align*}
	The factor $2$ indicated that the image of the respective group under reduction $\forgetful$ to the Chow group is generated by $2$.
\end{example}
\begin{example} \label{ex:chow-ring-of-Bmun}
	Denote by $V_a$ the line bundle over $\Spec(k)$ associated to the $1$-dimension representation of $\Gm$ given by $\lambda.v = \lambda^n \cdot v$.
	Then $(V_n \times (V_1^{r+1} \setminus \{0\}))/\Gm$ is a line bundle over $(V_1^{r+1} \setminus \{0\})/\Gm \cong \projective^r$ and the complement of its zero section is an approximation for $B\mu_n$ in codimension $<r$, as will be explained in \cref{sec:model-for-Bmun}.
	This approximation was already used by \cite{totaro-bluebook} and \cite{brosnan-steenrod} to compute the Chow groups of $B\mu_n$ and then adapted by \cite{diLorenzo-Mantovani} to compute the Chow-Witt groups of $B\mu_n$ if $n$ is even.
	For both computations, one can consider the localization sequence associated to the decomposition 
	\[ \projective^r \xrightarrow{s_0} (V_n \times (V_1^{r+1} \setminus \{0\}))/\Gm \hookleftarrow \left( \left(V_n \times (V_1^{r+1} \setminus \{0\})\right)/\Gm \right) \setminus s_0(\projective^r) \, . \]
	For Chow groups it then easily follows that
	\[ \chow^i(B\mu_n) \cong \chow^i(B\Gm)/n \]
	and one obtains 
	\[ \chow^*(B\mu_n) \cong \Integer[\chern]/n \cdot \chern \]
	where $\chern$ corresponds to the first Chern class of the tautological bundle $\trivialbundle_{B\mu_n}(-1)$.
	For Chow-Witt groups the sequence will turn out to be more complicated.
	This will be the subject of \cref{sec:chowwitt-groups-Bmun}.
\end{example}

The existence of such $V$ and $S$ for a given linear algebraic group $G$ and codimension $s$ is guaranteed by \cite[Remark 1.4]{totaro98}.
This allows for the following definition.

\begin{definition}\label{def:chowwitt-of-classifying-space}
	Let $G$ be a linear algebraic group over $k$.
	Set
	\begin{align*}
		\Pic(BG) & \cong \Pic((V\setminus S)/G) \\
		\intertext{for $V$ and $S$ as in \cref{lem:admissible-gadgets} with codimension of $S$ greater than $2$, and for $\linebundle \in \Pic(BG)$ set}
		\chowwitt^i(BG, \linebundle) & \cong \chowwitt^i((V \setminus S)/G, \linebundle)
	\end{align*}	
	for $V$ and $S$ as in \cref{lem:admissible-gadgets} with codimension of $S$ greater than $i+1$.
	The product of two classes $\alpha \in \chowwitt^i(B\Gm, \linebundle)$, $\beta \in \chowwitt^j(B\Gm, \linebundleb)$
	is defined as their product in an approximation of $BG$ in codimensions $\geq i+j$.
	This is well-defined because the map $j^*$ in the proof of \cref{lem:admissible-gadgets} is a ring homomorphism.
	The ring obtained in this way is the total Chow-Witt ring $\chowwitt^\tot(BG)$.
	The Chow ring $\chow^*(BG)$ is defined analogously.
\end{definition}
The product obtained this way is well-defined because the map $j^*$ in the localization sequence in the proof of \cref{lem:admissible-gadgets} is a ring homomorphism.

The same argument as for \cref{lem:admissible-gadgets} proves the following two statements.
\begin{corollary}\label{lem:admissible-gadgets-nondiag}
	Let $G$, $V$ and $S$ be as in \cref{lem:admissible-gadgets}.
	Then for $i < s-1$ and all $j$, the Milnor-Witt cohomology $H^i ((V\setminus S)/G, \KMW_j, \linebundle)$ is independent of the choice of $V$ and $S$.
\end{corollary}

\begin{corollary}\label{lem:admissible-gadgets-Ij}
	Let $G$, $V$ and $S$ be as in \cref{lem:admissible-gadgets}.
	Then for $i < s-1$ and all $j$, the $I^j$-cohomology $H^i ((V\setminus S)/G, I^j, \linebundle)$ is independent of the choice of $V$ and $S$.
\end{corollary}
This allows for a definition of non-diagonal Milnor-Witt and $I^j$-cohomology groups of a classifying space completely analogous to \cref{def:chowwitt-of-classifying-space}.

\begin{corollary}\label{lem:admissible-gadgets-products}
	Let $G$ and $H$ be linear algebraic groups over $k$.
	Let $V$ be a finite-dimensional faithful $G$-representation with subset $S \subseteq V$ of codimension $\geq s$ and $W$ a finite-dimensional faithful $H$-representation with subset $T \subseteq W$ of codimension $\geq t$ as in \cref{lem:admissible-gadgets}.
	Then $\chowwitt^i((V \setminus S)/G \times (W\setminus T)/H, \linebundle)$ for $i < \min(s,t)-1$ is independent of $V$, $S$, $W$ and $T$.
\end{corollary}

\begin{proof}
	Consider $V \times W$ as the obvious representation of the group $G \times H$.
	Then $G \times H$ acts freely outside of $(S \times W) \cup (V \times T)$ and the latter has codimension $\geq \min(s,t)$ in $V \times W$.
	Thus one can apply \cref{lem:admissible-gadgets}.
\end{proof}

\section{Euler Classes}

\begin{definition} \label{def:euler-class}
Let $X$ be a smooth scheme, $\pi \colon \vectorbundle \to X$ a vector bundle of rank $r$ with zero section $s_0$, $\linebundle$ a line bundle over $X$ and $i$ an integer.
The Euler map of $\pi$ is defined as the composition
\[ \chowwitt^{i-r}(X, \linebundle \otimes \omega_{\vectorbundle/X})) \xrightarrow{(s_0)_*} \chowwitt^i(\vectorbundle, \pi^* \linebundle) \xrightarrow{(\pi^*)^{-1}} \chowwitt^i(X, \linebundle) \]
and the Euler class $\eulercw(\vectorbundle)$ or $\eulercw(\pi) \in \chowwitt^r(X, \omega_{\vectorbundle/X}^\vee)$ is the image of $1 \in \chowwitt^0(X, \trivialbundle)$ under this map.
The same definition can be made for $I^j$-cohomology (also called Euler class $\euleri$), Chow groups (called top Chern class $\chern_r$) and mod $2$-Chow groups (called top Stiefel-Whitney class $\sw_r$).
\end{definition}
Since in this work we will only ever deal with top Chern and Stiefel-Whitney classes, we will omit the index $r$ from the notation.
As explained in \cite[Section 13.2]{fasel-groupes} there is an isomorphism of vector bundles $\det( \Omega_{E/X}) \cong \pi^* \det( E^\vee)$.
Note that the line bundles $\det E$ and $\det E^\vee$ differ by a square and thus the Chow-Witt groups twisted by these are isomorphic.
This justifies considering $\eulercw(E)$ to live in $\chowwitt^r(X, \det E^\vee)$ even though technically it lives in $\chowwitt^r(X, \det E)$, which in some cases makes the notation slightly more readable (e.g.\ writing $\trivialbundle(1)$ instead of $\trivialbundle(-1)$).
As explained after \cref{prop:proper-pushforward,prop:flat-pullback}, the Euler map is compatible with the maps $\forgetful$, $\reduction$, $\modulo \hyperbolic_\linebundle$ and $\modulo 2$ introduced in the diagram \cref{eq:key-diagram}, that is, $\forgetful (\eulercw(\vectorbundle)) = \chern(\vectorbundle)$, $\reduction(\euleri(\vectorbundle)) = \sw(\vectorbundle)$ and so on. 
Further the hyperbolic map $\hyperbolic_{\det \Omega_{\vectorbundle/X}}$ sends $ \chern(\vectorbundle)$ to $\hyperbolic \eulercw(\vectorbundle)$: 
The hyperbolic map commutes with the Euler map, and since the latter is $\GW(k)$-linear it sends $\hyperbolic_\trivialbundle(1) = \hyperbolic = \hyperbolic \cdot 1$ to $\hyperbolic \cdot \eulercw(\vectorbundle)$.

We will prove a formula for the Euler class of the tensor product of certain line bundles, following \cite[Theorem 10.1(2)]{levine-enumerative}.

\begin{proposition} \label{prop:euler-class-product}
	Let $X$ be a smooth scheme over $k$ and $\linebundle, \mathcal{M}$ line bundles over $X$.
	Then
	\[ \eulercw(\linebundle \otimes \linebundleb^{\otimes 2}) = \eulercw(\linebundle) + \hyperbolic_{\linebundle}(\chern(\linebundleb)) \]
	where $\hyperbolic_{\linebundle^\vee} \colon \chow^1(X) \to \chowwitt^1(X, \linebundle)$ is the hyperbolic map.
	In particular:
	\[ \eulercw(\linebundle^{\otimes n}) = \begin{cases*}
		\frac{n}{2} \cdot \hyperbolic_{\trivialbundle}(\chern(\linebundle)) & $n$ even \\
		\frac{n}{2} \hyperbolic \cdot \eulercw(\linebundle) & $n$ odd 
	\end{cases*} \]
\end{proposition}
\begin{proof}
	Let us first recall where all the occurring classes live.
	The Euler class $\eulercw(\linebundle)$ is an element of $\chowwitt^1(X, \linebundle)$, and likewise $\eulercw(\linebundle \otimes \mathcal{M}^{\otimes 2}) \in \chowwitt^1(X, \linebundle \otimes \mathcal{M}^{\otimes 2})$ which by quadratic periodicity (\cref{chowwitt-quadratic-periodicity}) is isomorphic to $\chowwitt^1(X, \linebundle)$. 
	The Chern class $\chern(\mathcal{M})$ lives in $\chow^1(X)$.
	
	We start by proving the statement for the classifying space for line bundles $B\Gm = \projective^\infty$ together with the tautological bundle $\linebundle = \mathcal{M} = \trivialbundle(-1)$.
	$B\Gm$ is not a scheme but a motivic space, but according to \cref{lem:admissible-gadgets} one can understand its Chow-Witt groups and Euler classes through $\projective^r$ for $r$ sufficiently large.
	Denote by $\trivialbundle(n,m)$ the line bundle over $B\Gm \times B\Gm$ that is the tensor product of the pullback of $\trivialbundle_{B\Gm}(n)$ on the first factor and the pullback of $\trivialbundle_{B\Gm}(m)$ on the second factor.
	We want to determine the Euler class of the bundle $\linebundle \otimes \mathcal{M}^2 = \trivialbundle(-1,-2)$ over $B\Gm \times B\Gm$.
	Consider the reduction map $\forgetful$ to $\chow^1(B\Gm \times B\Gm)$ which maps the Euler class of a line bundle to its first Chern class.
	The first Chern class is additive, thus $\forgetful \eulercw (\linebundle \otimes \linebundleb^{\otimes 2}) =\chern(\linebundle \otimes \mathcal{M}^{\otimes 2}) = \chern(\linebundle) + 2\chern(\mathcal{M})$.
	
	Under the same map $\eulercw(\trivialbundle(-1, 0)) \in \chowwitt^1(B\Gm \times B\Gm, \trivialbundle(-1,0))$ is sent to $\chern(\trivialbundle(0,-1))$, and $\hyperbolic_{\trivialbundle(-1,0)}(\chern(\trivialbundle(0,-1)))$ is sent to $2\chern(\trivialbundle(0,-1))$.
	The reduction map is injective in this case, since its kernel is the image of $H^1(B\Gm \times B\Gm, I^1, \trivialbundle(-1,-2))$ in the long exact sequence coming from the sequence of sheaves $I^{j+1} \to \KMW_{j} \to \KM_j$ as stated in \cref{eq:KMW-decomposition-ses}, and that group will be shown to vanish in \cref{thm:Ij-of-PxP} (which does not rely on this \namecref{prop:euler-class-product}).
	From this it follows that $\eulercw(\trivialbundle(-1,-2)) = \eulercw(\trivialbundle(-1,0)) + \hyperbolic_{\trivialbundle(-1,0)}(\chern(\trivialbundle(0,-1)))$.
	
	Now for the general statement.
	Let $\linebundle$ and $\mathcal{M}$ be two line bundles over $X$.
	According to \cite[Prop.\ 4.3.8]{morel-voevodsky} there exist classifying maps $f,g \colon X \to B\Gm$ so that $\linebundle \cong f^*\trivialbundle(-1)$ and $\linebundleb \cong g^* \trivialbundle(-1)$ and therefore $\pr_1^* \linebundle \otimes \pr_2^*\mathcal{M}^{\otimes 2} \cong (f,g)^* \trivialbundle(-1,-2)$ over $B\Gm \times B\Gm$.
	Thus we compute:
	\begin{align*}
		\eulercw(\linebundle \otimes \mathcal{M}^{\otimes 2}) &= \eulercw((f,g)^* \trivialbundle(-1,-2)) \\
		&= (f,g)^* \eulercw(\trivialbundle(-1,-2)) \\
		&= (f,g)^* \eulercw(\trivialbundle(-1,0)) + (f,g)^* \hyperbolic_{\trivialbundle(-1,0)}(\chern(\trivialbundle(0,-1))) \\
		&= (f,g)^* \eulercw(\trivialbundle{-1,0}) + \hyperbolic_{(f,g)^* \trivialbundle(-1,0)}((f,g)^*\chern(\trivialbundle(0,-1))) \\
		&= \eulercw(\linebundle) + \hyperbolic_{\linebundle}(\chern(\mathcal{M})) 
	\end{align*}
	The statement about $\linebundle^{\otimes n}$ can be deduced inductively, inserting $\hyperbolic_{\linebundle}(\chern \linebundle) = \hyperbolic \eulercw(\linebundle)$ in the odd case.
\end{proof}

\begin{remark}
	For arbitrary line bundles $\linebundle$, $\linebundleb$ over a scheme $X$ there is no formula of this kind as explained in the third paragraph of \cite[Section 10]{levine-enumerative}:
	Consider the universal case $X = B\Gm \times B\Gm$, $\linebundle = \pr_1^*\trivialbundle_{B\Gm}(-1)$, $\linebundleb = \pr_2^* \trivialbundle_{B\Gm}(-1)$.
	Then
	\begin{align*}
		\eulercw(\trivialbundle(-1,0)) &\in \chowwitt^1(B\Gm \times B\Gm, \trivialbundle(1,0)) \\
		\eulercw(\trivialbundle(0,-1)) & \in \chowwitt^1(B\Gm \times B\Gm, \trivialbundle(0,1)) \\
		\eulercw(\trivialbundle(-1,0)\otimes \trivialbundle(0,-1)) = \eulercw(\trivialbundle(-1,-1)) & \in \chowwitt^1(B\Gm \times B\Gm, \trivialbundle(1,1)) \, .
	\end{align*}
	Thus to express $\eulercw(\trivialbundle(-1,-1))$ in terms of $\eulercw(\trivialbundle(-1,0))$ and $\eulercw(\trivialbundle(0,-1))$, one would additionally need classes in degrees $(0, \trivialbundle(1,0))$ and $(0, \trivialbundle(0,1))$.
	Levine now claims that the groups in those degrees vanish, which is not quite correct since \cite[Theorem 1.1]{wendt} does not imply $\chowwitt^0(B\Gm, \trivialbundle(1)) = 0$ as stated in \cite[Section 10, p.2220]{levine-enumerative}.
	The rest of Levine's argument, however, still stands: in \cref{prop:chowwitt-of-BGmxBGm} we prove that the subgroup
	\begin{multline*} 
	\left( \chowwitt^0(B\Gm, \trivialbundle(1,0)) \cdot \chowwitt^1(B\Gm, \trivialbundle(0,1)) \right) \\
	+ \left( \chowwitt^0(B\Gm, \trivialbundle(0,1)) \cdot \chowwitt^1(B\Gm , \trivialbundle(1,0)) \right) \subseteq \chowwitt^1(B\Gm, \trivialbundle(1,1)) 
	\end{multline*}
	does not contain $\eulercw(\linebundle \otimes \linebundleb)$ but only $2 \cdot \eulercw(\linebundle \otimes \linebundleb)$ (using the symbols of \cref{prop:chowwitt-of-BGmxBGm}: the subgroup on the left is generated by $\hyperbolica \eulercwb$ and $\hyperbolicb \eulercwa$, and the only relation in degree $(1, \trivialbundle(1,1))$ is $\hyperbolica \eulercwb + \hyperbolicb \eulercwa - \hyperbolic \eulercwc$ from \cref{eq:relations-BGmxBGm-3}, so $\eulercwc$ cannot be expressed as a linear combination of $\hyperbolica \eulercwb$ and $\hyperbolicb \eulercwa$).
\end{remark}

\chapter{The Chow-Witt Ring of \texorpdfstring{$B\mu_n$}{Bµn}} %
\label{sec:Chow-Witt-of-Bmun}

Throughout this chapter, let $n$ be a positive natural number and the base field $k$ be a perfect field with characteristic coprime to $2$ and $n$.
Denote by $\mu_n$ the group of roots of unity, a linear algebraic group given as a scheme by $\Spec(k[x]/(x^n-1))$.
The main result in this chapter is to compute the total Chow-Witt ring
\[ \chowwitt^\tot(B\mu_n) = \bigoplus_{i \in \Integer} \bigoplus_{\linebundle \in \Pic(B\mu_n)/2} \chowwitt^i(B\mu_n, \linebundle) \]
for odd $n$.
The strategy closely follows that for even $n$ due to \cite{diLorenzo-Mantovani}.
The model for $B\mu_n$ described here was originally constructed in \cite[Lemma 6.3]{voevodsky-power-operations} and used in \cite[Theorem 7.1]{brosnan-steenrod} and \cite[Theorem 2.10]{totaro-bluebook} to compute its Chow groups.

\section{A Model for \texorpdfstring{$B\mu_n$}{Bµn}}\label{sec:model-for-Bmun}

We construct an approximation of $B\mu_n$ in the sense of \cref{lem:admissible-gadgets}.

For $a \in \Integer$ denote by $V_a$ the vector bundle over $\Spec(k)$ associated to the $1$-dimensional representation of $\Gm$ with action given by $\lambda.v = \lambda ^a \cdot v$.
Products of such representations are always considered with diagonal action.
As explained in \cref{ex:admissible-gadgets} $(V_1^r \setminus \{ 0\})/\Gm \simeq \projective^{r-1}$ is an approximation of $B\Gm$ in codimension $< r-1$.
Over this space, consider the bundle 
\[ \Bmun{r}{n} \coloneqq \left( V_n \times (V_1^{r+1} \setminus \{ 0\})\right)/\Gm \to (V_1^{r+1} \setminus \{ 0\})/\Gm \]
given by projection on the second factor.
\Cref{lem:Ern-equals-O(n)} will show that this is a line bundle.
Note that the zero coordinate in the factor $V_n$ is the image of the zero section $s_0$ of the line bundle $(V_n \times (V_1^{r+1} \setminus \{0\}))/\Gm$.

\begin{lemma}\label{lem:Ern-equals-Vr-mod-mun}
	There is a map of $\Gm$-bundles over $\projective^r$
	\[ \left( V_1^{r+1} \setminus \{0\} \right) / \mu_n \to \left( (V_n \setminus \{0\}) \times (V_1^{r+1} \setminus \{0\}) \right) / \Gm \]
	inducing isomorphisms on Chow-Witt groups in codimension $<r-1$.
\end{lemma}

\begin{proof}
	See the proof of \cite[Theorem 7.1(i)]{brosnan-steenrod}.
\end{proof}

Since $(V_1^{r+1} \setminus \{0\}) / \mu_n$ is an approximation for $B\mu_n$ in codimension $< r-1$, this shows that $\Bmun{r}{n}\setminus s_0(\projective^r)$ is also an approximation for computing these Chow-Witt groups.
The advantage of using $\Bmun{r}{n}\setminus s_0(\projective^r)$ over $(V_1^{r+1} \setminus \{0\}) / \mu_n$ is that the former can be embedded into a line bundle over $\projective^r$, namely the twisting bundle $\trivialbundle(n)$, as the complement of the zero section as shown in the following lemma.

\begin{lemma} [{\cite[Section 4.1.4]{diLorenzo-Mantovani}}]\label{lem:Ern-equals-O(n)}
	For any $n \in \Integer$, the line bundle $\Bmun{r}{n}$ is isomorphic to $\trivialbundle(n)$ over $\projective^r$.
\end{lemma}
\begin{proof}
	We start by proving the statement for $n= -1$.
	The tautological line bundle $\trivialbundle(-1)$ is the subscheme of $\projective^r \times \affine^{r+1}$ containing those points $([X_0 : \ldots : X_r], (Y_0, \ldots, Y_r))$ for which there is a scalar $t$ such that $(Y_1, \ldots, Y_r) = t(X_0, \ldots, X_r)$.
	Denote by $L$ the pullback of $\trivialbundle(-1)$ along $V_1^{r+1}\setminus\{0\} \to V_1^{r+1}/\Gm \cong \projective^r$.
	This is the subscheme of $(V_1^{r+1} \setminus \{0\}) \times V_0^{r+1}$ containing those points $((X_0, \ldots, X_r), (Y_0, \ldots, Y_r))$ for which there is a scalar $t$ such that $(Y_1, \ldots, Y_r) = t(X_0, \ldots, X_r)$.
	This is an $\Gm$-invariant subscheme and thus inherits the $\Gm$-action from $(V_1^{r+1} \setminus \{0\}) \times V_0^{r+1}$.
	
	Now consider the map
	\begin{align*}
		\varphi \colon L & \to (V_1^{r+1} \setminus \{0\}) \times V_{-1} \\
		((X_0, \ldots, X_r),(Y_0,\ldots,Y_r)) & \mapsto (X_0, \ldots, X_r, t) \\
		= ((X_0, \ldots, X_r), (tX_0, \ldots, tX_r)) & 
	\end{align*}
	which is readily checked to be an isomorphism.
	This map is equivariant with respect to the $\Gm$-action previously defined on each side:
	\[ \varphi(\lambda.(X,Y)) = \varphi(\lambda.(X, tX)) = \varphi((\lambda X, tX)) = (X, t \cdot \lambda^{-1}) = \lambda.\varphi((X, tX)) \]
	Therefore $(V_1^{r+1} \setminus \{0\}) \times V_{-1} / \Gm \cong L/\Gm \cong \trivialbundle(-1)$.
	
	All other line bundles $\trivialbundle(n)$ can be constructed as tensor powers of $\trivialbundle(-1)$, adhering to the rule $\trivialbundle(m) \otimes \trivialbundle(n) \cong \trivialbundle(m+n)$.
	Using that tensor products commute with quotients we have
	\begin{align*} 
		&\left(\left(V_m \times (V_1^{r+1} \setminus \{0\})\right) / \Gm\right) \otimes \left(\left(V_n \times (V_1^{r+1} \setminus \{0\})\right) / \Gm\right) \\
		\cong &\left((V_m \otimes V_n) \times (V_1^{r+1} \setminus \{0\})\right) / \Gm \\
		\cong &\left(V_{m+n} \times (V_1^{r+1} \setminus \{0\})\right) / \Gm
	\end{align*}
	as line bundles over $\projective^r \cong (V_1^{r+1} \setminus \{0\}) / \Gm$ for all $m, n \in \Integer$.
	Thus the statement holds for all $n \in \Integer$.
\end{proof}

The scheme $\Bmun{r}{n}\setminus s_0(\projective^r)$ will serve as an approximation of $B\mu_n$ in the sense of \cref{lem:admissible-gadgets} throughout the next sections.
The bundle map $\pi \colon \Bmun{r}{n}\setminus s_0(\projective^r) \to \projective^r$ induces a ring homomorphism
\[ pi^* \colon \chowwitt^\tot(\projective^r) \to \chowwitt^\tot(\Bmun{r}{n}\setminus s_0(\projective^r)) \]
for all $r$ and thus also 
\[ \pi^* \colon \chowwitt^\tot(B\Gm) \to \chowwitt^\tot(B\mu_n) \, . \] 

\section{Computing the Chow-Witt Groups}\label{sec:chowwitt-groups-Bmun}

The following known result will be used in the computation.

\begin{theorem}[{\cite[Theorem 11.7]{fasel-projbundle}}] %
	\label{prop:milnorwitt-of-BGm}
	Let $k$ be a perfect field of characteristic coprime to $2$.
	Then we have the following isomorphisms of $\GW(k)$-modules.
	\begin{align*} 
		H^i(B\Gm, \KMW_j, \trivialbundle) &\cong \begin{cases*} 
			\KMW_j(k) & i=0 \\ 
			\KM_{j-i}(k) & $i \geq 2$ even \\ 
			2 \KM_{j-i}(k) & $i \geq 1$ odd \\
		\end{cases*} \\
		H^i(B\Gm, \KMW_j, \trivialbundle(1)) 	& \cong \begin{cases*}
			2\KM_{j-i}(k) & $i \geq 0$ even \\
			\KM_{j-i}(k) & $i \geq 1$ odd. \\
		\end{cases*} \\
	\end{align*}
	The factor $2$ indicates that the image of the respective group under reduction $\forgetful$ to the Chow group is generated by $2$.
\end{theorem}
The total Chow-Witt ring of $B\Gm$ has been computed in \cite[Theorem 1.1]{wendt}: 
\begin{proposition}[Wendt]%
	\label{chowwitt-ring-of-BGm}
	There is an isomorphism of $\GW(k)$-algebras
	\[ \chowwitt^\tot(B\Gm) \cong \GW(k)[\eulercwz,\hyperbolicz]/(I(k) \cdot \eulercwz, I(k) \cdot \hyperbolicz, \hyperbolicz^2 - 2h) \]
	where $\eulercwz \in \chowwitt^1(B\Gm, \trivialbundle(1))$ corresponds to the Euler class of $\trivialbundle_{B\Gm}(-1)$ and $\hyperbolicz$ corresponds to the element represented by $(0,2)$ in
	\[ \chowwitt^0(B\Gm, \trivialbundle(1)) \subseteq H^0(B\Gm, I^0, \trivialbundle(1)) \times \chow^0(B\Gm) \, .  \]
\end{proposition}
The inclusion in the last line is in fact an inclusion by \cite[Prop.\ 2.11]{hornbostel-wendt} if $\chow^0(B\Gm)$ has trivial $2$-torsion which is satisfied since by \cite[Thm 2.10]{totaro-bluebook}, $\chow^0(B\Gm) \cong \Integer$.

By \cite[Prop.\ 1.30]{eisenbud-harris}, the Picard group of $B\mu_n$ is isomorphic to $\chow^1(B\mu_n)$ which is shown in \cite[Theorem 2.10]{totaro-bluebook} to be isomorphic to $\Integer/n$.
Thus for even $n$, we need to consider two equivalence classes of line bundles in $\Pic(B\mu_n)/2$ represented by $\trivialbundle$ and $\trivialbundle(1)$ (defined as the pullbacks of the respective line bundles on $\projective^r$), and for odd $n$ only the trivial line bundle $\trivialbundle$.

For even $n$, the Chow-Witt ring has been computed in \cite[Theorem 5.3.4]{diLorenzo-Mantovani}.

\begin{definition}[di Lorenzo-Mantovani] \label{def:classu} 
	$\,$ \vspace{-3ex} \\
	\begin{enumerate}
		\item Let $X$ be a smooth scheme and $\linebundle$ a line bundle over $X$.
		Let $s$ be a global section of $\linebundle^{\otimes 2}$ with smooth and non-empty vanishing locus $D \subseteq X$,
		and denote $U \coloneqq X \setminus D$. 
		Then there is a non-degenerate quadratic form on $\linebundle \vert_U$ 
		\begin{align*}
			q \colon \linebundle \vert_U \otimes \linebundle \vert_U & \to \trivialbundle_U \\
			a \otimes b & \mapsto( a \otimes b) /s
		\end{align*}
		determining an element $q_{\mathrm{gen}} \in \chowwitt^0(U, \trivialbundle_U)$.
		This element satisfies $\hyperbolic \cdot (q_{\mathrm{gen}} - 1) = 0$ and is mapped to $\eta \otimes \bar{f}^\vee$ by the boundary map
		\[ \partial \colon \chowwitt^0(U) \to H^0(D, \KMW_{-1}, \trivialbundle_D) \]
		in the localization sequence associated to the embedding $D \subseteq X$,
		where $f$ is a local equation with vanishing set $D$ and $f^\vee$ its dual \cite[Lemma 3.3.2]{diLorenzo-Mantovani}.
		\item Let $n$ be even and consider the scheme $\Bmun{r}{n}$ constructed in the previous section with the line bundle $\trivialbundle_{\Bmun{r}{n}}(n/2)$.
		Since $\Bmun{r}{n}$ is isomorphic to $\trivialbundle(n)$ as line bundle over $\projective^r$, the line bundle
		\[ \trivialbundle_{\Bmun{r}{n}}(n/2)^{\otimes 2} \cong \trivialbundle_{\Bmun{r}{n}}(n) \cong \trivialbundle_{\projective^r}(n) \times_{\projective^r} \Bmun{r}{n} \cong \trivialbundle_{\projective^r}(n) \times_{\projective^r} \trivialbundle_{\projective^r}(n) \]
		comes equipped with a section given by the diagonal map whose vanishing locus is precisely the zero section $s_0$ of the bundle $\Bmun{r}{n} \to \projective^r$.
		Applying the construction from the first part of this definition yields an element in the zeroth Chow-Witt group and we set
		\[ \classu \coloneqq q_{\mathrm{gen}} - 1 \in \chowwitt^0(\Bmun{r}{n}\setminus s_0, \trivialbundle) \cong \chowwitt^0(B\mu_n, \trivialbundle) \, . \]
	\end{enumerate}
\end{definition}

\begin{theorem}[di Lorenzo-Mantovani]\label{chowwitt-groups-of-Bmun-even}
	There are isomorphisms of groups
	\begin{align*}
		\chowwitt ^i (B\mu_n, \trivialbundle) &\cong \begin{cases*}
			\GW(k) \oplus \witt(k) & i = 0 \\
			\Integer/2n & i $\geq$ 2 \text{even} \\
			\Integer/(\frac{n}{2}) & i \text{odd} 
		\end{cases*} \\
		\chowwitt ^i (B\mu_n, \trivialbundle(1)) & \cong
		\begin{cases*}
		\Integer & i = 0 \\
		\Integer/(\frac{n}{2}) & i $\geq$ 2 \text{even} \\
		\Integer/2n& i \text{odd.}
		\end{cases*} \\ 
	\end{align*}
	Further there is an isomorphism of $\GW(k)$-algebras
	\[ \chowwitt^\tot(B\mu_n) \cong \GW(k)[\classu, \hyperbolicz, \eulercwz]/(I(k) \hyperbolicz, I(k) \eulercwz, \hyperbolic \classu, \hyperbolicz \classu, n \hyperbolicz \eulercwz, \hyperbolicz^2 - 2 \hyperbolic, \classu^2 + 2 \classu, \classu \eulercwz - 2n\eulercwz) \]
	where $\eulercwz$ and $\hyperbolicz$ correspond to the pullbacks of the respective elements described in \cref{chowwitt-ring-of-BGm} and $\classu \in \chowwitt^0(B\mu_n, \trivialbundle)$ is as constructed in \cref{def:classu}.
\end{theorem}

We will compute the case of odd $n$ .
Following the strategy of \cite[Prop.\ 5.2.3]{diLorenzo-Mantovani} for even $n$
consider the localization sequence associated to the decomposition 
\[ \projective^r \xhookrightarrow{s_0} \Bmun{r}{n} \xhookleftarrow{\iota} \Bmun{r}{n} \setminus s_0(\projective^r) \]
which reads as follows.
\begin{multline*} 
	\ldots \to \chowwitt^{i-1}(\projective^r, s_0^* \linebundle \otimes \det \Omega_{\Bmun{r}{n}/\projective^r}) \xrightarrow{(s_0)_*} \chowwitt^i(\Bmun{r}{n}, \linebundle) \xrightarrow{\iota^*} \chowwitt^i(\Bmun{r}{n} \setminus s_0, \iota^*\linebundle) \\
	\to H^i(\projective^r, \KMW_{i-1}, s_0^* \linebundle \otimes \det \Omega_{\Bmun{r}{n}/\projective^r}) \to \ldots 
\end{multline*}
In the first term \cref{lem:Ern-equals-O(n)} implies $\det \Omega_{\Bmun{r}{n}/\projective^r} \cong \det \Omega_{\trivialbundle_{\projective^r}(n)/\projective^r}$ which in turn is isomorphic to $\trivialbundle(n)$ \cite[Example 8.20.1]{hartshorne} and thus trivial for odd $n$.
By homotopy invariance, the pullback along the bundle map $\pr \colon \Bmun{r}{n} \to \projective^r$ induces an isomorphism on Chow-Witt groups.
The composition 
\[ \chowwitt^{i-1}(\projective^r, s_0^* \linebundle \otimes \trivialbundle(n)) \xrightarrow{(s_0)_*} \chowwitt^i(\Bmun{r}{n}, \linebundle) \xleftarrow[\cong]{\pr^*} \chowwitt^i(\projective^r, \linebundle) \]
is multiplication with the Euler class of the line bundle $\Bmun{r}{n}$ by \cref{def:euler-class}.
Now using that $\projective^r$ and $\Bmun{r}{n}\setminus s_0$ are approximations for $B\Gm$ respectively $B\mu_n$ in the sense of \cref{lem:admissible-gadgets}, for $r>i$ this localization sequence is isomorphic to the following one.
\begin{multline}\label{eq:locseq-Bmun}
	\ldots \to \chowwitt^{i-1}(B\Gm, s_0^*\linebundle) \xrightarrow{\eulercw(\Bmun{r}{n}) \cdot} \chowwitt^i(B\Gm, \linebundle) \xrightarrow{(\iota \circ \pr)^*} \chowwitt^i(B\mu_n, \iota^*\linebundle) \\
	\to H^i(B\Gm, \KMW_{i-1}, s_0^* \linebundle) \to \ldots
\end{multline}
Note that $\iota \circ \pr = \pi$.

Equipped with this we can compute the group structure.

\begin{theorem} \label{chowwitt-groups-of-Bmun-odd}
	Let $n$ be odd.
	Then:
	\begin{align*}
	\chowwitt ^i (B\mu_n, \trivialbundle) &\cong \begin{cases*}
		\GW(k) & i = 0 \\
		\Integer/n & i $\geq$ 1
	\end{cases*} 
	\end{align*}
\end{theorem}

\begin{proof}
	For $i > 0$, insert \cref{prop:milnorwitt-of-BGm,chowwitt-ring-of-BGm} to see that the localization sequence \cref{eq:locseq-Bmun} evaluates to
	\begin{center}
	\begin{tikzcd} 
		\phantom{\hyperbolicz}\KM_0(k)\langle \eulercwz^{i-1} \rangle \ar[r, "\cdot \eulercw(\Bmun{r}{n})"] \ar[d, phantom, sloped, "\cong"] & \KM_0(k)\langle \eulercwz^i \rangle \ar[r] \ar[d, phantom, sloped, "\cong"] &  \chowwitt^i(B\mu_n, \trivialbundle) \ar[r] & \KM_{-1}(k) \ar[d, phantom, sloped, "\cong"] \\
		\Integer\langle \eulercwz^{i-1} \rangle \ar[r, "\cdot (-n) \eulercwz"] & \Integer\langle \eulercwz^i \rangle  & & 0 \\
	\end{tikzcd}
	\end{center}
	for $i$ even and
	\begin{center}
	\begin{tikzcd}
		2\KM_0(k)\langle \hyperbolicz\eulercwz^{i-1} \rangle \ar[r, "\cdot \eulercw(\Bmun{r}{n})"] \ar[d, phantom, sloped, "\cong"] & 2\KM_0(k)\langle \hyperbolicz\eulercwz^i \rangle \ar[r] \ar[d, phantom, sloped, "\cong"] &  \chowwitt^i(B\mu_n, \trivialbundle) \ar[r] & \KM_{-1}(k) \ar[d, phantom, sloped, "\cong"] \\
		\Integer\langle \hyperbolicz\eulercwz^{i-1} \rangle \ar[r, "\cdot (-n) \eulercwz"] & \Integer\langle \hyperbolicz\eulercwz^i \rangle  & & 0 \\
	\end{tikzcd}
	\end{center}
	for $i$ odd.
	Since $\Bmun{r}{n} \cong \trivialbundle(n)$ as line bundles over $\projective^{r}$ by \cref{lem:Ern-equals-O(n)}, its Euler class evaluates to $n/2 \hyperbolic \cdot \eulercw(\trivialbundle(1)) = -n/2 \hyperbolic \eulercwz = -n \eulercwz$ using \cref{prop:euler-class-product}.
	From this it follows that $\chowwitt^i(B\mu_n) \cong \Integer/n$ for $0 < i \leq n$.
	Denoting $\pi \colon \Bmun{r}{n} \to \projective^r$ or by abuse of notation also $\pi \colon B\mu_n \to B\Gm$ and using the symbols from \cref{chowwitt-ring-of-BGm}, one can also write
	\begin{equation} \label{eq:chowwitt-groups-of-Bmun} \chowwitt^i(B\mu_n) \cong \begin{cases*}
		\Integer/n \langle \pi^*(\eulercwz^i) \rangle & $0 < i \leq n$, $i$ even \\
		\Integer/n \langle \pi^*(\hyperbolicz \eulercwz^i) \rangle & $0 < i \leq n$, $i$ odd. \\
	\end{cases*} \end{equation}
	
	For $i=0$ we get
	\begin{center}
	\begin{tikzcd}[column sep=small]
		\chowwitt^{-1}(B\Gm, \trivialbundle(n)) \ar[r] \ar[d, phantom, sloped, "\cong"] & \chowwitt^0(B\Gm, \trivialbundle) \ar[r] \ar[d, phantom, sloped, "\cong"] & \chowwitt^0(B\mu_n, \trivialbundle) \ar[r] & H^0(B\Gm, \KMW_{-1}, \trivialbundle(n)) \ar[d, phantom, sloped, "\cong"] \\
		0 & \GW(k)& & \KM_{-1}(k) = 0 
	\end{tikzcd}
	\end{center}
	which immediately implies $\chowwitt^0(B\mu_n, \trivialbundle) \cong \GW(k)$.
\end{proof}
	
\section{Multiplicative Structure}

For odd $n$ we observed in \cref{chowwitt-groups-of-Bmun-odd} that all Chow-Witt groups of $B\mu_n$ are isomorphic to those of $B\Gm$ modulo $n \eulercwz$.
Since the quotient map is the map $\pi^*$ from the localization sequence \ref{eq:locseq-Bmun} and this is a ring homomorphism, the Chow-Witt ring of $B\mu_n$ is also isomorphic to that of $B\Gm$ modulo $n \pi^*(\eulercwz)$ and modulo the quadratic periodicity isomorphism 
\[ \psi_1 \colon \chowwitt^i(B\mu_n, \trivialbundle(1)) \cong \chowwitt^i\left(B\mu_n, \trivialbundle\left((n+1)/2\right)^{\otimes 2}\right) \cong \chowwitt^i(B\mu_n, \trivialbundle) \, . \]
This is not true for even $n$, since in that case there is an element in $\chowwitt^0(B\mu_n, \trivialbundle)$ (denoted $U$ in \cite{diLorenzo-Mantovani}) which is not in the image of $\chowwitt^\tot(B\Gm)$.

Since $\sqperiod$ and thus also $\psi_1$ is compatible with multiplication (\cref{square-periodicity-compatibility}) it suffices to understand its action on the generators $\pi^*(\hyperbolicz)$ and $\pi^*(\eulercwz)$.
Consider
\[ \pi^*(\hyperbolicz) = (0,2) \in \chowwitt^0(B\mu_n, \trivialbundle(1)) \subseteq H^0(B\mu_n, I^0, \trivialbundle(1)) \times_{\chowmodtwo^0(B\mu_n)} \chow^0(B\mu_n)  \]
where the inclusion is justified by \cref{prop:chowwitt-is-chow-x-I} and the fact that $\chow^0(B\mu_n) \cong \Integer$ has trivial $2$-torsion.
On the second factor, $\psi_1$ is the identity.
On the first, $\psi_1$ being a group homomorphism already implies that it maps $0$ to $0$. 
Thus
\begin{equation}\label{eq:square-periodicity-H-h} \psi_1(\pi^*(\hyperbolicz)) = (0,2) \in \chowwitt^0(B\mu_n, \trivialbundle) \subseteq H^0(B\mu_n, \trivialbundle) \times_{\chowmodtwo^0(B\mu_n)} \chow^0(B\mu_n)  \end{equation}
which under the isomorphism $\chowwitt^0(B\mu_n, \trivialbundle) \cong \GW(k)$ corresponds to the hyperbolic form $\hyperbolic$.
Therefore the class $\pi^*(\hyperbolicz)$ is not required as a generator for the Chow-Witt ring of $B\mu_n$.

The relation $I(k) \cdot \hyperbolicz$ from $B\Gm$ now becomes $I(k) \cdot \hyperbolic$ which already holds in $\GW(k)$.
Similarly $\hyperbolicz^2 - 2 \hyperbolic$ becomes $\hyperbolic ^2 - 2 \hyperbolic$ which is also true in $\GW(k)$.
The relation $I(k) \cdot \eulercwz$ just becomes $I(k) \cdot \pi^*(\eulercwz)$.

Because Euler classes are compatible with pullbacks, $\pi^*(\eulercwz)$ is the Euler class of $\trivialbundle_{B\mu_n}(-1)$ and we can rename it into $\eulercwz$.
This yields the following concise description.

\begin{theorem} \label{chowwitt-ring-of-Bmun-odd}
	For $n$ odd, there is an isomorphism of $\GW(k)$-algebras
	\begin{align*} 
		\chowwitt^\tot (B\mu_n) & \cong \GW(k)[\eulercwz]/(I(k) \cdot \eulercwz, n \cdot \eulercwz)
	\end{align*}
	where $\eulercwz \in \chowwitt^1(B\mu_n)$ is the Euler class of the tautological line bundle $\trivialbundle(-1)$ on $B\mu_n$.
\end{theorem}

To obtain this description it is not necessary to compute $\psi_1 \pi^*(\eulercwz)$, but for better understanding we will illustrate the argument anyway.
Consider a second isomorphism
\[ \psi_2 \colon \chowwitt^i(B\mu_n, \trivialbundle(2)) \cong \chowwitt^i(B\mu_n, \trivialbundle(n+1)^{\otimes 2}) \cong \chowwitt^i(B\mu_n, \trivialbundle) \]
coming from the isomorphism of line bundles $\trivialbundle(2) \cong \trivialbundle(n+2)$ and the square periodicity isomorphism $\sqperiod_{\trivialbundle(n+1)}$.
This fits into the following commutative diagram: 
\[ \begin{tikzcd}
	\chowwitt^i(B\mu_n, \trivialbundle(1)) \times \chowwitt^j(B\mu_n, \trivialbundle(1)) \ar[d] \ar[r, "\operatorname{mult}"] \ar[bend right=60, swap, xshift=-55pt]{dd}[yshift=15pt]{\psi_1 \times \psi_1} & \chowwitt^{i+1}(B\mu_n, \trivialbundle(2)) \ar[d] \ar[bend left=50, xshift=30pt]{dd}[yshift=15pt]{\psi_2} \\
	\hspace{30pt} \chowwitt^i(B\mu_n, \trivialbundle(n+1)) \times \chowwitt^j(B\mu_n, \trivialbundle(n+1)) \ar[d, "\sqperiod_{\trivialbundle((n+1)/2)} \times \sqperiod_{\trivialbundle((n+1)/2)}"] \ar[r, "\operatorname{mult}"] & \chowwitt^{i+j}(B\mu_n, \trivialbundle(2n+2)) \ar[d, swap, "\sqperiod_{\trivialbundle(n+1)}"] \\
	\chowwitt^i(B\mu_n, \trivialbundle) \times \chowwitt^j(B\mu_n, \trivialbundle) \ar[r, "\operatorname{mult}"] & \chowwitt^{i+j}(B\mu_n, \trivialbundle)
\end{tikzcd}\]
The bottom square commutes by \cref{square-periodicity-compatibility} and the top one because the ring multiplication is constructed to be compatible with isomorphisms of line bundles.
With this one can compute:
\begin{align*} 
	\psi(\pi^*(\eulercwz)) &= (n+1) \cdot \psi(\pi^*(\eulercwz)) \\
	&= \frac{n+1}{2} h \psi(\pi^*(\eulercwz)) = \frac{n+1}{2} \cdot \psi(\pi^*(\hyperbolicz))\psi(\pi^*(\eulercwz)) \\
	&= \frac{n+1}{2} \cdot \psi^* \pi^*(\hyperbolicz\eulercwz) = \frac{n+1}{2} \cdot \pi^*(\hyperbolicz\eulercwz) \, . 
\end{align*}
The first equality follows from the relation $n\pi^*(\eulercwz)$, the second from the relation $I(k) \cdot \eulercwz$ and $\GW(k)$-linearity of $\pi^*$ and $\psi$, and the third from \cref{eq:square-periodicity-H-h}.
The second to last equality comes from the above diagram
and the last one holds because $\trivialbundle(2n+2)$ is already a square over $B\Gm$ and thus $\psi_2$ is already divided out in $\chowwitt^\tot(B\Gm)$.

\section{Milnor-Witt Cohomology in Non-Diagonal Bidegrees}

The following will become useful in later computations.
\begin{lemma}\label{prop:KMW-of-Bmun}
	Let $i,\, j$ integers with $i \neq 0$.
	If $j < i$ then
	\[ H^i(B\mu_n, \KMW_j, \linebundle) \cong 0 \, . \]
\end{lemma}

\begin{proof}
	Consider the localization sequence from the proof of \cref{chowwitt-groups-of-Bmun-odd}.
	Inserting \cref{prop:milnorwitt-of-BGm} shows that unless $i=0$ and $\linebundle$ is trivial, this is isomorphic to
	\[ \ldots \to \KM_{j-i}(k) \to \KM_{j-i}(k) \to H^i(B\mu_n, \KMW_j, \linebundle) \to \KM_{j-i-1}(k) \to \ldots \]
	Since $\KM_{<0}(k) = 0$ this implies $H^i(B\mu_n, \KMW_j, \linebundle) = 0$ for $j < i$.
\end{proof}

\section{\texorpdfstring{$I^j$}{Ij}-Cohomology of \texorpdfstring{$B\mu_n$}{Bmu\_n}}

\begin{corollary}\label{prop:Ij-of-Bmun}
	\begin{align*}
		\shortintertext{If $n$ is odd:}
		H^\tot(B\mu_n, I^*) &\cong \witt(k) \\
		\shortintertext{If $n$ is even:}
		H^\tot(B\mu_n, I^*) &\cong \witt(k)[\classu, \euleriz]/(I(k) \euleriz, \classu^2 + 2 \classu, \classu\euleriz) 
	\end{align*}
\end{corollary}	

\begin{proof}
	Combine \cref{chowwitt-groups-of-Bmun-even,chowwitt-ring-of-Bmun-odd} and the fact that the sequence of abelian groups
	\[ \chow^i(B\mu_n) \xrightarrow{\hyperbolic_\trivialbundle} \chowwitt^i(B\mu_n, \trivialbundle)\to H^i(B\mu_n, I^i, \trivialbundle) \to 0 \]
	is exact.
	
	First consider odd $n$.
	The composition $\forgetful \circ \hyperbolic_\trivialbundle$ equals multiplication with $2$ and further it follows from the previous computations of the groups $\chowwitt^i(B\mu_n)$ and $\chow^i(B\mu_n)$ that the reduction map $\forgetful$ is injective in positive degrees.
	For degree zero consider the fiber product formula from \cref{prop:chowwitt-is-chow-x-I} using that $\chow^0(B\mu_n) \cong \Integer$ has no $2$-torsion.
	Note that the structure map $\chowwitt^0(B\mu_n, \trivialbundle) \to H^0(B\mu_n, I^0, \trivialbundle)$ is given by dividing out the image of $\hyperbolic_\trivialbundle$ and clearly $\modulo \hyperbolic_\trivialbundle \circ \hyperbolic_\trivialbundle = 0$,
	and recall that the pair $(0,2)$ in this fiber product corresponds to the hyperbolic form $\hyperbolic \in \chowwitt^0(B\mu_n, \trivialbundle)$.
	Thus the hyperbolic map $\hyperbolic_\trivialbundle$ acts by
	\begin{align*} 
		\hyperbolic_\trivialbundle \colon \chow^*(B\mu_n) \cong \Integer[\chern]/(n \chern) & \to \GW(k)[\eulercwz]/(I(k) \cdot \eulercwz, n \eulercwz) \cong \chowwitt^\tot(B\mu_n) \\  
		1 &\mapsto \hyperbolic \\
		\chern^i &\mapsto \hyperbolic \eulercwz^i = 2 \eulercwz^i 
	\end{align*}
	Since $2 \eulercwz^i$ generates $\chowwitt^i(B\mu_n) \cong \Integer/n \langle \eulercwz^i \rangle$ as a $\GW(k)$-module for $i \geq 1$, this means that all higher $I^j$-cohomology groups vanish.
	In degree $0$ 
	\[ H^0(B\mu_n, I^0, \trivialbundle) \cong \chowwitt^0(B\mu_n, \trivialbundle) / \hyperbolic \cong \GW(k) / \hyperbolic \cong \witt(k) \, . \]
	
	If $n$ is even an analogous argument shows that the hyperbolic maps are given by
	\begin{align*}
		\hyperbolic_\trivialbundle \colon \chow^*(B\mu_n) \cong \Integer[\chern]/(n \chern) &\to \GW(k)[\classu, \hyperbolicz, \eulercwz]/(\ldots) \cong \chowwitt^\tot(B\mu_n) \\
		1 & \mapsto \hyperbolic \\
		\chern^i &\mapsto \begin{cases*}
			\hyperbolic \eulercwz^i = 2 \eulercwz^i & $i$ even \\
			\hyperbolicz \eulercwz^i & $i$ odd 
		\end{cases*} \\
		\hyperbolic_{\trivialbundle(1)} \colon \chow^*(B\mu_n) \cong \Integer[\chern]/(n \chern) &\to \GW(k)[\classu, \hyperbolicz, \eulercwz]/(\ldots) \cong \chowwitt^\tot(B\mu_n) \\
		1 &\mapsto \hyperbolicz \\
		\chern^i &\mapsto \begin{cases*}
			\hyperbolicz \eulercwz^i & $i$ even \\
			\hyperbolic \eulercwz^i = 2 \eulercwz^i & $i$ odd 
		\end{cases*}
	\end{align*}
	This shows that the groups $\chowwitt^i(B\mu_n, \trivialbundle(i+1))$ which are generated by $\hyperbolicz \eulercwz^i$ become trivial in $I^j$-cohomology.
	For $i \geq 1$, this further implies
	\[ H^i(B\mu_n, I^i, \trivialbundle(i)) \cong \Integer/2n \langle \eulercwz^i \rangle / 2 \eulercwz^i \cong \Integer/2 \langle \euleriz^i \rangle \cong \witt(k)/I(k) \langle \euleriz^i \rangle \, . \]
	The relation $n \hyperbolicz \eulercwz$ is a multiple of $\hyperbolic \eulercwz = 2 \eulercwz$ and thus no longer appears.
	In degree $0$ and trivial twist compute 
	\[ H^0(B\mu_n, I^0, \trivialbundle) \cong \chowwitt^0(B\mu_n, \trivialbundle) / \hyperbolic \cong (\GW(k) \oplus \witt(k) \langle \classu \rangle) / \hyperbolic \cong \witt(k)\langle 1, \classu \rangle  \, . \]
	
	The statement follows by adding up these groups and inheriting the multiplication and relations from the Chow-Witt ring.
\end{proof}

\begin{remark}\label{remark:real-cycles-Bmun}
	Let the base field $k$ contain the real numbers $\Real$.
	Our scheme approximations of $B\mu_n$ are not cellular, thus the condition under which  \cite[Theorem 5.7]{HWXZ} proves that the real cycle class map
	\[ H^i(B\mu_n, I^j, \linebundle) \to H_{\sing}^i (B\mu_n(\Real), \Integer(\linebundle)) \]
	is an isomorphism for $j=i$ is not met.

	If $n$ is even, this is in fact not an isomorphism:
	The real realization, i.e.\ taking the set of real points of a scheme equipped with the analytic topology, commutes with products and quotients.
	This means that the real realization of $\Bmun{r}{n} \cong \trivialbundle_{\projective^r}(n)$ which is a line bundle over $\projective^r$, is again a real line bundle over $\projective^r(\Real) = \Real\projective^r$, and in case $n$ is even this is orientable.
	The zero section is also preserved under real realization, and removing the zero section of an oriented line bundle over $\Real\projective^r$ divides the total space into two connected components which are both homotopy equivalent to $\Real\projective^r$.
	Finally, taking the colimit over our approximations of increasing dimensions also commutes with real realization, thus $B_{\mathrm{gm}}\mu_n(\Real)$ is homotopy equivalent to $\Real\projective^\infty \amalg \Real\projective^\infty$.
	Its (untwisted) singular cohomology ring is 
	\[ H_{\sing}^*(B\mu_n(\Real); \Integer) \cong H_{\sing}^*(\Real\projective^\infty; \Integer) \oplus H^*(\Real\projective^\infty; \Integer) \cong \Integer[x]/2x \oplus \Integer[y]/2y \]
	where $x$ and $y$ live in degree $2$ and this is not isomorphic to the untwisted part of the $I^j$-cohomology we have computed above.
	The real cycle class map is, however, an isomorphism for  $j \geq i + 3$:
	Jacobson \cite[Corollary 8.3]{jacobson} proves this for $j \geq \dim X$.
	By \cref{lem:admissible-gadgets-Ij} we can use $\Bmun{r}{n}$ to compute $H^i(B\mu_n, I^j, \linebundle)$ if $r \geq i+2$ and this scheme has dimension $r+1 \geq i+3$.
	What happens in these bidegrees is roughly that the last term of the localization sequence \ref{eq:locseq-Bmun} does not vanish, contributing an additional generator $\classu \eulercwz^i$ to the Milnor-Witt and thus also the $I^j$-cohomology group.
	 
	If $n$ is odd, the space of real points $B\mu_n(\Real)$ is contractible with cohomology ring $\Integer \cong \witt(\Real)$ concentrated in degree $0$, thus coinciding with our results for $I^j$-cohomology.
\end{remark}

\chapter{The \texorpdfstring{$I^j$}{Ij}-Cohomology of \texorpdfstring{$\projective^q \times \projective^r$}{P x P}} \label{sec:I-cohomology-of-PxP}

\section{Fasel's Projective Bundle Formula}

Let $X$ be a smooth scheme, $\vectorbundle$ a vector bundle of rank $r$ on $X$ and $p \colon \projective(\vectorbundle) \to X$ the associated projective bundle.
Denoting by $\vectorbundlesheaf$ the locally free $\trivialbundle_X$-module associated to $\vectorbundle$, there is a $\trivialbundle_{\projective(\vectorbundle)}$-module $\vectorbundlesheafG$ defined by the short exact sequence 
\[ 0 \to \vectorbundlesheafG \to p^* \vectorbundlesheaf \to \trivialbundle_{\projective(\vectorbundle)}(1) \to 0 \, . \]
We denote the total space of the vector bundle associated to $\vectorbundlesheafG$ by $\vectorbundleG$.
Further let $\linebundle$ be a line bundle on $X$.
For $a \in \Integer$ denote by $\linebundle(a)$ the line bundle $p^* \linebundle \otimes \trivialbundle_{\projective(\vectorbundle)}(a)$ over $\projective(\vectorbundle)$.
Recall that by quadratic periodicity (\cref{chowwitt-quadratic-periodicity}), $I^j$-cohomology twisted by $\linebundle(a)$ is isomorphic to that twisted in $\linebundle$ if $a$ is even and $\linebundle(-1)$ if $a$ is odd.

In \cite{fasel-projbundle}, Fasel constructs maps
\begin{multline*} 
	\mu_a^\linebundle \colon H^i(X, \ibar^j) \xrightarrow{p^*} H^i(\projective(\vectorbundle), \ibar^j) \xrightarrow{\bockstein_{\linebundle(-1)}} H^{i+1}(\projective(\vectorbundle), I^{j+1}, \linebundle(-1)) \\
	\xrightarrow{\cdot \euleri(\trivialbundle_{\projective(\vectorbundle)}(1))^{a-1}} H^{i+a}(\projective(\vectorbundle), I^{j+a}, \linebundle(-a)) 
	\end{multline*}
for $a \geq 1$ and $i \in \Integer$, where $\bockstein_{\linebundle(-1)}$ is the Bockstein introduced in \cref{eq:key-diagram}, and 
\begin{align*}
	\Theta_{\mathrm{even}}^\linebundle \coloneqq \sum_{\substack{1 \leq a \leq r-1 \\ a \; \mathrm{even}}} \mu_a^\linebundle & \colon \bigoplus_a H^{i-a}(X, \ibar^{j-a}) \to H^i(\projective(\vectorbundle), I^j, p^*\linebundle) \\
	\Theta_{\mathrm{odd}}^\linebundle \coloneqq \sum_{\substack{1 \leq a \leq r-1 \\ a \; \mathrm{odd}}} \mu_a^\linebundle & \colon \bigoplus_a H^{i-a}(X, \ibar^{j-a}) \to H^i(\projective(\vectorbundle), I^j, p^*\linebundle(-1))
\end{align*}
for $i \in \Integer$.
These two maps are split injective by \cite[Cor.\ 5.8]{fasel-projbundle}.

\begin{definition}[Fasel] \label{def:reduced-cohomology}
	For a vector bundle $E \to X$ and its associated projective bundle $\projective(E)$, define the groups
	\begin{align*}
		\widetilde{H}^i(\projective(E), I^j, p^*\linebundle) &\coloneqq \coker \Theta_{\mathrm{even}}^\linebundle \\
		\widetilde{H}^i(\projective(E), I^j, p^*\linebundle(-1)) & \coloneqq \coker \Theta_{\mathrm{odd}}^\linebundle \, .
	\end{align*}
\end{definition}
Fasel calls this group the reduced $I^j$-cohomology, but we will not adopt this because we are using the term to describe $\ibar^j$-cohomology.
Note also that the tilde in this definition has nothing to do with the notation for Chow-Witt groups.

Denote by $\xi_\vectorbundle \in H^{r-1}(\projective(E), I^{r-1}, \trivialbundle_{\projective(\vectorbundle)}(-r))$ the orientation class of the projective bundle $\projective(\vectorbundle)$ associated to a trivial vector bundle $\vectorbundle$ of rank $r$ over $X$, as defined in \cite[Definition 6.1]{fasel-projbundle}.
By \cite[Prop.\ 6.4]{fasel-projbundle} the triviality condition on the vector bundle can be omitted if $r$ is odd, as in this case $\xi$ agrees with the Euler class of $\vectorbundleG^\vee$.
Fasel's main theorems \cite[Theorems 9.1, 9.2, 9.4]{fasel-projbundle} summarize to the following.
\begin{proposition}[Fasel] \label{thm:proj-bundle-Ij-thm}
	Let $X$ be a scheme, $\linebundle$ a line bundle over $X$, $\vectorbundle$ a vector bundle of rank $r$ over $X$ and $p \colon \projective(\vectorbundle) \to X$ the associated projective bundle, and $i \in \Integer$.
	Then if $r$ is odd, the compositions
	\begin{enumerate}[ref=\cref{thm:proj-bundle-Ij}(\arabic*)]
		\item $H^i(X, I^j, \linebundle) \xrightarrow{p^*} H^i(\projective(\vectorbundle), I^j, p^*\linebundle) \twoheadrightarrow \widetilde{H}^i(\projective(\vectorbundle), I^j, p^*\linebundle)$ and
		\item $H^{i-r+1}(X, I^{j-r+1}, \linebundle) \xrightarrow{\cdot \xi \circ p^*} H^i(\projective(E), I^j, p^*\linebundle \otimes \omega_{\projective(\vectorbundle)/X}) \twoheadrightarrow \widetilde{H}^i(\projective(\vectorbundle), I^j, p^*\linebundle \otimes \omega_{\projective(\vectorbundle)/X})$
	\end{enumerate}
	are isomorphisms of $\witt(k)$-algebras.
	If $r$ is even, \vspace{1ex}
	\begin{enumerate}[resume, ref=\cref{thm:proj-bundle-Ij}(\arabic*)]
		\item $\widetilde{H}^i(\projective(\vectorbundle), I^j, p^*\linebundle(-1)) = 0$,
		\item and if $\vectorbundlesheaf$ is free then \vspace{-1em}
		\begin{multline*} 
			H^i(X, I^j, \linebundle) \oplus H^{i-r+1}(X, I^{j-r+1}, \linebundle) \xrightarrow{p^* \oplus (\cdot \xi \circ p^*)} H^i(\projective(\vectorbundle), I^j, p^*\linebundle) \\ 
			\twoheadrightarrow \widetilde{H}^i(\projective(E), I^j, p^*\linebundle)
		\end{multline*} 
		is an isomorphism. 
	\end{enumerate}
\end{proposition}
Together with the fact that the maps $\Theta_{\mathrm{even}}^\linebundle$ and $\Theta_{\mathrm{odd}}^\linebundle$ are split injective, this theorem yields a direct formula to compute the $I^j$-cohomology groups of a projective bundle.

Caution: $\projective^r = \projective(k^{r+1})$ is the projective bundle associated to the trivial rank $r+1$ vector bundle over $k$ and thus a rank $r+1$ projective bundle.

Before computing the $I^j$-cohomology of the product $\projective^q \times \projective^r$ we need to know the $I^j$-cohomology ring of a single copy of projective space $\projective^r$ over a base field $k$.
The additive structure is an immediate consequence of the previous theorem and the multiplicative structure is due to \cite[Prop.\ 4.1]{wendt}.
The reduced $I^j$-cohomology ring follows for example from \cite[Theorem 4.1]{fasel-projbundle}.

\begin{corollary}[Fasel, Wendt] %
		\label{lem:Ij-of-Pn}
	The total $I^j$-cohomology ring of projective space $\projective^r$ over a base field $k$ is 
	\begin{align*} 
		H^\tot(\projective^r, I^*) &\cong \witt(k)[\euleri, \orientationz]/(I(k) \cdot \euleri, \euleri^{r+1}, \euleri \cdot \orientationz, \orientationz^2) \\
	\intertext{as $\witt(k)$-algebras,
	where $\euleri$ corresponds to the class $\euleri(\trivialbundle_{\projective^r}(-1)) \in H^1(\projective^r, I^1, \trivialbundle_{\projective^r}(1))$ and $\orientationz$ corresponds to $\xi \in H^r(\projective^r, I^r, \trivialbundle_{\projective^r}(r-1))$.
	The reduced cohomology ring is}
	\bigoplus_i H^i(\projective^r, \ibar^i) & \cong \Integer/2\Integer [\sw] / \sw^{r+1} 
	\end{align*}
	where $\sw$ corresponds to the class $\sw(\trivialbundle_{\projective^r}(-1)) \in H^1(\projective^r, \ibar^1)$.
	The reduction morphism $\reduction \colon H^{\ast}(\projective^r, I^{\ast}, -) \to H^{\ast}(\projective^r, \ibar^{\ast})$ maps $\euleri$ to $\sw$ and $\orientationz$ to $\sw^r$.
\end{corollary}
Wendt characterizes the class $\euleri$ as $\bockstein_{\trivialbundle(1)}(1)$.
That this is in fact the Euler class of $\trivialbundle(-1)$ follows from \cref{lem:bockstein-of-sw}.

\section{The \texorpdfstring{$I^j$}{I}-Cohomology Ring of \texorpdfstring{$\projective^q \times \projective^r$}{P x P}}

Consider the product space $\projective^q \times \projective^r$ for $q, \, r \geq 1$.
By \cite[Prop.\ 1.30]{eisenbud-harris}, the Picard group of a smooth scheme is isomorphic to its first Chow group, which in the case of $\projective^q \times \projective^r$ is shown in \cite[Theorem 2.12]{totaro-bluebook} to be isomorphic to $\Pic(\projective^q) \times \Pic(\projective^r) \cong \Integer \times \Integer$.

\begin{notation}
We denote the two projections by 
\[ \pr_1 \colon \projective^q \times \projective^r \to \projective^r \text{  and  } \pr_2 \colon \projective^q \times \projective^r \to \projective^r \, . \]
Line bundles on $\projective^q \times \projective^r$ are of the form $\pr_1^* \trivialbundle_{\projective^q}(s) \otimes \pr_2^*\trivialbundle_{\projective^r}(t)$ and we will use the shorthand notation
\[ \trivialbundle(s,t) \coloneqq \pr_1^* \trivialbundle_{\projective^q}(s) \otimes \pr_2^*\trivialbundle_{\projective^r}(t) \, . \]
\end{notation}

$I^j$-cohomology considers twists in $\Pic(\projective^q \times \projective^r)/2 \cong \Integer/2 \times \Integer/2$,
whose elements can be represented e.g.\ by the line bundles $\trivialbundle(0,0)$, $\trivialbundle(0,1)$, $\trivialbundle(1,0)$ and $\trivialbundle(1,1)$.

\begin{theorem} \label{thm:Ij-of-PxP}
	Let $k$ be a perfect field of characteristic coprime to $2$.
	The total $I^j$-cohomology ring of $\projective^q \times \projective^r$ is isomorphic to the following $(\Integer, \Pic(\projective^q \times \projective^r)/2)$-graded $\witt(k)$-algebra:
	\begin{multline*} 
		\witt(k)[\euleria, \eulerib, \euleric, \orientationa, \orientationb]/(I(k) \cdot \euleria, I(k) \cdot \eulerib, I(k) \cdot \euleric, \euleria^2 + \eulerib^2 - \euleric^2, \\
		\euleria^{q+1}, \eulerib^{r+1}, \euleria \orientationa, \eulerib \orientationb,  \orientationa^2, \orientationb^2, \\
		\eulerib \orientationa - \euleria^q \euleric, \euleria \orientationb - \eulerib^r \euleric, \euleric \orientationa - \euleria^q \eulerib, \euleric \orientationb - \euleria \eulerib^r)
	\end{multline*}
	where
	\begin{align*}
	\euleria &\mapsto \euleri(\trivialbundle(-1,0)) &&\in H^1(\projective^q \times \projective^r, I^1, \trivialbundle(1,0)) \\
	\eulerib &\mapsto \euleri(\trivialbundle(0,-1)) &&\in H^1(\projective^q \times \projective^r, I^1, \trivialbundle(0,1)) \\
	\euleric &\mapsto \euleri(\trivialbundle(-1,-1)) &&\in H^1(\projective^q \times \projective^r, I^1, \trivialbundle(1,1)) \\
	\orientationa &\mapsto \xi_{\trivialbundle_{\projective^r}^{q+1}} &&\in H^q(\projective^q \times \projective^r, I^q, \trivialbundle(q+1,0)) \\
	\orientationb &\mapsto \xi_{\trivialbundle_{\projective^q}^{r+1}} &&\in H^r(\projective^q \times \projective^r, I^r, \trivialbundle(0, r+1)) \, .
	\end{align*}
\end{theorem}

To prove this theorem we will apply Fasel's formula \cref{thm:proj-bundle-Ij-thm} to the projective bundle 
\[ p \coloneqq \pr_2 \colon \projective^q \times \projective^r \cong \projective(\trivialbundle_{\projective^r}^{\oplus q + 1}) \to \projective^r \, . \]

The reduced cohomology of $\projective^q \times \projective^r$ can be computed immediately using \cite[Theorem 4.1]{fasel-projbundle}:

\begin{proposition}
	\[ \bigoplus_{i} H^i(\projective^q \times \projective^r, \ibar^i) \cong \Integer/2\Integer[\swa, \swb]/(\swa^{q+1}, \swb^{r+1}) \]
	where 
	\begin{align*}
		\swa &\mapsto \sw(\trivialbundle(-1,0)) &&\in H^1(\projective^q \times \projective^r, \ibar^1) \\
		\swb &\mapsto \sw(\trivialbundle(0,-1)) &&\in H^1(\projective^q \times \projective^r, \ibar^1) 
	\end{align*}
\end{proposition}

To shorten notation, in the following we will use the symbols $\euleria$, $\eulerib$, $\euleric$, $\orientationa$, $\orientationb$, $\swa$, $\swb$ to also denote their corresponding classes in the (reduced) $I^j$-cohomology ring.
Observe that
\begin{align*}
\reduction(\euleria) &= \swa \\
\reduction(\eulerib) &= \swb \\
\reduction(\euleric) &= \sw(\trivialbundle(-1,-1)) = \sw(\trivialbundle(-1,0) \otimes \trivialbundle(0,-1)) = \swa + \swb \\
\reduction(\orientationa) &= \swa^q \\
\reduction(\orientationb) &= \swb^r \, .
\end{align*}
where the third row uses the formula for Stiefel-Whitney classes of tensor products of line bundles, and the rest follows by pulling the respective classes back to one factor of the product where the reduction morphism is already known.

The classes $\sw^i$ for $0 \leq i \leq r$ generate $H^\tot(\projective^r, I^*)$ as a bi-graded $\witt(k)$-module
and thus their images under the maps
\[ p^*, \, \mu_a^\linebundle \text{ and } (\cdot \xi_{\trivialbundle_{\projective^r}^{q+1}}) \circ p^* = (\cdot \orientationa) \circ p^* \colon H^\tot(\projective^r, \ibar^*) \to H^\tot(\projective^q \times \projective^r, \ibar^*) \] 
appearing in Fasel's formula produce a generating set of $H^\tot(\projective^q \times \projective^r, I^*)$.
Our strategy will be to express all of these in terms of the classes $\euleria$, $\eulerib$, $\euleric$, $\orientationa$, $\orientationb$ to show that those generate $H^\tot(\projective^q \times \projective^r, I^*, -)$ as a $\witt(k)$-algebra.
For the pullback map $p^*$ we simply have $p^* \sw^i = \swb^i$, and likewise $\orientationa \cdot p^*(\sw^i) = \orientationa \cdot \swb^i$.
The main obstacle is to understand the Bockstein homomorphism $\bockstein_{\trivialbundle(1,t)}$ that occurs as a component of $\mu_a^\linebundle$.
We will use the following lemma of Fasel \cite[Lemma 3.1]{fasel-projbundle}.

\begin{lemma}[Fasel]%
	\label{lem:bockstein-of-sw}
	Let $X$ be a smooth scheme and $\linebundle$, $\linebundleb$ line bundles over $X$.
	Let $\alpha \in H^i(X, I^j, \linebundle)$.
	Then the Bockstein homomorphism
	\[ \bockstein_{\linebundle \otimes \linebundleb^\vee} \colon H^i(X, \ibar^i) \to H^{i+1}(X, I^{i+1}, \linebundle \otimes \linebundleb^\vee) \]
	maps $\reduction(\alpha)$ to $\alpha \cdot e(\linebundleb)$.
\end{lemma}

To compute $\bockstein_{\trivialbundle(s,t)}(\swa^k \swb^l)$, insert $\linebundle = \trivialbundle(k,l)$, $\alpha = \euleria^k \eulerib^l \in H^i(\projective^q \times \projective^r, I^i, \trivialbundle(k,l))$ and $\linebundleb = \trivialbundle(k-s, l-t)$ and obtain
\[ \bockstein_{\trivialbundle(s,t)} (\swa^k \swb^l) = \bockstein_{\trivialbundle(s,t)} ( \rho_{\trivialbundle(k,l)} (\euleria^k \eulerib^l)) = \euleria^k \eulerib^l \cdot \euleri(\trivialbundle(k-s,l-t)) \, . \]
Note that Euler classes of squares vanish in $I^j$-cohomology, hence only the parity of $k-s$ and $l-t$ is relevant.

In the projective bundle formula there appear only Bocksteins whose twist is tautological on the first factor of the product, and they are composed with the map $p^* \colon H^i(\projective^r, \ibar^i) \to H^i(\projective^q \times \projective^r, \ibar^i)$ whose image is generated by $p^* c^i = \swb^i$. 
So the relevant values are:
\begin{align*}
	\bockstein_{\trivialbundle(1,i)} (\swb^i) &= \euleria \eulerib^i \\
	\bockstein_{\trivialbundle(1,i+1)} (\swb^i) &= \eulerib^i \euleric \\
\shortintertext{and thus}
	\mu_a^{\trivialbundle_{\projective^r}(i)} (\sw^i) &= \euleria^a \eulerib^i \\
	\mu_a^{\trivialbundle_{\projective^r}(i+1)} (\sw^i) &= \euleria^{a-1} \eulerib^i \euleric \, .
\end{align*}

Having determined the images of the maps $p^*$, $\mu_a^\linebundle$ and $\cdot \orientationa$ appearing in Fasel's formula yields a set of generators of all diagonal $I^j$-cohomology groups of $\projective^q \times \projective^r$ as $\witt(k)$-modules.
By describing all of these as products of $\euleria, \eulerib, \euleric, \orientationa, \orientationb$ it follows that these five classes generate the total $I^j$-cohomology ring as a $\witt(k)$-algebra.

Concerning relations, there are $\euleria^{q+1}$, $\euleria \orientationa$, $\orientationa^2$ as well as $\eulerib^{r+1}$, $\eulerib \orientationb$, $\orientationb^2$ which are pulled back from the individual factors.
Since all other possible products of $\euleria$, $\eulerib$, $\orientationa$ and $\orientationb$ generate their own summands in the $I^j$-cohomology groups of $\projective^q \times \projective^r$ there are no new relations between these four generators.

So it remains to determine the relations involving $\euleric$.
Since $\euleric$ lives in 
\begin{align*} 
	H^1(\projective^q \times \projective^r, I^1, \trivialbundle(1,1)) &= \mu_1^{\trivialbundle_{\projective^r}(-1)} H^0(\projective^r, \ibar^0) \\ 
	&= \mu_1^{\trivialbundle_{\projective^r}(-1)}(1) \cdot \witt(k)/I(k) \\
	&= \euleric \cdot \witt(k)/I(k) 
\end{align*} 
one can deduce the relation $I(k) \cdot \euleric$.
From our computations of the $\mu_a^\linebundle$ deduce further that the $\euleria^i \eulerib^j \euleric$ for $0 \leq i < q$ and $0 \leq j < r$ each generate their own summand and thus are not subject to any relations.
To determine $\euleric^2$ consider its reduction:
\[ \reduction(\euleric^2) = \reduction(\euleric)^2 = (\swa + \swb)^2 = \swa^2 + 2\swa\swb + \swb^2 = \swa^2 + \swb^2 \]
The last equality holds because $2 \in I(k)$ vanishes in $I(k)$-torsion.
Since $\euleric^2$ lives in $H^2(\projective^q \times \projective^r, I^2, \trivialbundle)$ which by the projective bundle formula equals $\witt(k)/I(k) \cdot \euleria^2 \oplus \witt(k)/I(k) \cdot \eulerib^2$, 
the only possible preimage of $\swa^2 + \swb^2$ is $\euleria^2 + \eulerib^2$.
Therefore $\euleric^2 = \euleria^2 + \eulerib^2$.
Similar arguments yield $\euleria^q \euleric = \eulerib \orientationa$, $\eulerib^r \euleric = \euleria \orientationb$, $\euleric \orientationa = \euleria^q \eulerib$ and $\euleric \orientationb = \euleria \eulerib^r$.
Now we have covered all possible monomials involving $\euleric$ and thus determined all relations, finishing the computation.

\begin{corollary}\label{lem:Ij-of-BGmxBGm}
	Let $k$ be a perfect field of characteristic coprime to $2$.
	The total $I^j$-cohomology ring of $B\Gm \times B\Gm$ is isomorphic to the following $(\Integer, \Pic(B\Gm \times B\Gm)/2)$-graded $\witt(k)$-algebra:
	\[ \witt(k)[\euleria, \eulerib, \euleric]/(I(k) \cdot \euleria, I(k) \cdot \eulerib, I(k) \cdot \euleric, \euleria^2 + \eulerib^2 - \euleric^2) \]
	where
	\begin{align*}
		\euleria &\mapsto \euleri(\trivialbundle(-1,0)) &&\in H^1(B\Gm \times B\Gm, I^1, \trivialbundle(1,0)) \\
		\eulerib &\mapsto \euleri(\trivialbundle(0,-1)) &&\in H^1(B\Gm \times B\Gm I^1, \trivialbundle(0,1)) \\
		\euleric &\mapsto \euleri(\trivialbundle(-1,-1)) &&\in H^1(B\Gm \times B\Gm, I^1, \trivialbundle(1,1))  \, . \\
	\end{align*}
\end{corollary}

\begin{proof}
	This follows from the computation of $H^\tot(\projective^q \times \projective^r, I^*)$ in \cref{thm:Ij-of-PxP} and the approximation of classifying spaces in \cref{lem:admissible-gadgets-Ij,lem:admissible-gadgets-products}.
	$\projective^q \times \projective^r$ is an approximation for $B\Gm \times B\Gm$ in degrees $< \min(q,r)-1$.
	The classes $\orientationa$ and $\orientationb$ as well as all relations involving $\orientationa$, $\orientationb$, $\euleria^q$ or $\eulerib^r$ live in degrees above this boundary and thus don't appear here.
\end{proof}

\begin{remark}
	The considered schemes $\projective^q \times \projective^r$ are all cellular and thus the real cycle class map 
	\[ H^i(X, I^j, \linebundle) \to H^i (X(\Real), \Integer(\linebundle)) \]
	is an isomorphism for all $i$ by \cite[Theorem 5.7]{HWXZ}.
	Compare this for example with the computation of \cite[Example 3.E.5]{hatcher} of the singular cohomology of $\Real \projective^\infty \times \Real \projective^\infty$ in (non-twisted) $\Integer$-coefficients:
	\[ H^*(\Real\projective^\infty \times \Real\projective^\infty; \Integer) \cong \Integer[\lambda, \mu, \nu]/(2\lambda, 2\mu, 2\nu, \nu^2 + \lambda^2 \mu + \lambda \mu^2) \]
	with $\lambda$ and $\mu$ living in degree $2$ and $\nu$ in degree $3$.
	This embeds as the untwisted subring of the total $I^j$-cohomology ring of $B\Gm \times B\Gm$ via
	\begin{align*}
		\lambda &\mapsto \euleria^2 \\ 
		\mu &\mapsto \eulerib^2 \\ 
		\nu &\mapsto \euleria \eulerib \euleric
	\end{align*} 
	and the relation  
	\[ \nu^2 + \lambda^2 \mu + \lambda \mu^2 \mapsto \euleria^2 \eulerib^2 \euleric^2 + \euleria^4 \eulerib^2 + \euleria^2 \eulerib^4 = \euleria^2 \eulerib^2 (\euleria^2 + \eulerib^2 + \euleric^2) \] 
	is respected since $\euleria^2 + \eulerib^2 + \euleric^2 = 0$.
\end{remark}

\section{\texorpdfstring{$I^j$}{Ij}-Cohomology in Non-Diagonal Bidegrees}

The computations in the following chapters will involve $H^i(\projective^q \times \projective^r, I^j, \linebundle)$ for $j = i-1$.
It turns out that this group is straightforward to compute for $j < i$ which we will do in this section.

\begin{proposition} \label{lem:Ij-of-PxP-nondiag}
	Let $j < i$.
	\[ H^i(\projective^q \times \projective^r, I^j, \linebundle) \cong \begin{cases*}
		\witt(k) & $i=0$, $\linebundle = \trivialbundle(0,0)$ \\
		\witt(k) \oplus \witt(k) & $i=q=r$ odd, $\linebundle = \trivialbundle(0,0)$ \\
		\witt(k) & $i=r$, $\linebundle = \trivialbundle(0,r-1)$ and [$r \neq q$ or $r$ even] \\
		\witt(k) & $i=q$, $\linebundle = \trivialbundle(q-1,0)$ and [$q \neq r$ or $q$ even] \\
		\witt(k) & $i=q+r$, $\linebundle = \trivialbundle(q-1,r-1)$ \\
		0 & else \\
	\end{cases*} \]
\end{proposition}

\begin{proof}
Observe that for $j < i$, $I^{j-i} = \witt(k)$ by definition and $\ibar^{j-i} = 0$.
This implies that in these cases the $I^j$-cohomology groups are isomorphic to the groups $\widetilde{H}^i(\projective^q \times \projective^r, I^j, \linebundle)$ as defined in \cref{def:reduced-cohomology}.

Recall:
\[ H^i(\projective^r, I^j, \linebundle) \cong \begin{cases*}
	\witt(k) & $i=0$, $\linebundle = \trivialbundle$ \\
	\witt(k) & $i=r$, $\linebundle = \trivialbundle(r-1)$ \\
	0 & else
\end{cases*} \]
The statement then follows immediately from the projective bundle theorem (\cref{thm:proj-bundle-Ij-thm}).
\end{proof}

\chapter{The Chow-Witt Ring of \texorpdfstring{$B\Gm \times B\Gm$}{BG\_m x BG\_m}} \label{sec:Chow-Witt-of-PxP}

\section{The Chow-Witt Ring of \texorpdfstring{$\projective^q \times \projective^r$}{P x P} and \texorpdfstring{$B\Gm \times B\Gm$}{BG{\_}m x BG{\_}m}}

According to \cref{prop:chowwitt-is-chow-x-I} there is an isomorphism
\[ c \colon \chowwitt^\tot(\projective^q \times \projective^r) \cong H^{*}(\projective^q \times \projective^r, I^{*}, -) \times_{\bigoplus_\linebundle \chowmodtwo^{*}(\projective^q \times \projective^r)} \bigoplus_{\linebundle \in \Pic(\projective^q \times \projective^r)/2} \ker \partial_\linebundle \]
where
\[ \partial_\linebundle^i \colon \chow^i(\projective^q \times \projective^r) \to H^{i+1}(\projective^q \times \projective^r, I^{i+1}, \linebundle) \]
is the boundary map of the long exact cohomology sequence associated to the sequence $I \to \KMW \to \KM$ of chain complexes.
In order to compute this kernel, first compute the Chow ring in which it lives using a result of Totaro about Chow rings of products:

\begin{proposition}[{\cite[Lemma 2.3, Theorem 2.12]{totaro-bluebook}}]
	\begin{align*}
		\chow^*(\projective^r) &\cong \Integer[\chern]/(\chern^{r+1}) \\
		\chow^*(\projective^q \times \projective^r) &\cong \Integer[\cherna,\chernb]/(\cherna^{q+1}, \chernb^{r+1}) \\
		\chow^*(\projective^q \times \projective^r)/2 & \cong H^*(\projective^q \times \projective^r, \ibar^*) \cong \Integer/2\Integer[\sw_1,\sw_2]/(\sw_1^{q+1}, \sw2^{r+1})
	\end{align*}
\end{proposition}

According to the \enquote{key diagram} explained in \cref{sec:fiber-products},  $\partial_\linebundle$ equals $\bockstein_\linebundle \circ \modulo 2$.
We have already computed this Bockstein $\bockstein_\linebundle$ in the previous section after \cref{lem:bockstein-of-sw}, so now it is straightforward to describe its kernel.
From this explicit description one checks that for $\alpha \in \ker \partial_\linebundle$ and $\beta \in \ker \partial_\linebundleb$, the product $\alpha \beta$ lies in $\ker \partial_{\linebundle \otimes \linebundleb}$ and thus the kernels assemble into the $\Integer \times (\Pic/2)$-graded ring
\begin{align*}
	\Integer [C_1, C_2, C_3, r_1, r_2, h_1, h_2, h_3]/(\mathcal{J}^\prime, \mathcal{K}) &\cong \bigoplus_{\linebundle \in \Pic(\projective^q \times \projective^r)/2} \ker \partial_\linebundle \\
	C_1 &\mapsto \cherna \in \ker \partial_{\trivialbundle(1,0)}^1 \\
	C_2 &\mapsto \chernb \in \ker \partial_{\trivialbundle(0,1)}^1 \\
	C_3 &\mapsto \cherna + \chernb \in \ker \partial_{\trivialbundle(1,1)}^1 \\
	r_1 &\mapsto \cherna^q \in \ker \partial_{\trivialbundle(q+1,0)}^q \\
	r_2 &\mapsto \chernb^r \in \ker \partial_{\trivialbundle(0,r+1)}^r \\
	h_1 &\mapsto 2 \in \ker \partial_{\trivialbundle(1,0)}^0 \\
	h_2 &\mapsto 2 \in \ker \partial_{\trivialbundle(0,1)}^0 \\
	h_3 &\mapsto 2 \in \ker \partial_{\trivialbundle(1,1)}^0
\end{align*}
where $\mathcal{J}^\prime$ is the ideal generated by
\begin{align*}
	& h_1^2 - 4, h_2^2 - 4, h_3^2 - 4,  \\
	& h_1h_2 - 2h_3, h_2h_3 - 2h_1, h_1h_3 - 2h_2,  \\
	& 2 C_1 + h_3 C_2 - h_2 C_3, h_1 C_1 + h_2 C_2 - h_3 C_3, h_2 C_1 + h_1 C_2 - 2 C_3, h_3 C_1 + 2 C_2 - h_1 C_3
\end{align*}
and $\mathcal{K}$ is the ideal generated by
\begin{align*}
	&C_1^2 + C_2^2 + h_3 C_1 C_2 - C_3^2, \\ 
	&C_1^{q+1}, C_2^{r+1}, C_1 r_1, C_2 r_2, r_1^2, r_2^2, \\
	&C_1^q C_3 - C_2 r_1, C_2^r C_3 - C_1 r_2, C_3 r_1 - C_1^q C_2, C_3 r_2 - C_1 C_2^r \, .
\end{align*}
Note that $\ker \partial_{\trivialbundle(0,0)}$ also contains $1 = \cherna^0\chernb^0$ and $(\cherna + \chernb)^2 = \cherna^2 + 2\cherna\chernb + \chernb^2$.

So now we can apply \cref{prop:chowwitt-is-chow-x-I} and \cref{thm:Ij-of-PxP} to obtain the following.

\begin{corollary} \label{prop:chowwitt-of-PxP}
	For $k$ a perfect field of characteristic coprime to $2$, there is an isomorphism of graded $\GW(k)$-algebras
	\begin{multline*} 
		\chowwitt^\tot(\projective^q \times \projective^r) \cong 
		\GW(k) [\hyperbolica, \hyperbolicb, \hyperbolicc, \eulercwa, \eulercwb, \eulercwc, \orientationa, \orientationb]/(I(k) \cdot (\hyperbolica, \hyperbolicb, \hyperbolicc, \eulercwa, \eulercwb, \eulercwc) \\
		\eulercwa^2 + \eulercwb^2 + \hyperbolicc \eulercwa \eulercwb - \eulercwc^2, \mathcal{J}, \eulercwa^{q+1}, \eulercwb^{r+1}, \eulercwa \orientationa,\eulercwb \orientationb, \orientationa^2, \orientationb^2, \\
		\eulercwa^q \eulercwc - \eulercwb \orientationa, \eulercwb^r \eulercwc - \eulercwa \orientationb, \eulercwc \orientationa - \eulercwa^q \eulercwb, \eulercwc \orientationb - \eulercwa \eulercwb^r) 
	\end{multline*}
	where $\eulercwa$, $\eulercwb$, $\eulercwc$ correspond to 
	\[ \eulercw(\linebundle^\vee) \in \chowwitt^1(\projective^q \times \projective^r, \linebundle) \]
	and $\hyperbolica$, $\hyperbolicb$, $\hyperbolicc$ to
	\[ (0,2) \in \chowwitt^0(\projective^q \times \projective^r, \linebundle) \subseteq H^0(\projective^q \times \projective^r, I^0, \linebundle) \times_{\chowmodtwo^0(\projective^q \times \projective^r)} \chow^0(\projective^q \times \projective^r) \]	
	for $\linebundle = \trivialbundle(1,0), \trivialbundle(0,1), \trivialbundle(1,1)$, respectively,
	and $\orientationa$ and $\orientationb$ are the pullbacks of the orientation classes along the projections onto each factor.
	The ideal $\mathcal{J}$ is generated by the relations
	\begin{align}
		& \hyperbolica^2 - 2\hyperbolic, \hyperbolicb^2 - 2\hyperbolic, \hyperbolicc^2 - 2\hyperbolic, \label{eq:relations-BGmxBGm-1} \\
		& \hyperbolica\hyperbolicb - 2\hyperbolicc, \hyperbolicb\hyperbolicc - 2\hyperbolica, \hyperbolica\hyperbolicc - 2\hyperbolicb, \label{eq:relations-BGmxBGm-2} \\
		& 2 \eulercwa + \hyperbolicc \eulercwb - \hyperbolicb \eulercwc, \hyperbolica \eulercwa + \hyperbolicb \eulercwb - \hyperbolicc \eulercwc, \hyperbolicb \eulercwa + \hyperbolica \eulercwb - 2 \eulercwc, \hyperbolicc \eulercwa + 2 \eulercwb - \hyperbolica \eulercwc \label{eq:relations-BGmxBGm-3} \, .
	\end{align}
\end{corollary}

\noindent Note that 
\[ (0,2) = \hyperbolic_\trivialbundle(1) \in \chowwitt^0(\projective^q \times \projective^r, \trivialbundle) \subseteq H^0(\projective^q \times \projective^r, I^0, \trivialbundle) \times_{\chowmodtwo^0(\projective^q \times \projective^r)} \chow^0(\projective^q \times \projective^r) \]
is the hyperbolic form $\hyperbolic \in \GW(k)$,
and the classes $\hyperbolica$, $\hyperbolicb$, $\hyperbolicc$ are the images of $1$ under the hyperbolic maps $\hyperbolic_\linebundle \colon \chow^0(\projective^q \times \projective^r) \to \chowwitt^0(\projective^q \times \projective^r, \linebundle))$ for $\linebundle = \trivialbundle(1,0), \trivialbundle(0,1), \trivialbundle(1,1)$, respectively.
They can be thought of as twisted versions of the hyperbolic form $\hyperbolic$.

\begin{proof}
We will use the fiber product formula from \cref{prop:chowwitt-is-chow-x-I}.
The fiber product of two $\GW(k)$-algebras contains precisely those pairs in the set-wise product where both entries have the same image in $\chowmodtwo^*(\projective^q \times \projective^r)$.
This entire proof consists of spelling out that set-wise characterization.

Recall that we have a set of generators $C_1, \, C_2,\, C_3,\, h_1,\, h_2,\, h_3,\, r_1,\, r_2$ of $\oplus_\linebundle \ker \partial_\linebundle$ and $\euleria, \, \eulerib, \, \euleric, \, \orientationa, \, \orientationb$ of $H^\tot(\projective^q \times \projective^r, I^*)$ as $\GW(k)$-algebras.
Thus as $\GW(k)$-modules, these two are generated by all monomials of the form $h_1^{i_1} h_2^{i_2} h_3^{i_3} C_1^{i_4} C_2^{i_5} C_3^{i_6} r_1^{i_7} r_2^{i_8}$ respectively $\euleria^{j_4} \eulerib^{j_5} \euleric^{j_6} \orientationa^{j_7} \orientationb^{j_8}$  where the exponents are non-negative integers.
Observe that monomials of the first form are mapped to zero in $\chowmodtwo^*(\projective^q \times \projective^r)$ if and only if the monomial vanishes in $\oplus_\linebundle \ker \partial_\linebundle$ already or if at least one of $i_1, \, i_2, \, i_3$ is not zero.
A monomial of the second form is in the kernel of $\reduction$ if and only if it already vanishes in $H^\tot(\projective^q \times \projective^r, I^*)$.

To determine for a given monomial $\alpha$ of the first form all pairs in the fiber product that have $\alpha$ as second coordinate, compute the $\modulo 2$-reduction $\overline{\alpha}$ and its preimage $\reduction^{-1}(\overline{\alpha})$ in the $I^j$-cohomology ring.
The latter can be read off from our description of $H^\tot(\projective^q \times \projective^r, I^*)$ in \cref{thm:Ij-of-PxP}.

If $\alpha = 0$ or one of the exponents $i_1, \, i_2, \, i_3$ is non-zero, then $\reduction^{-1}(\overline{\alpha}) = \ker (\reduction)$.
Otherwise if $\alpha = C_1^{i_4} C_2^{i_5} C_3^{i_6} r_1^{i_7} r_2^{i_8}$ we find that $\reduction^{-1}(\overline{\alpha}) = \euleria^{i_4} \eulerib^{i_5} \euleric^{i_6} \orientationa^{i_7} \orientationb^{i_8} + \ker(\reduction)$.
Hence the following form a set of generators of the fiber product as a $\GW(k)$-module.
\vspace{0.5em}
\begin{itemize} \setlength\itemsep{0.5em}
	\item $(0, h_1^{i_1} h_2^{i_2} h_3^{i_3} C_1^{i_4} C_2^{i_5} C_3^{i_6} r_1^{i_7} r_2^{i_8})$ if one out of $i_1, \, i_2, \, i_3$ non-zero
	\item $(\euleria^{i_4} \eulerib^{i_5} \euleric^{i_6} \orientationa^{i_7} \orientationb^{i_8}, C_1^{i_4} C_2^{i_5} C_3^{i_6} r_1^{i_7} r_2^{i_8})$ for $i_4, \, i_5, \, i_6, \, i_7, \, i_8$ any non-negative integers such that the second entry of the pair is not zero
	\item $(\beta, 0)$ with $\beta \in \ker(\reduction)$
\end{itemize}
\vspace{0.5em}
The third case, however, is redundant:
Any $\beta \in \ker(\reduction)$ can be expressed as a $\witt(k)$-linear combination of monomials $\euleria^{j_4} \eulerib^{j_5} \euleric^{j_6} \orientationa^{j_7} \orientationb^{j_8}$ with coefficients $\beta_{j_4, \ldots, j_8}$ in $I(k)\subseteq \witt(k)$.
Now if $\beta_{j_4, \ldots, j_8}$ is in $I(k)$, then the pair $(\beta_{j_4, \ldots, j_8},0)$ is in the fiber product $\witt(k) \times_{\Integer/2} \Integer \cong \GW(k)$ and therefore one can write
\[ (\beta,0) = \sum_{j_4, \ldots, j_8} (\beta_{j_4, \ldots, j_8}, 0) \cdot (\euleria^{j_4} \eulerib^{j_5} \euleric^{j_6} \orientationa^{j_7} \orientationb^{j_8}, C_1^{j_4} C_2^{j_5} C_3^{j_6} r_1^{j_7} r_2^{j_8}) \]
thus reducing to the second case.

All pairs in the first two cases can be expressed as monomials in the following eight pairs, which therefore form a set of generators of the fiber product as a $\GW(k)$-\textit{algebra}. 
We assign the symbols on the left-hand side to the corresponding elements in the Chow-Witt ring.
\begin{align*}
	c \colon \chowwitt^\tot(\projective^q \times \projective^r) \xrightarrow{\cong}& H^\tot(\projective^q \times \projective^r, I^{*}) \times_{\bigoplus_\linebundle \chowmodtwo^{*}(\projective^q \times \projective^r)} \bigoplus_{\linebundle \in \Pic(\projective^q \times \projective^r)/2} \ker \partial_\linebundle \\
	\hyperbolica, \hyperbolicb, \hyperbolicc \mapsto &(0, h_1), (0, h_2), (0, h_3) \\
	\eulercwa, \eulercwb, \eulercwc \mapsto &(\euleria, C_1), (\eulerib, C_2), (\euleric, C_3) \\
	\orientationa, \orientationb \mapsto& (\orientationa, r_1), (\orientationb, r_2) \\
	\intertext{Relations in the fiber product on the right hand side consist of pairs such that each coordinate is a relation in the respective factor. %
	The fundamental ideal becomes}
	I(k) \mapsto& (I(k),0) \, 
	\intertext{thus all pairs which contain $I(k)$ in the first coordinate become relations in the fiber product, that is, }
	I(k) \eulercwa, I(k)\eulercwb, I(k) \eulercwc \mapsto& (I(k), 0) \cdot (\euleria, C_1) = (I(k) \euleria, 0) = (0,0), \, \ldots
	\intertext{Further one finds relations}
	I(k) \hyperbolica, I(k) \hyperbolicb, I(k) \hyperbolicc \mapsto& (I(k), 0) \cdot (0, h_1), \, \ldots
	\intertext{The relations in $\mathcal{J}^\prime$ live in degree $0$ and $1$ where the only non-trivial relations in the $I^*$-cohomology ring are $I(k) \euleriz_i$ which are already covered above. %
	Thus each of these has a straightforward analogue in the Chow-Witt ring}
	\hyperbolica^2 - 2h = \hyperbolica^2 - h^2, \, \ldots \mapsto& (0, h_1^2) - (0, 4)	, \, \ldots \\
	\hyperbolic \eulercwa + \hyperbolicc \eulercwb - \hyperbolicb \eulercwc \mapsto& (0, 2) (\euleria, C_1) + (0, h_3) (\eulerib, C_2) - (0, h_2) (\euleric, C_3) \\
	=& (0, 2C_1 + h_3 C_2 - h_2 C_3)
	\intertext{and we will denote the ideal generated by these by $\mathcal{J}$. %
	Something interesting happens with the relation $\euleria^2 + \eulerib^2 - \euleric^2$. %
	There are several pairs in the fiber product containing this in the first coordinate, the most obvious one being}
	\eulercwa^2 + \eulercwb^2  - \eulercwc^2 \mapsto& (\euleria^2 + \eulerib^2 -\euleric^2, C_1^2 + C_2^2 - C_3^2) \\
	 =& (\euleria^2 + \eulerib^2 -\euleric^2, C_1^2 + C_2^2 - (C_1^2 + C_2^2 + h_3 C_1 C_2)) \\
	= &(\euleria^2 + \eulerib^2 -\euleric^2, - h_3 C_1 C_2) \\
	\intertext{whose second coordinate is not a relation in the Chow ring. %
	Only subtracting $(0, - h_3 C_1 C_2)$ yields the following relation.}
	\eulercwa^2 + \eulercwb^2 + \hyperbolicc \eulercwa \eulercwb - \eulercwc^2 \mapsto& (\euleria^2 + \eulerib^2 -\euleric^2, C_1^2 + C_2^2 + h_3 C_1 C_2 - C_3^2) \\
	=& (\euleria^2 + \eulerib^2 - \euleric^2, 0)
\intertext{All other relations from \cref{thm:Ij-of-PxP} appear in a unique pair in the fiber product whose second coordinate is in $\mathcal{K}$, thus each yields one relation:}
	\eulercwa^{q+1}, \, \ldots \mapsto& (\euleria^{q+1}, C_1^{q+1}), \, \ldots \\
	\eulercwa^q \eulercwc - \eulercwb \orientationa, \, \ldots \mapsto& (\euleria^q \euleric - \eulerib \orientationa, C_1^q C_3 - C_2 r_1), \, \ldots \qedhere
\end{align*}
\end{proof}

\begin{corollary}\label{prop:chowwitt-of-BGmxBGm}
	For $k$ a perfect field of characteristic coprime to $2$, there is an isomorphism of graded $\GW(k)$-algebras
	\begin{multline*} 
		\chowwitt^\tot(B\Gm \times B\Gm) \cong 
		\GW(k) [\eulercwa, \eulercwb, \eulercwc, \hyperbolica, \hyperbolicb, \hyperbolicc]/(I(k) \cdot (\hyperbolica, \hyperbolicb, \hyperbolicc, \eulercwa, \eulercwb, \eulercwc) \\
		\eulercwa^2 + \eulercwb^2 + \hyperbolicc \eulercwa \eulercwb - \eulercwc^2, \mathcal{J})  
	\end{multline*}
	with the symbols on the right hand side corresponding to the same classes as in \cref{prop:chowwitt-of-PxP}.
\end{corollary}

\begin{proof}
	We have already observed that $(V_1^{r+1}\setminus \{0\})/\Gm \cong \projective^r$ is an approximation of $B\Gm$ in codimension $<r$ in the sense of \cref{lem:admissible-gadgets}.
	By \cref{lem:admissible-gadgets-products} the product space $\projective^q \times \projective^r$ is an approximation of $B\Gm \times B\Gm$ in codimension $< \min{q,r} -1$.
	
	In the description of \cref{prop:chowwitt-of-PxP} the classes $\eulercwa$, $\eulercwb$, $\eulercwc$, $\hyperbolica$, $\hyperbolicb$, $\hyperbolicc$ remain stable for all $q,r \geq 2$ and thus survive in the Chow-Witt ring of $B\Gm \times B\Gm$.
	The classes $\orientationa$ and $\orientationb$ on the other hand live in codimension $q$ respectively $r$ which is not in the range where $\projective^q \times \projective^r$ is an approximation for $B\Gm \times B\Gm$.
	The same holds for all relations involving $\eulercwa^q$ or $\eulercwb^r$.
\end{proof}

\section{Milnor-Witt \texorpdfstring{$K$}{K}-Theory Groups in Non-Diagonal Bidegrees}

For the computations in \cref{sec:Chow-Witt-of-PmxBmun} we will also need to know the group $H^i(\projective^q \times \projective^r, \KM_{j})$ for $j \in \{i-1, i-2\}$.
To this end observe that the argument of \cite[Prop.\ 2.11]{hornbostel-wendt} extends to a statement about Milnor-Witt cohomology in non-diagonal bidegrees.
In this case the key diagram of \cite{hornbostel-wendt} reads as follows:
\begin{center}
	\begin{tikzcd}
		& H^i(X, \KM_j) \ar[r, "\id"] \ar[d, "\hyperbolic_\linebundle"] & H^i(X, \KM_j) \ar[d, "\cdot 2"] & \\
		H^i(X, I^{j+1}, \linebundle) \ar[r] \ar[d, "\id"] & H^i(X, \KMW_j, \linebundle) \ar[r, "\forgetful"] \ar[d, "\modulo \hyperbolic_\linebundle"] & H^i(X, \KM_j) \ar[r, "\partial"] \ar[d, "\modulo 2"] & H^{i+1}(X, I^{j+1}, \linebundle) \ar[d, "\id"] \\
		H^i(X, I^{j+1}, \linebundle) \ar[r, "\eta"] & H^i(X, I^j, \linebundle) \ar[r, "\reduction"] \ar[d, "\delta"] & H^i(X, \KM_j/2) \ar[r, "\bockstein"] \ar[d] & H^{i+1}(X, I^{j+1}, \linebundle) \\
		& H^{i+1}(X, \KM_j) \ar[r, "\id"] & H^{i+1}(X, \KM_j) & 
	\end{tikzcd}
\end{center}

\begin{proposition}\label{hornbostel-wendt-fiber-product}
	The canonical group homomorphism
	\[ \phi \coloneqq (\modulo \hyperbolic_\linebundle) \times \forgetful \colon H^i(X, \KMW_j, \linebundle) \to \ker\delta \times_{H^i(X, \KM_j/2)} \ker \partial \]
	is always surjective. 
	It is injective if one of the following conditions hold:
	\begin{enumerate}
		\item $H^i(X, \KM_j)$ has no non-trivial $2$-torsion.
		\item The map $\eta \colon H^i(X, I^{j+1}, \linebundle) \to H^i(X, I^j, \linebundle)$ is injective.
	\end{enumerate}
	If further $H^{i+1}(X, \KM_j)$ has no non-trivial $2$-torsion, then $\ker \delta$ equals $H^i(X, I^j, \linebundle)$.
\end{proposition}
%
%
\begin{proof}
The existence of the group homomorphism $\alpha$ is implied by the universal property of the fiber product.

For surjectivity of $\phi$ consider an element $(\alpha, \beta) \in \ker\delta \times_{H^i(X, \KM_j/2)} \ker \partial$.
Choose an arbitrary lift $\widetilde{\alpha} \in H^i(X, \KMW_j, \linebundle)$ of $\alpha$ and denote the image of this lift in $\ker \partial$ by $\gamma$.
The reduction of $\gamma$ agrees with the reduction of $\beta$,
thus $\beta - \gamma$ has a lift $\widetilde{\beta} \in H^i(X, \KM_j)$.
Then $\widetilde{\alpha} + \hyperbolic_\linebundle(\widetilde{\beta})$ is in the preimage of $(\alpha, \beta)$ under $\phi$.

Assume condition (1) holds and consider an element $\varepsilon \in \ker \phi = \ker \forgetful \, \cap \, \ker (\modulo \hyperbolic_\linebundle)$.
By exactness this must have a lift $\widetilde{\varepsilon}$ along $\hyperbolic_\linebundle$.
Due to commutativity of the upper square in the key diagram $\widetilde{\varepsilon}$ must be $2$-torsion and thus trivial by assumption,
hence $\varepsilon = 0$.
This proves injectivity.
Under condition (2) choose a lift $\widetilde{\varepsilon}^\prime$ in $H^i(X, I^{j+1}, \linebundle)$. 
Then $\widetilde{\varepsilon}^\prime$ must be zero due to injectivity of $\eta$ and commutativity of the left-most square, and thus $\varepsilon = 0$.

The last statement about $\ker \delta$ follows from the commutativity of the top square in the diagram: 
If multiplication with $2$ is injective, then so is $\forgetful \circ \hyperbolic_\linebundle$ and thus $\hyperbolic_\linebundle$. 
Since $\delta$ sits in the left vertical exact sequence before $\hyperbolic_\linebundle$, it follows that $\delta = 0$.
\end{proof}

\begin{lemma} \label{lem:KMWi-1-of-PxP}
	For $j < i$:
	\begin{align*} 
		H^i(\projective^q \times \projective^r, \KMW_{j}, \linebundle) &\cong
		\begin{cases*}
			\witt(k) & $i=0$, $\linebundle = \trivialbundle$ \\
			\witt(k) \langle \orientationa, \orientationb \rangle & $i=q=r$ odd, $\linebundle = \trivialbundle(0,0)$ \\
			\witt(k) \langle \orientationa \rangle & $i=q$, $\linebundle = \trivialbundle(q-1,0)$ and [$q \neq r$ or $q$ even] \\
			\witt(k) \langle \orientationb \rangle & $i=r$, $\linebundle = \trivialbundle(0,r-1)$ and [$q \neq r$ or $r$ even] \\
			\witt(k) \langle \orientationa \cdot \orientationb \rangle & $i = q+r$, $\linebundle = \trivialbundle(q-1, r-1)$ \\
			0 & else
		\end{cases*} \\
		H^i(B\Gm \times B\Gm, \KMW_j, \linebundle) &\cong
		\begin{cases*}
			\witt(k) & $i=0$, $\linebundle = \trivialbundle$ \\
			0 & else
		\end{cases*}
	\end{align*}
\end{lemma}

\begin{proof}
	Note that $H^i(\projective^q \times \projective^r, \KM_j)$ vanishes for $j < i$ (as follows e.g.\ from \cite[Theorem 2.12]{totaro-bluebook}) thus condition (1) of the fiber product formula \cref{hornbostel-wendt-fiber-product} is satisfied, and its subgroup $\ker \partial$ vanishes as well.
	Hence $H^i(\projective^q \times \projective^r, \KMW_j, \linebundle) \cong \ker \delta \subseteq H^i(\projective^q \times \projective^r, I^j, \linebundle)$ where the latter has been computed in \cref{lem:Ij-of-PxP-nondiag}.
	Since $H^{i+1}(\projective^q \times \projective^r, \KM_j)$ also vanishes, the last inclusion is in fact an equality.
	
	For the statement on $B\Gm \times B\Gm$, use the approximation of equivariant Chow-Witt groups from \cref{lem:admissible-gadgets-products}.
\end{proof}

\chapter{Chow-Witt Ring of \texorpdfstring{$B\Gm \times B\mu_n$}{BG\_m x Bmu\_n}} %
	\label{sec:Chow-Witt-of-PmxBmun}
	
Throughout this chapter, let $n$ be a positive integer, and $k$ a perfect field with characteristic coprime to $2$ and $n$.

The Picard group of $B\Gm \times B\mu_n$ is isomorphic to its first Chow group by \cite[Prop.\ 1.30]{eisenbud-harris} which is computed in \cite[Theorem 2.10, Lemma 2.12]{totaro-bluebook} to be isomorphic to $\Integer \times \Integer/n$.
For line bundles on $B\Gm \times B\mu_n$ (or more precisely, its scheme approximations) we adopt the same notation
\[ \trivialbundle(s,t) = \pr_1^*\trivialbundle_{B\Gm} (s) \otimes \pr_2^* \trivialbundle_{B\mu_n}(t) \]
as for bundles on $B\Gm \times B\Gm$.
Let $\Bmun{r}{n}$ be the approximation of $B\mu_n$ constructed in \cref{sec:Chow-Witt-of-Bmun} and consider the vector bundle $\pi \colon \projective^q \times \Bmun{r}{n} \to \projective^q \times \projective^r$.
Let $\linebundle$ be a line bundle over $\projective^q \times \Bmun{r}{n}$.
Consider the localization sequence associated to the closed subscheme
\[ \projective^q \times \projective^r \xhookrightarrow{(\id, s_0)} \projective^q \times \Bmun{r}{n} \xleftarrow{(\id, \iota)} \projective^q \times (\Bmun{r}{n} \setminus s_0(\projective^r)) \, . \]
Using the same arguments as in \cref{chowwitt-groups-of-Bmun-odd} \textemdash{} scheme approximations of classifying spaces, the isomorphism $\det \Omega_{\projective^q \times \Bmun{r}{n}/\projective^q \times \projective^r}^\vee \cong \det \Omega_{\trivialbundle_{\projective^q \times \projective^r}(n)/\projective^q \times \projective^r}^\vee \cong \trivialbundle(0,n)$ of line bundles over $\projective^q \times \projective^r$, and homotopy invariance \textemdash{} for $q$ and $r$ sufficiently large this sequence is isomorphic to
\begin{multline} \label{eq:locseq-BGmxBmun}
	\ldots \to \chowwitt^{i-1}(B\Gm \times B\Gm, s_0^*\linebundle \otimes \trivialbundle(0,-n)) \xrightarrow{\cdot \eulercw(\pi)} \chowwitt^i(B\Gm \times B\Gm, \linebundle) \\
	\xrightarrow{\iota^*} \chowwitt^i(B\Gm \times B\mu_n, \iota^*\linebundle) \to H^i(B\Gm \times B\Gm, \KMW_{i-1}, s_0^*\linebundle \otimes \trivialbundle(0,n)) \to \ldots 
\end{multline}
As already shown in \cref{sec:Chow-Witt-of-Bmun} the line bundle $\Bmun{r}{n} \to \projective^r$ is isomorphic to $\trivialbundle_{\projective^r}(n)$.
Thus $\projective^q \times \Bmun{r}{n}$ is isomorphic to $\pr_2^* \trivialbundle_{\projective^r}(n) = \trivialbundle(0,n)$ and therefore the map $\cdot \eulercw(\pi)$ can also be expressed as $\cdot \eulercw(\trivialbundle(0,n))$.
This Euler class can be computed using \cref{prop:euler-class-product}:
\[ \eulercw (\trivialbundle(0,n))) = \begin{cases*} 
	-\frac{n}{2} \cdot \hyperbolic_{\trivialbundle}(\cherna(\trivialbundle(0,1)))= - \hyperbolicb \eulercwb & $n$ even \\
	- \frac{n}{2} \hyperbolic \eulercwb = -n \eulercwb & $n$ odd 
\end{cases*} \]

\section{Group Structure}

\begin{theorem}\label{chowwitt-groups-of-BGmxBmun}
	For $\linebundle$ any line bundle on $B\Gm \times B\mu_n$ the following are isomorphisms of $\GW(k)$-modules.
	\begin{align*}
		\shortintertext{If $n$ is odd:}
		\chowwitt^i(B\Gm \times B\mu_n, \linebundle) & \cong \chowwitt^i(B\Gm \times B\Gm) / n \eulercwb \\
		\shortintertext{If $n$ is even:}
		\chowwitt^i(B\Gm \times B\mu_n, \linebundle) & \cong \begin{cases*}
			\GW(k) \oplus \witt(k)\langle \classub \rangle & $i=0$, $\linebundle$ trivial \\
			\chowwitt^i(B\Gm \times B\Gm, \linebundle)/\frac{n}{2} \hyperbolicb \eulercwb & else
		\end{cases*}
	\end{align*}
	Here $\classub$ is the pullback of the class $\classu$ from \cref{chowwitt-groups-of-Bmun-even} along the projection $\pr_2 \colon B\Gm \times B\mu_n \to B\mu_n$ onto the second factor, and
	$\eulercwb$ is the pullback of $\eulercwb$ defined in \cref{prop:chowwitt-of-BGmxBGm} along $\pi$.
\end{theorem}

Since the bundle $\trivialbundle(0,1)$ on $B\Gm \times B\mu_n$ is defined as the pullback along $\pi$ of the same bundle on $B\Gm \times B\Gm$, and Euler classes commute with pullbacks, $\eulercwb$ is actually the Euler class of $\trivialbundle(0,-1)$ on $B\Gm \times B\mu_n$.

\begin{proof}
	Except for $i=0$ and trivial twist, the rightmost term in the localization sequence \cref{eq:locseq-BGmxBmun} is shown to vanish in \cref{lem:KMWi-1-of-PxP}, proving the statement in these cases.
	
	Assume that $n$ is odd.
	Then the localization sequence implies
	\begin{align*} 
		&\chowwitt^i(B\Gm \times B\mu_n, \linebundle) \\
		&\cong \chowwitt^i(B\Gm \times B\Gm, \linebundle) / \eulercw(\trivialbundle(0,n)) \cdot \chowwitt^{i-1}(B\Gm \times B\Gm, \linebundle \otimes \trivialbundle(0,n)) \\
		&\cong \chowwitt^i(B\Gm \times B\Gm, \linebundle) / n \eulercwb \cdot \chowwitt^{i-1}(B\Gm \times B\Gm, \linebundle \otimes \trivialbundle(0,n)) \, .
\Theta_{\mathrm{even}}^\linebundle	\end{align*}
	In fact this argument also covers the case $i=0$ with trivial twist, since over $B\Gm \times B\mu_n$ the line bundle $\trivialbundle(1,0)$ is trivial and thus
	\[ \chowwitt^0(B\Gm \times B\mu_n, \trivialbundle) \cong \chowwitt^0(B\Gm \times B\mu_n, \trivialbundle(0,1)) \, . \]
	This works because $\trivialbundle(1,0)$ is trivial over $B\Gm \times B\mu_n$ but not over $B\Gm \times B\Gm$.
	
	Now assume that $n$ is even.
	For $i=0$ and trivial twist, \cref{lem:KMWi-1-of-PxP} shows that $H^0(B\Gm \times B\Gm, \KMW_{-1}, s_0^*\trivialbundle \otimes \trivialbundle(0,n)) \cong \witt(k)$  and $H^{1}(B\Gm \times B\Gm, \KMW_0, \trivialbundle) \cong 0$.
	In the previous section we have computed $\chowwitt^0(B\Gm \times B\Gm, \trivialbundle) \cong \GW(k)$.
	Furthermore $\chowwitt^{-1}(B\Gm \times B\Gm, s_0^*\trivialbundle \otimes \trivialbundle(0,-n))$ vanishes for degree reasons.
	Thus the localization sequence \ref{eq:locseq-BGmxBmun} reads:
	\[ \label{eq:locseq-BGmxBmun-deg0} 0 \to \GW(k) \to \chowwitt^0(B\Gm \times B\mu_n, \trivialbundle) \to \witt(k) \to 0 \]
	Now consider the class $\classu \in \chowwitt^0(B\mu_n, \trivialbundle)$ constructed in \cite[Prop.\ 5.2.1]{diLorenzo-Mantovani} and its pullback along the projection $\pr_2 \colon B\Gm \times B\mu_n \to B\mu_n$ (or $\pr_2 \colon \projective^q \times \Bmun{r}{n} \setminus s_0 \to \Bmun{r}{n} \setminus s_0$ to be precise).
	This class is shown in \cite{diLorenzo-Mantovani} to define a split of the boundary morphism
	\[ \chowwitt^0(B\mu_n, \iota^*\linebundle) \to H^0(B\Gm, \KMW_{-1}, s_0^*\linebundle \otimes \trivialbundle(0,n)) \]
	of the localization sequence \cref{eq:locseq-Bmun}.
	The map $\pr_2$ together with its respective restrictions induces the following maps between the sequences \cref{eq:locseq-Bmun} and \cref{eq:locseq-BGmxBmun} (all groups have trivial twist):
	\[ \begin{tikzcd}[column sep=small]
		\ldots \ar[r, "\cdot n \eulercwz"] & \chowwitt^0(B\Gm) \ar[d, "\pr_2^*"] \ar[r, "\iota^*"] & \chowwitt^0(B\mu_n) \ar[d, "\pr_2^*"] \ar[r, "\partial"] & H^0(B\Gm, \KMW_{-1}) \ar[d, "\pr_2^*"] \ar[r] & \ldots \\
		\ldots \ar[r, "\cdot n \eulercwb"]  & \chowwitt^0(B\Gm \times B\Gm) \ar[r, "(\id \times \iota)^*"] & \chowwitt^0(B\Gm \times B\mu_n) \ar[r, "\partial"] & H^0(B\Gm \times B\Gm, \KMW_{-1}) \ar[r] & \ldots
	\end{tikzcd}\]
	This diagram commutes according to \cref{lem:locseq-pullback} since we have
	\[ \pr_2^* N(s_0 \colon \projective^r \to \Bmun{r}{n}) \cong \pr_2^* \trivialbundle_{\projective^r(n)} \cong \trivialbundle(0,n) \cong N(\id \times s_0 \colon \projective^q \times \projective^r \to \projective^q \to \Bmun{r}{n})  \, . \]
	
	The right-most vertical arrow is an isomorphism which follows from the projective bundle theorem (\cref{thm:proj-bundle-Ij-thm}) together with the fact that these two Milnor-Witt cohomology groups are naturally isomorphic to the corresponding $I^j$-cohomology groups as explained in the proof of \cref{lem:KMWi-1-of-PxP}.
	Together with commutativity of the right-hand square this shows that $\pr_2^*(\classu) \eqqcolon \classub$ defines a split of the boundary map in the localization sequence \ref{eq:locseq-BGmxBmun-deg0} which proves the statement for $i=0$ and $\linebundle$ trivial.
\end{proof}

\section{Ring Structure}

\begin{theorem} \label{thm:chowwitt-of-BGmxBmun}
	We have the following isomorphisms of graded $\GW(k)$-algebras.
	\vspace{-3ex} \begin{align*}
		\intertext{If $n$ is odd:}
		\chowwitt^\tot(B\Gm \times B\mu_n)  \cong &\GW(k)[\eulercwa, \eulercwb, \hyperbolica]/(I(k) \cdot (\hyperbolica, \eulercwa, \eulercwb), \hyperbolica^2 - 2 \hyperbolic, n \eulercwb) \\
	\shortintertext{where}
		\hyperbolica \mapsto & \hyperbolic_{\trivialbundle(1,0)}(1) \in \chowwitt^0(B\Gm \times B\mu_n, \trivialbundle(1,0)) \\
		\eulercwa \mapsto &\eulercw (\trivialbundle(1,0)) \in \chowwitt^1(B\Gm \times B\mu_n, \trivialbundle(1,0)) \\
		\eulercwb \mapsto &\eulercw (\trivialbundle(0,1)) \in \chowwitt^1(B\Gm \times B\mu_n, \trivialbundle) \\
	\intertext{If $n$ is even:}
		\chowwitt^\tot(B\Gm \times B\mu_n)  \cong & \GW(k)[ \hyperbolica, \hyperbolicb, \hyperbolicc, \classub, \eulercwa, \eulercwb, \eulercwc]/(I(k) \cdot (\hyperbolica, \hyperbolicb, \hyperbolicc, \eulercwa, \eulercwb, \eulercwc), \\
		&\mathcal{J}, \eulercwa^2 + \eulercwb^2 + \hyperbolicc \eulercwa \eulercwb - \eulercwc^2, \frac{n}{2}\hyperbolicb \eulercwb, \classub^2 + 2 \classub, (\hyperbolic, \hyperbolica, \hyperbolicb, \hyperbolicc) \cdot \classub, \\
		& \classub\eulercwa - \frac{n}{2}\hyperbolicc\eulercwb, \classub \eulercwb - n \eulercwb, \classub \eulercwc - \frac{n}{2}\hyperbolica\eulercwb)
	\shortintertext{where $\mathcal{J}$ is the ideal of relations from \cref{prop:chowwitt-of-PxP}, }
		\hyperbolica \mapsto & \;\hyperbolic_{\trivialbundle(1,0)}(1) \in \chowwitt^0(B\Gm \times B\mu_n, \trivialbundle(1,0)) \\
		\hyperbolicb \mapsto & \;\hyperbolic_{\trivialbundle(0,1)}(1) \in \chowwitt^0(B\Gm \times B\mu_n, \trivialbundle(0,1)) \\
		\hyperbolicc \mapsto & \;\hyperbolic_{\trivialbundle(1,1)}(1)  \in \chowwitt^0(B\Gm \times B\mu_n, \trivialbundle(1,1)) \\
		\eulercwa \mapsto &\eulercw (\trivialbundle(-1,0)) \in \chowwitt^1(B\Gm \times B\mu_n, \trivialbundle(1,0)) \\
		\eulercwb \mapsto &\eulercw (\trivialbundle(0,-1)) \in \chowwitt^1(B\Gm \times B\mu_n, \trivialbundle(0,1)) \\
		\eulercwc \mapsto &\eulercw (\trivialbundle(-1,-1)) \in \chowwitt^1(B\Gm \times B\mu_n, \trivialbundle(1,1)) 
	\end{align*}
	and $\classub \in \chowwitt^0(B\Gm \times B\mu_n, \trivialbundle)$ is the pullback of the class $\classu$ introduced in \cref{def:classu} along the projection map $\pr_2 \colon B\Gm \times B\mu_n \to B\mu_n$.
\end{theorem}
\noindent Note that $\hyperbolic_\linebundle(1)$ equals
\[ (0,2) \in \chowwitt^0 (B\Gm \times B\mu_n, \linebundle) \subseteq H^0(B\Gm \times B\mu_n, I^0, \linebundle) \times_{\chowmodtwo^0(B\Gm \times B\mu_n)} \chow^0(B\Gm \times B\mu_n) \]
for all line bundles $\linebundle$ in $\Pic(B\Gm \times B\mu_n)$.
The inclusion here is justified by \cite[Prop.\ 2.11]{hornbostel-wendt} if $\chowwitt^0(B\Gm \times B\mu_n)$ has no non-trivial $2$-torsion which is satisfied by \cite[Theorem 2.10, Lemma 2.12]{totaro-bluebook}.

\begin{proof}
	For odd $n$ this computation is very similar to that of $B\mu_n$ in \cref{chowwitt-ring-of-Bmun-odd}:
	The Chow-Witt ring of $B\Gm \times B\mu_n$ is isomorphic to that of $B\Gm \times B\Gm$ modulo $n \eulercwb$ and the quadratic periodicity isomorphisms for $\trivialbundle(0,1) \cong \trivialbundle(0,n+1) \cong \trivialbundle(0, (n+1)/2) ^{\otimes 2}$ and $\trivialbundle(1,1) \cong \trivialbundle(1,n+1) \cong \trivialbundle(1, (n+1)/2) ^{\otimes 2}$.
	
	Consider the class 
	\begin{align*} 
		\hyperbolicb = (0,2) &\in \chowwitt^0(B\Gm \times B\mu_n, \trivialbundle(0,1)) \\ &\subseteq H^0(B\Gm \times B\mu_n, I^0, \trivialbundle(0,1)) \times_{\chowmodtwo^0(B\Gm \times B\mu_n)} \chow^0(B\Gm \times B\mu_n) \, . 
	\end{align*} 
	By  \cref{chowwitt-quadratic-periodicity} the square periodicity isomorphism $\sqperiod$ acts as the identity on the second factor, and on the first factor it maps $0$ to $0$, thus it sends $\hyperbolicb$ to $\hyperbolic \in \chowwitt^0(B\Gm \times B\mu_n, \trivialbundle)$.
	Similarly, $\hyperbolicc$ maps to $\hyperbolica$.
	
	As a consequence, $\eulercwb$ is mapped to $(n+1)/2 \hyperbolicb \eulercwb$.
	The latter is a generator of the free summand $\Integer/n \langle \hyperbolicb \eulercwb$ of $\chowwitt^1(B\Gm \times B\mu_n, \trivialbundle(0,0))$ and would also serve as a $\GW(k)$-algebra generator of the total Chow-Witt ring, but in order to achieve a more readable presentation we will instead write down $\eulercwb$.
	
	For $\sqperiod(\eulercwc) \in \chowwitt^1(B\Gm \times B\mu_n, \trivialbundle(1,0))$ note that the reduction map to Chow groups sends $\eulercwc$ and thus also $\sqperiod(\eulercwc)$ to $\cherna + \chernb$.
	As we have computed in \cref{chowwitt-groups-of-BGmxBmun},
	\begin{align*} 
		\chowwitt^1(B\Gm \times B\mu_n, \trivialbundle(1,0)) &= \chowwitt^1(B\Gm \times B\Gm, \trivialbundle(1,0))/(n \cdot \eulercwb) \\
		&= \Integer \langle \eulercwa \rangle \oplus \Integer \langle \hyperbolicc \eulercwb \rangle \oplus \Integer \langle \hyperbolicb \eulercwc \rangle / (n \cdot \eulercwb, \hyperbolic \eulercwa + \hyperbolicc \eulercwb - \hyperbolicb \eulercwc) \\
		&= \Integer \langle \eulercwa \rangle \oplus \Integer/n \langle \hyperbolicc \eulercwb \rangle
	\end{align*}
	and the only element here whose image in the Chow group is $\cherna + \chernb$, is $\eulercwa + (n+1)/2 \cdot \hyperbolicc \eulercwb$.

	With this one can simplify the relations inherited from $B\Gm \times B\Gm$ as follows.
	\begin{align*}
		\text{\cref{eq:relations-BGmxBGm-1,eq:relations-BGmxBGm-2}} & \text{ are all equivalent to } \hyperbolica^2 - 2\hyperbolic. \\
		\hyperbolic \eulercwa + \hyperbolicc \eulercwb - \hyperbolicb \eulercwc &= \hyperbolic \eulercwa + \hyperbolica \eulercwb - \hyperbolic(\eulercwa + \frac{n+1}{2}\hyperbolica \eulercwb) \\
		&= \hyperbolica \eulercwb - \frac{n+1}{2} 2\hyperbolica \eulercwb \\
		&= 0 \quad \text{and analogously for the other relations from \cref{eq:relations-BGmxBGm-3}} \\
		I(k) \cdot \hyperbolicb = I(k) \cdot \hyperbolicc &= I(k) \cdot \hyperbolica \\
		I(k) \cdot \eulercwc &= I(k) \cdot \eulercwa + I(k) \cdot (n+1)/2 \hyperbolicc \eulercwb \text{ follows from } I(k)\eulercwa \text{ and } I(k) \eulercwb \\
		\eulercwa^2 + \eulercwb^2 + \hyperbolica \eulercwa \eulercwb - \eulercwc^2 &= \eulercwa^2 + \eulercwb^2 + \hyperbolica \eulercwa \eulercwb - (\eulercwa + \frac{n+1}{2}\hyperbolicc\eulercwb)^2 \\
		&= \eulercwa^2 + \eulercwb^2 + \hyperbolica \eulercwa \eulercwb - \eulercwa^2 - (n+1) \eulercwa \hyperbolica \eulercwb - (\frac{n+1}{2})^2 4\eulercwb^2 = 0 
	\end{align*}
	This completes the proof for odd $n$.
	
	For even $n$, the relations in the Chow-Witt ring of $B\Gm \times B\mu_n$ are also inherited from $B\Gm \times B\Gm$ and since the Picard group of these two spaces are isomorphic there are no quadratic periodicity isomorphisms to be divided out. 
	The only thing left to determine are the relations involving $\classub$.
	
	By construction, $\hyperbolic \classub = 0$.
	To compute $\hyperbolica \classub$, consider the reduction map
	\[ \begin{tikzcd}
		\chowwitt^0(B\Gm \times B\mu_n, \trivialbundle(1,0) \ar[r, "\forgetful"] \ar[d, phantom, sloped, "\cong"] & \chow^0(B\Gm \times B\mu_n) \ar[d, phantom, sloped, "\cong"] \\
		\Integer \langle \hyperbolica \rangle \ar[r] & \Integer \\
	\end{tikzcd} \]
	This sends the generator $\hyperbolica$ to $2$, and since the target is torsion-free this map is injective.
	It is also compatible with multiplication.
	Since $\forgetful (\classub) =0$ by construction, it follows that $\forgetful( \hyperbolica \classub) = 0$ and thus one concludes $\hyperbolica \classub = 0$.
	Analogous arguments show $\hyperbolicb \classub = \hyperbolicc \classub  =0$.

	To compute $\classub ^2$ recall that $\classub$ is defined as the image of $\classu \in \chowwitt^0(B\mu_n, \trivialbundle)$ as constructed in \cref{def:classu} under the ring homomorphism 
	\[ \pr_2^* \colon \chowwitt^\tot (B\mu_n) \to \chowwitt^\tot (B\Gm \times B\mu_n) \]
	induced by the projection $\pr_2 \colon B\Gm \times B\mu_n \to B\mu_n$ on the second factor.
	By \cite[Prop.\ 5.2.3]{diLorenzo-Mantovani}, for $B\mu_n$ there is a relation $\classu^2 = - 2 \classu$.
	This implies
	\[ \classub^2 = \pr_2^*(\classu)^2 = \pr_2^*(\classu^2) = \pr_2^*(-2\classu) = -2 \classub \]
	over $B\Gm \times B\mu_n$.
	The same argument yields
	\[ \classub \eulercwb = \pr_2^*(\classu) \pr_2^*(\eulercwz) = \pr_2^*(\classu \eulercwz) = \pr_2^*(n \eulercwz) = n\eulercwb \]
	where $\eulercwz$ is the Euler class of the bundle $\trivialbundle_{B\mu_n}(1)$ (denoted $T$ in the \cite{diLorenzo-Mantovani}).
	
	For $\classub \eulercwa$ consider the reduction map
	\[ \forgetful \colon \chowwitt^1(B\Gm \times B\mu_n, \trivialbundle(1,0)) \to \chow^1(B\Gm \times B\mu_n) \]
	which sends $\classub$ and thus also $\classub \eulercwa$ to zero. 
	By \cref{chowwitt-groups-of-BGmxBmun} the Chow-Witt group on the left is
	\begin{align*}
		\chowwitt^1(B\Gm \times B\mu_n, \trivialbundle(1,0)) &\cong \chowwitt^1(B\Gm \times B\Gm, \trivialbundle(1,0))/n \hyperbolicc \eulercwb \\
		&= \Integer \langle \eulercwa \rangle \oplus \Integer \langle \hyperbolicc \eulercwb \rangle \oplus \Integer \langle \hyperbolicb \eulercwc \rangle /(\hyperbolicb \eulercwc - \hyperbolic \eulercwa - \hyperbolicc \eulercwb, n \hyperbolicc \eulercwb) \\
		&= \Integer \langle \eulercwa \rangle \oplus \Integer/n \langle \hyperbolicc \eulercwb \rangle 
	\end{align*}
	and the Chow group on the right is $\Integer \langle \cherna \rangle \oplus \Integer/n \langle \chernb \rangle$ by \cite[Theorem 2.10, Lemma 2.12]{totaro-bluebook}.
	The reduction map $\forgetful$ maps Euler classes to the corresponding Chern classes, thus the kernel of $\forgetful$ contains only the two elements $0$ and $n/2 \hyperbolicc \eulercwb$.
	To distinguish between these two, consider
	\[ \classub \eulercwa \eulercwb = \langle -1 \rangle ^2 \classub \eulercwb \eulercwa = (\classub \eulercwb) \eulercwa = n \eulercwb \eulercwa \]
	which implies that $\classub \eulercwb$ cannot be zero, leaving only the possibility $\classub \eulercwa = n/2 \hyperbolicc \eulercwb$.
	An analogous argument shows $\classub \hyperbolicc = n/2 \hyperbolica \eulercwb$.
\end{proof}



\section{Milnor-Witt \texorpdfstring{$K$}{K}-Theory Groups in Non-Diagonal Bidegrees}

The computation in the next section will require $H^i(B\Gm \times B\mu_n, \KMW_{i-1}, \linebundle)$.

\begin{lemma} \label{lem:KMWi-1-of-BGMxBmun}
	Let $j  < i$.
	If $i \neq 0$ or $\linebundle$ is not trivial, then $H^i(B\Gm \times B\mu_n, \KMW_j, \linebundle)$ vanishes.
\end{lemma}

\begin{proof}
	Consider again the localization sequence \cref{eq:locseq-BGmxBmun}
	which for arbitrary bidegrees reads
	\begin{multline*} 
		\ldots \to H^{i-1}(B\Gm \times B\Gm, \KMW_{j-1}, s_0^*\linebundle \otimes \trivialbundle(0,n)) \to H^i(B\Gm \times B\Gm, \KMW_j, \linebundle) \\ \xrightarrow{\iota^*} H^i(B\Gm \times B\mu_n, \KMW_j, \iota^*\linebundle) \to H^i(B\Gm \times B\Gm, \KMW_{j-1}, s_0^*\linebundle \otimes \trivialbundle(0,n)) \to \ldots 
	\end{multline*}
	The second and fourth terms are shown in \cref{lem:KMWi-1-of-PxP} to vanish - implying that the third term vanishes as well - except when $i=0$ and $\linebundle$ respectively $s_0^*\linebundle \otimes \trivialbundle(0,n)$ are trivial.
\end{proof}

\section{\texorpdfstring{$I^j$}{Ij}-Cohomology of \texorpdfstring{$B\Gm \times B\mu_n$}{BG\_m x Bmu\_n}}

\begin{corollary}
	\begin{align*}
		\shortintertext{If $n$ is odd:}
		H^\tot(B\Gm \times B\mu_n, I^*) \cong & W(k)[\euleria]/(I(k) \euleria) \\
		\shortintertext{If $n$ is even:}
		H^\tot(B\Gm \times B\mu_n, I^*) \cong & W(k)[\classub, \euleria, \eulerib, \euleric](I(k) \cdot (\euleria,\eulerib, \euleric), \euleria^2 + \eulerib^2 - \euleric^2, \\
		&\classub^2 + 2\classub, \classub\cdot (\euleria, \eulerib, \euleric)) 
	\end{align*}
	The symbols $\classub$, $\euleria$, $\eulerib$, $\euleric$ are as in \cref{thm:chowwitt-of-BGmxBmun}.
\end{corollary}

\begin{proof}
	Combine the computation of $\chowwitt^\tot(B\Gm \times B\mu_n)$ from \cref{thm:chowwitt-of-BGmxBmun} and the fact that the sequence 
	\[ \chow^i(B\Gm \times B\mu_n) \xrightarrow{\hyperbolic_\linebundle} \chowwitt^i(B\Gm \times B\mu_n, \linebundle) \to H^i(B\Gm \times B\mu_n, I^i, \linebundle) \to 0 \]
	is exact.
	
	Using a similar argument as in \cref{prop:Ij-of-Bmun} we find that for $n$ odd, 
	the images of the hyperbolic maps $\hyperbolic_{\trivialbundle(0,0)}$ and $\hyperbolic_{\trivialbundle(1,0)}$ are the submodules of $\chowwitt^\tot(B\Gm \times B\mu_n)$ generated by $\hyperbolic$ and $\hyperbolica$, respectively.
	Dividing out $\hyperbolic \eulercwb = 2 \eulercwb$ also identifies the relation $n \eulercwb$ with $\eulercwb$, therefore the latter will be identified with zero as well.
	
	If $n$ is even, the images of the hyperbolic maps are generated by $\hyperbolic$, $\hyperbolica$, $\hyperbolicb$ and $\hyperbolicc$.
	After dividing these out, the relation $\frac{n}{2}\hyperbolicb\eulercwb$ disappears because it is a multiple of $\hyperbolicb$.
	The same holds for all of the relation ideal $\mathcal{J}$.
	
	The statement then follows by adding up the computed cohomology groups and inheriting the ring structure from $\chowwitt^\tot(B\Gm \times B\mu_n)$.
\end{proof}

\begin{remark}\label{remark:kuenneth-formula-1}
If $n$ is odd, both the Chow-Witt ring and the $I^j$-cohomology ring satisfy a Künneth isomorphism
\[ \chowwitt^\tot(B\Gm) \otimes_{\GW(k)} \chowwitt^\tot(B\mu_n) \cong \chowwitt^\tot(B\Gm \times B\mu_n) \] 
as $\GW(k)$-algebras.
If $n$ is even, however, the Künneth map is neither injective nor surjective:
Both the Chow-Witt and $I^j$-cohomology ring contain an element $\eulercwc$ that does not come from one of the factors, and a new relation $\classub \eulercwa - (n/2)\hyperbolicc \eulercwb$ respectively $\classub \euleria$.

Comparing with singular cohomology, the Künneth map not being surjective is not surprising as the classical Künneth theorem also features an additional Tor-term when none of the factor spaces has all free cohomology groups.
The failure of injectivity, on the other hand, is unexpected as this can never happen in singular cohomology.
This can be explained by considering the real cycle class map
\[ H^i(B\Gm \times B\mu_n, I^j, \linebundle) \to H^i_{\sing}((B\Gm \times B\mu_n)(\Real); \Integer(\linebundle)) \, . \] 
\cite[Theorem 5.7]{HWXZ} does not apply in this case since (our scheme approximations of) $B\Gm \times B\mu_n$ are not cellular, but \cite[Corollary 8.3]{jacobson} shows that this is an isomorphism if $j \geq 2i+5$ (choose the approximation $\projective^r \times \Bmun{r}{n}$ with $r \geq i+2$ which has dimension $2r + 1 \geq 2i + 5$ to compute $H^i(B\Gm \times B\mu_n, I^j, \linebundle)$).
Similarly to \cref{remark:real-cycles-Bmun}, for higher powers of the fundamental ideal the product $\classub \euleria$ does not vanish but presents an additional free generator of the respective $I^j$-cohomology groups.

This differs from Chow theory, where the Künneth isomorphism holds for $B\Gm \times B\mu_n$ by \cite[Lemma 2.12]{totaro-bluebook}.
\end{remark}

\chapter{Chow-Witt Ring of \texorpdfstring{$B\mu_m \times B\mu_n$}{Bmu\_m x Bmu\_n}}

Throughout this chapter let $m$ and $n$ be positive integers and $k$ a perfect field of characteristic coprime to $2$, $m$ and $n$.

The Picard group of $B\mu_m \times B\mu_n$ is isomorphic to its first Chow group which is computed in \cite[Theorem 2.10, Lemma 2.12]{totaro-bluebook} to be isomorphic to $\Integer/m \times \Integer/n$.
For line bundles on (scheme approximations of) $B\Gm \times B\mu_n$ we adopt the usual notation
\[ \trivialbundle(s,t) = \pr_1^*\trivialbundle_{B\Gm} (s) \otimes \pr_2^* \trivialbundle_{B\mu_n}(t) \, . \]

We are going to use another instance of the localization sequence like in the previous sections, this time associated to the decomposition
\[ \projective^q \times \Bmun{r}{n} \xhookrightarrow{s_0} \Bmun{q}{m} \times \Bmun{r}{n} \xhookleftarrow{\iota} (\Bmun{q}{m} \times \Bmun{r}{n}) \setminus s_0(\projective^q \times \Bmun{r}{n}) \, . \]
Under the usual isomorphisms this reads as follows:
\begin{multline} 
\label{eq:locseq1} \ldots \to \chowwitt^{i-1}(B\Gm \times B\mu_n, s_0^*\linebundle \otimes \trivialbundle(m,0)) \xrightarrow{\eulercw(\trivialbundle(m,0))} \chowwitt^i(B\Gm \times B\mu_n, \linebundle) \\
\xrightarrow{\iota^*} \chowwitt^i(B\mu_m \times B\mu_n, \iota^*\linebundle) \to H^i(B\Gm \times B\mu_n, \KMW_{i-1}, s_0^*\linebundle \otimes \trivialbundle(m,0)) \to \ldots 
\end{multline}

Further we will use a trick for degree $0$ building on work of Hudson-Matszangosz-Wendt \cite{HMW}.
Note that the group $H^0(X, I^0, \trivialbundle_X)$ forms a subring of the total $I^j$-cohomology ring of $X$.
\begin{lemma}\label{lem:kuenneth-witt}
	The $\witt(k)$-algebra homomorphism 
	\[ \pr_1 \cdot \pr_2 \colon H^0(B\mu_m, I^0, \trivialbundle) \otimes_{\witt(k)} H^0(B\mu_n, I^0, \trivialbundle) \to H^0(B\mu_m \times B\mu_n, I^0, \trivialbundle) \]
	is an isomorphism.
\end{lemma}

\begin{proof}	
Recall from \cref{def:I-cohomology} that $H^i(X, \witt, \linebundle) \cong H^i(X, I^0, \linebundle)$, thus in degree $0$ results about Witt cohomology can be applied.
In the following all groups have trivial twist and we will omit the twist from notation.

Consider the following diagram of schemes
\[ \begin{tikzcd}
	\projective^q \times \Bmun{r}{n} \ar[r, "s_0 \times \id"] \ar[d, "\pr_1"] & \trivialbundle_{\projective^q}(m) \times \Bmun{r}{n} \ar[d, "\pr_1"] & (\trivialbundle_{\projective^q}(m)\setminus s_0) \times \Bmun{r}{n} \ar[d, "\pr_1"] \ar[l, "\iota \times \id"] \\
	\projective^q \ar[r, "s_0"] & \trivialbundle_{\projective^q}(m) & \trivialbundle_{\projective^q}(m) \setminus s_0 \ar[l, "\iota"]
\end{tikzcd}\]
and the $\witt(k)$-module homomorphism $\pr_2^* \colon H^0(B\mu_n, \witt) \to H^\tot(\trivialbundle_{\projective^q}, \witt)$.
The ladder lemma of \cite[Lemma 4.5]{HMW} then states that the long exact localization sequence of the bottom row tensored with $H^0(\Bmun{r}{n}, \witt)$ and that of the top row form a commutative diagram of $\witt(k)$-modules:
\[ \begin{tikzcd}
	\vdots \ar[d] & \vdots \ar[d] \\
	H^{-1}(\projective^q, \witt) \otimes_{\witt(k)} H^0(\Bmun{r}{n}, \witt) \ar[d, "(s_0)_* \otimes \id"] \ar[r, "\pr_1 \cdot \pr_2"] & H^{-1}(\projective^q \times B\mu_n, \witt) \ar[d, "(s_0 \times \id)_*"] \\
	H^0(\projective^q, \witt) \otimes_{\witt(k)} H^0(\Bmun{r}{n}, \witt) \ar[d, "\iota^* \otimes \id"] \ar[r, "\pr_1 \cdot \pr_2"] & H^0(\projective^q \times B\mu_n, \witt) \ar[d, "(\iota \times \id)^*"] \\
	H^0(B\mu_m, \witt) \otimes_{\witt(k)} H^0(\Bmun{r}{n}, \witt) \ar[d, "\partial \otimes \id"] \ar[r, "\pr_1 \cdot \pr_2"] & H^0(B\mu_m \times B\mu_n, \witt) \ar[d, "\partial^\prime"] \\
	H^0(\projective^q, \witt) \otimes_{\witt(k)} H^0(\Bmun{r}{n}, \witt) \ar[d, "(s_0)_* \otimes \id"]  \ar[r, "\pr_1 \cdot \pr_2"] & H^0(\projective^q \times B\mu_n, \witt) \ar[d, "(s_0 \times \id)_*"] \\
	H^1(\projective^q, \witt) \otimes_{\witt(k)} H^0(\Bmun{r}{n}, \witt) \ar[d] \ar[r, "\pr_1 \cdot \pr_2"] & H^1(\projective^q \times B\mu_n, \witt) \ar[d] \\
	\vdots &\vdots
\end{tikzcd} \]
By our previous computation in \cref{prop:Ij-of-Bmun} $H^0(\Bmun{r}{n}, \witt) \cong H^0(B\mu_n, \witt)$ is a free and therefore flat $\witt(k)$-module, hence the left-hand column is exact.
The first, second, fourth and fifth horizontal maps are isomorphisms by the Künneth formula for Witt cohomology \cite[Theorem 4.7]{HMW}; the scheme $\projective^q$ is cellular and $H^0(B\mu_n, \witt) \cong \witt(k)\langle 1, U \rangle$ (as computed in \cref{prop:Ij-of-Bmun}) as well as $H^{> 0}(B\mu_n, \witt) \cong 0$ (as computed in \cref{prop:KMW-of-Bmun}) are free $\witt(k)$-modules.
Thus by the five-lemma, the third horizontal map is an isomorphism as well.
\end{proof}

\section{Group Structure}

\begin{theorem}\label{chowwitt-groups-of-BmumxBmun}
	Let $i \in \Integer$ and $\linebundle$ a line bundle over $B\mu_m \times B\mu_n$.
	\begin{align*}
		\shortintertext{(i) If $m$ and $n$ are odd:}
		\chowwitt^i(B\mu_m \times B\mu_n, \linebundle) & \cong \chowwitt^i(B\Gm \times B\mu_n)/m \eulercwa \\
		& \cong \chowwitt^i(B\Gm \times B\Gm, \linebundle)/(m \eulercwa, n \eulercwb) \\ \shortintertext{(ii) If $m$ is odd and $n$ even:}
		\chowwitt^i(B\mu_m \times B\mu_n, \linebundle) & \cong \chowwitt^i(B\Gm \times B\mu_n)/m \eulercwa \\
		& \cong \begin{cases*}
			\GW(k) \oplus \witt(k)\langle \classub \rangle & $i=0$, $\linebundle$ trivial \\
			\chowwitt^i(B\Gm \times B\Gm, \linebundle)/(m \eulercwa, \frac{n}{2}\hyperbolicb\eulercwb) & else
		\end{cases*}
		\shortintertext{If $m$ is even and $n$ odd, exchanging $m$ and $n$ reduces to case (ii).}
		\shortintertext{(iii) If $m$ and $n$ are both even:}
		\chowwitt^i(B\mu_m \times B\mu_n, \linebundle) & \cong \begin{cases*}
			\GW(k) \oplus \witt(k) \langle \classua, \classub, \classua \classub \rangle & $i=0$, $\linebundle$ trivial \\
			\chowwitt^i(B\Gm \times B\Gm)/(\frac{m}{2}\hyperbolica \eulercwa, \frac{n}{2} \hyperbolicb \eulercwb) & else
		\end{cases*}
	\end{align*}
	
\end{theorem}

\begin{proof}
	Assume first that $m$ is odd.
	Then \cref{lem:KMWi-1-of-BGMxBmun} shows that the last term of \cref{eq:locseq1} vanishes except when $i=0$ and $s_0^*\linebundle \otimes \trivialbundle(m,0)$ is trivial in $\Pic(B\Gm \times B\mu_n)/2$
	which is the case for $\linebundle = \trivialbundle(1,0)$.
	This case however can be circumvented by using that in $\Pic(B\mu_m \times B\mu_n)/2$, $\trivialbundle(1,0) \sim \trivialbundle(0,0)$ and thus there is an isomorphism
	\[ \chowwitt^i(B\mu_m \times B\mu_n, \trivialbundle(1,0)) \cong \chowwitt^i(B\mu_m \times B\mu_n, \trivialbundle(0,0)) \, . \]
	Therefore 
	\[ \chowwitt^i(B\mu_m \times B\mu_n, \linebundle) \cong \chowwitt^i(B\Gm \times B\mu_n, \linebundle)/ \eulercw(\trivialbundle(m,0)) \cdot \chowwitt^{i-1}(B\Gm \times B\mu_n, \linebundle \otimes \trivialbundle(m,0)) \]
	for all line bundles $\linebundle$ and $i \in \Integer$.
	This solves cases (i) and (ii).
	
	(iii) In this case the last term of the localization sequence \ref{eq:locseq1} vanishes and there is again an isomorphism
	\[ \chowwitt^i(B\mu_m \times B\mu_n, \linebundle) \cong \chowwitt^i(B\Gm \times B\mu_n, \linebundle)/ \eulercw(\trivialbundle(m,0)) \cdot \chowwitt^{i-1}(B\Gm \times B\mu_n, \linebundle \otimes \trivialbundle(m,0)) \]
	except for $i=0$ and $s_0^*\linebundle \otimes \trivialbundle(m,0)$ trivial, i.e.\ $\linebundle$ trivial.

	For this remaining case, \cref{lem:kuenneth-witt} shows 
	\begin{align*} 
		H^0(B\mu_m \times B\mu_n, I^0, \trivialbundle(0,0)) &\cong H^0(B\mu_m, I^0, \trivialbundle(0,0)) \otimes_{\witt(k)} H^0(B\mu_n, I^0, \trivialbundle(0,0)) \\
		&= \witt(k) \langle 1, \classua \rangle \otimes_{\witt(k)} \witt(k) \langle 1, \classub \rangle 
	\end{align*}
	as $\witt(k)$-algebras.
	To compute the Chow-Witt group, apply Hornbostel-Wendt's fiber product formula \cref{prop:chowwitt-is-chow-x-I}:
	From \cite[Theorem 2.10, Lemma 2.12]{totaro-bluebook} deduce $\chow^0(B\mu_m \times B\mu_n) \cong \Integer \langle 1 \rangle$.
	The kernel of the boundary map 
	\[ \partial_{\trivialbundle(0,0)} \colon \chow^0(B\mu_m \times B\mu_n) \to H^1(B\mu_m \times B\mu_n, I^1, \trivialbundle(0,0)) \]
	is the whole domain as it follows from \cref{lem:bockstein-of-sw} that $\partial_{\trivialbundle(0,0)} = \bockstein_{\trivialbundle(0,0)} \circ (\operatorname{mod} 2)$ maps $1$ to $1 \cdot \euleri(\trivialbundle(0,0)) = 0$.
	The structure map from $\chow^0$ to $\chowmodtwo^0$ is the usual mod $2$ map.
	The structure map from $H^0(B\mu_m \times B\mu_n, I^0, \trivialbundle(0,0)) = \witt(k) \langle 1, \classua, \classub, \classua \classub \rangle$ sends $1$ to $1$, and $\classua$, $\classub$ and thus also $\classua \classub$ to zero as di~Lorenzo-Mantovani have already shown in their work on the Chow-Witt ring of $B\mu_n$ \cite[Remark 5.2.2]{diLorenzo-Mantovani}.
	Therefore the Chow-Witt group in degree $0$ evaluates to
	\begin{align*}
		\chowwitt^0(B\mu_m \times B\mu_n, \trivialbundle) & \cong H^0(B\mu_m \times B\mu_n, I^0, \trivialbundle) \times_{\chowmodtwo^0(B\mu_m \times B\mu_n)} \chow^0(B\mu_m \times B\mu_n) \\
		&= \witt(k) \langle 1, \classua, \classub, \classua \classub \rangle \times_{\Integer/2} \Integer \langle 1 \rangle \\
		&= \GW(k) \langle 1 \rangle \oplus \witt(k) \langle \classua, \classub, \classua \classub \rangle \, . \qedhere
	\end{align*}
\end{proof}

\section{Ring Structure}

\begin{theorem} \label{thm:chowwitt-of-BmumxBmun}
	Let $m$ and $n$ be positive natural numbers.
	\begin{align*}
		\shortintertext{(i) If $m$ and $n$ are odd:}
		\chowwitt^\tot(B\mu_m \times B\mu_n) \cong &\GW(k)[\eulercwa, \eulercwb]/\left( I(k) \cdot (\eulercwa, \eulercwb), m \eulercwa, n \eulercwb\right) \\
		\shortintertext{(ii) If $m$ is odd and $n$ even:}
		\chowwitt^\tot(B\mu_m \times B\mu_n) \cong& \GW(k)[\classub, \hyperbolicb, \eulercwa, \eulercwb]/( I(k) \cdot (\hyperbolicb, \eulercwa, \eulercwb), m\eulercwa, \frac{n}{2}\hyperbolicb\eulercwb, \\
		&\hyperbolicb^2 - 2h, \hyperbolic \classub, \hyperbolicb \classub, \classub^2 + 2\classub, \classub \eulercwa, \classub \eulercwb - n \eulercwb) \\
		\shortintertext{(iii) If $m$ and $n$ are even:}
		\chowwitt^\tot(B\mu_m \times B\mu_n) \cong& \GW(k)[\classua, \classub, \hyperbolica, \hyperbolicb, \hyperbolicc, \eulercwa, \eulercwb, \eulercwc]/( \frac{m}{2} \hyperbolica \eulercwa, \frac{n}{2} \hyperbolicb \eulercwb,\\
		&I(k) \cdot (\hyperbolica, \hyperbolicb, \hyperbolicc, \eulercwa, \eulercwb, \eulercwc), \mathcal{J}, \eulercwa^2 + \eulercwb^2 + \hyperbolicc \eulercwa \eulercwb - \eulercwc^2, \\
		& (\hyperbolic, \hyperbolica, \hyperbolicb, \hyperbolicc) \cdot (\classua, \classub), \classua^2 + 2\classua, \classub^2 + 2 \classub, \\
		&\classua \eulercwa - m \eulercwa, \classua \eulercwb - \frac{m}{2}\hyperbolicc\eulercwa, \classua \eulercwc - \frac{m}{2} \hyperbolicb \eulercwa, \\
		&\classub \eulercwa - \frac{n}{2}\hyperbolicc \eulercwb, \classub \eulercwb - n \eulercwb, \classub \eulercwc - \frac{m}{2} \hyperbolica \eulercwb)
	\end{align*}
	Here $\classua$ and $\classub$ correspond to the pullbacks along the projections $\pr_1$ and $\pr_2$ of the class $\classu$ defined in \cref{def:classu}.
	The symbols $\eulercwz_i$ map to $\eulercw(\linebundle^\vee) \in \chowwitt^1(B\mu_m \times B\mu_n, \linebundle)$ and $\hyperbolicz_i$ to
	\[ (0,2) \in \chowwitt^0(B\mu_m \times B\mu_n, \linebundle) \subseteq H^0(B\mu_m \times B\mu_n, I^0, \linebundle) \times_{\chowmodtwo^0(B\mu_m \times B\mu_n)} \chow^0(B\mu_m \times B\mu_n) \]
	for $\linebundle = \trivialbundle(1,0), \, \trivialbundle(0,1), \, \trivialbundle(1,1)$, respectively.
\end{theorem}

\begin{proof}
	(i) In \cref{chowwitt-groups-of-BmumxBmun} we have seen that all Chow-Witt groups of $B\mu_m \times B\mu_n$ are isomorphic to those of $B\Gm \times B\mu_n$ modulo $m \eulercwa$,
	thus the Chow-Witt ring is isomorphic to that of $B\Gm \times B\mu_n$ modulo $m \eulercwb$ and the quadratic periodicity isomorphism $\sqperiod$ for $\trivialbundle(1,0) \cong \trivialbundle(m+1,0) \cong \trivialbundle((m+1)/2,0)^{\otimes 2}$.
	
	The same argument as in the proof of \cref{thm:chowwitt-of-BGmxBmun} for $n$ odd shows $\sqperiod(\hyperbolica) = \hyperbolic$.
	This implies that the generator $\hyperbolica$ as well as the relation $I(k) \hyperbolica$ from $\chowwitt^\tot(B\Gm \times B\mu_n)$ become redundant in $\chowwitt^\tot(B\mu_m \times B\mu_n)$.
	
	For (ii) we can use an analogous argument, but now there are two quadratic periodicities to consider, namely $\sqperiod_{\trivialbundle((m+1)/2,0)}$ and $\sqperiod_{\trivialbundle((m+1)/2,1)}$.
	Again using the arguments from \cref{thm:chowwitt-of-BGmxBmun} yields
	\begin{align*}
		\sqperiod_{\trivialbundle((m+1)/2,0)}(\hyperbolica) &= \hyperbolic \\
		\sqperiod_{\trivialbundle((m+1)/2,1)}(\hyperbolicc) &= \hyperbolicb \\
		\sqperiod_{\trivialbundle((m+1)/2,1)}(\eulercwc) &= (m+1)/2 \hyperbolicc\eulercwa + \eulercwb \, .
	\end{align*}
	Thus the generators $\hyperbolica$, $\hyperbolicc$ and $\eulercwc$ as well as the relations $\hyperbolica \classub$, $\hyperbolicc \classub$, $\eulercwa^2 + \eulercwb^2 + \hyperbolicc \eulercwa \eulercwb - \eulercwc^2$ and everything in the relation ideal $\mathcal{J}$ from \cref{eq:relations-BGmxBGm-1,eq:relations-BGmxBGm-2,eq:relations-BGmxBGm-3} except for $\hyperbolicb^2 - 2 \hyperbolic$ become redundant.
	The relation $\classub \eulercwa - (n/2) \hyperbolicc \eulercwb$ becomes $\classub \eulercwa$.
	
	(iii) In this case there are no quadratic periodicity isomorphisms to be divided out.
	All relations from the Chow-Witt ring of $B\Gm \times B\mu_n$ are inherited. 
	Exchanging $m$ and $n$ further allows to inherit the product of $\classua$ with $\hyperbolica$, $\hyperbolicb$, $\hyperbolicc$, $\eulercwa$, $\eulercwb$ and $\eulercwc$ from $B\mu_m \times B\Gm$.
	The localization sequence \cref{eq:locseq1} forces the relation $\eulercw(\trivialbundle(m,0))$ which equals $(m/2)\hyperbolica \eulercwa$.
	The Künneth formula in degree $0$ from \cref{lem:kuenneth-witt} shows that $\classua \classub$ is $\witt(k)$-linearly independent of the generators $1$, $\classua$ and $\classub$ 
	hence there are no additional relations for this product.
\end{proof}

\section{\texorpdfstring{$I^j$}{Ij}-Cohomology of \texorpdfstring{$B\mu_m \times B\mu_n$}{Bmu\_m x Bmu\_n}}

\begin{corollary}
	\begin{align*}
		\shortintertext{If $m$ and $n$ are odd:}
		H^\tot(B\mu_m \times B\mu_n, I^*) & \cong W(k) \\
		\shortintertext{If $m$ is odd and $n$ is even:}
		H^\tot(B\mu_m \times B\mu_n, I^*) & \cong W(k)[\classub, \eulerib](I(k) \eulerib, \classub^2 + 2\classub, \classub \eulerib) \\
		\shortintertext{If $m$ and $n$ are even:}
		H^\tot(B\mu_m \times B\mu_n, I^*) &\cong \witt(k)[\classua, \classub, \euleria, \eulerib, \euleric]/(I(k) \cdot (\euleria, \eulerib, \euleric), \euleria^2 + \eulerib^2 - \euleric^2, \\
		& \classua^2 + 2\classua, \classub^2 + 2 \classub, \classua \euleria, \classua \eulerib, \classua \euleric, \classub \euleria, \classub \eulerib, \classub \euleric)
	\end{align*}
\end{corollary}

\begin{proof}
	Combine the computation of $\chowwitt^\tot(B\mu_m \times B\mu_n)$ from \cref{thm:chowwitt-of-BmumxBmun} and the fact that the sequence 
	\[ \chow^i(B\mu_m \times B\mu_n) \xrightarrow{\hyperbolic_\linebundle} \chowwitt^i(B\mu_m \times B\mu_n, \linebundle) \to H^i(B\mu_m \times B\mu_n, I^i, \linebundle) \to 0 \]
	is exact.
	
	If $m$ and $n$ are odd, the image of the hyperbolic map $\hyperbolic_{\trivialbundle(0,0)}$ is the submodule of $\chowwitt^\tot(B\mu_m \times B\mu_n)$ generated by $\hyperbolic$.
	Dividing this out identifies the relation $m \eulercwa$ with $\eulercwa$ and $n \eulercwb$ with $\eulercwb$.
	This leaves only the group $\chowwitt^0(B\mu_m \times B\mu_n, \trivialbundle)/h \cong \GW(k)/h \cong \witt(k)$.
	
	If $m$ is odd and $n$ is even, the images of the hyperbolic maps $\hyperbolic_{\trivialbundle(0,0)}$ and $\hyperbolic_{\trivialbundle(0,1)}$ are generated by $\hyperbolic$ and $\hyperbolicb$, respectively.
	Dividing out $\hyperbolic$ identifies $m \eulercwa$ with $\eulercwa$ (since $m$ is odd and $2 \eulercwa - \hyperbolic \eulercwa \in I(k) \eulercwa$) causing $\euleria$ to become trivial, and all relations containing $\hyperbolic$ or $\hyperbolicb$ vanish.
	
	If $m$ and $n$ are even, the images of the hyperbolic maps are generated by $\hyperbolic$, $\hyperbolica$, $\hyperbolicb$, $\hyperbolicc$, respectively.
	The relation $\classua \eulercwa - m \eulercwa$ becomes $\classua \euleria$, as $m \euleria$ is a multiple of $2 \euleria = \hyperbolic \euleria$ and thus vanishes.
	All relations containing $\hyperbolic$, $\hyperbolica$, $\hyperbolicb$ or $\hyperbolicc$ vanish.
	
	The statement on ring structure follows by adding up the computed cohomology groups and inheriting the ring structure from $\chowwitt^\tot(B\mu_m \times B\mu_n)$.
\end{proof}

\begin{remark}\label{remark:kuenneth-formula}
For $m$ and $n$ both odd there is again a Künneth isomorphism
\begin{align*} 
	\chowwitt^\tot(B\mu_m) \otimes \chowwitt^\tot(B\mu_n) &\cong \chowwitt^\tot(B\mu_m \times B\mu_n) \\
	H^\tot(B\mu_m, I^*) \otimes H^\tot(B\mu_n, I^*) &\cong H^\tot(B\mu_m \times B\mu_n, I^*)
	\end{align*}
for both the Chow-Witt and the $I^j$-cohomology ring.
If $m$ and $n$ are both even this map is injective but not surjective because the right-hand side contains the class $\eulercwc$ that does not come from one of the factors.
If $m$ is odd and $n$ even, the map is surjective but not injective because there is a new relation $\classub \eulercwa - (n/2) \hyperbolicb \eulercwb$ respectively $\classub \euleria$ on the right-hand side.

Like in \cref{remark:kuenneth-formula-1} the real cycle class map 
\[ H^i(B\mu_m \times B\mu_n, I^j, \linebundle) \to H_{\sing}^i((B\mu_m \times B\mu_n)(\Real); \Integer(\linebundle)) \]
is not an isomorphism in all bidegrees but only for $j \geq 2i+6$.
In these bidegrees the product $\classub \eulercwa$ does not vanish but presents an additional generator of the $I^j$-cohomology group, making the Künnneth map injective just like the one for singular cohomology.

In Chow theory, for comparison, the Künneth isomorphism holds by \cite[Lemma 2.12]{totaro-bluebook}: the group $\mu_m$ satisfies condition (ii) because $k$ contains all $e$-th roots of unity where $e$ is the exponent of $\mu_n$ (even if it does not contain all $n$-th roots of unity).

To summarize, of the product spaces considered in this work only $B\Gm \times B\mu_n$ with $n$ odd and $B\mu_m \times B\mu_n$ with both $m$ and $n$ odd satisfy a Künneth isomorphism for Chow-Witt rings.
For $B\mu_m \times B\mu_n$ with $m$ and $n$ odd the Künneth map is surjective but not injective, for $\projective^q \times \projective^r$, $B\Gm \times B\Gm$ and $B\mu_m \times B\mu_n$ with $m$ and $n$ even it is injective but not surjective, and for $B\Gm \times B\mu_n$ with $n$ even and $B\mu_m \times B\mu_n$ with $m$ odd and $n$ even it is neither injective nor surjective.
At this point it is unclear to the author if this behavior can be predicted in any way.

\end{remark}

\backmatter
\printbibliography[heading=bibintoc]
	
\end{document}